\documentclass[reqno,a4paper]{amsart}

\usepackage{amsmath}
\usepackage{amssymb}
\usepackage{amsthm}
\usepackage[dvipdfmx]{graphicx}
\usepackage{color}

\newtheorem{theorem}{Theorem}[section]
\newtheorem{proposition}[theorem]{Proposition}

\newtheorem{lemma}[theorem]{Lemma}
\theoremstyle{definition}
\newtheorem{definition}[theorem]{Definition}

\newtheorem{remark}[theorem]{Remark}
\newtheorem{notation}{Notation}
\newtheorem{example}[theorem]{Example}

\newcommand{\oline}[1]{\mathbin{\overline{#1}}}
\newcommand{\uline}[1]{\mathbin{\underline{#1}}}

\begin{document}

\title{Biquandle (co)homology and handlebody-links}

\author[A.~Ishii]{Atsushi Ishii}
\address{Institute of Mathematics, University of Tsukuba, Ibaraki 305-8571, Japan}
\email{aishii@math.tsukuba.ac.jp}

\author[M.~Iwakiri]{Masahide Iwakiri}
\address{Graduate School of Science and Engineering, Saga University, Saga, 840-8502, Japan}
\email{iwakiri@ms.saga-u.ac.jp}

\author[S.~Kamada]{Seiichi Kamada}
\address{Department of Mathematics, Osaka City University, Osaka 558-8585, Japan}
\email{skamada@sci.osaka-cu.ac.jp}

\author[J.~Kim]{Jieon Kim}
\address{Department of Mathematics, Pusan National University, Busan, 46241, Republic of Korea}
\email{jieonkim@pusan.ac.kr}

\author[S.~Matsuzaki]{Shosaku Matsuzaki}
\address{Faculty of Engineering, Takushoku University, Tokyo, 193-0985, Japan}
\email{smatsuza@ner.takushoku-u.ac.jp}

\author[K.~Oshiro]{Kanako Oshiro}
\address{Department of Information and Communication Sciences, Sophia University, Tokyo 102-8554, Japan}
\email{oshirok@sophia.ac.jp}

\thanks{Atsushi Ishii was supported by JSPS KAKENHI Grant Number 15K0486. Seiichi Kamada was supported by JSPS KAKENHI Grant Number  26287013. Jieon Kim was supported by JSPS KAKENHI Grant Number 15F15319 and a JSPS Postdoctral Fellowship for Foreign Researchers. Kanako Oshiro  was supported by JSPS KAKENHI Grant Number 16K17600.}

\date{}

\maketitle

\begin{abstract}
In this paper, we introduce the (co)homology group of a multiple conjugation biquandle.
It is the (co)homology group of the prismatic chain complex, which is related to the homology of foams introduced by J.~S.~Carter, modulo a certain subchain complex.
We construct invariants for $S^1$-oriented handlebody-links using $2$-cocycles.
When a multiple conjugation biquandle $X\times\mathbb{Z}_{\operatorname{type}X_Y}$ is obtained from a biquandle $X$ using $n$-parallel operations, we provide a $2$-cocycle (or $3$-cocycle) of the multiple conjugation biquandle $X\times\mathbb{Z}_{\operatorname{type}X_Y}$ from a $2$-cocycle (or $3$-cocycle) of the biquandle $X$ equipped with an $X$-set $Y$.
\end{abstract}

\section{Introduction}

The first and the second authors of this paper studied quandle cocycle invariants for spatial graphs and handlebody-links by using flows of spatial graphs~\cite{IshiiIwakiri12}, which is a generalization of quandle cocycle invariants for knots and links in \cite{CarterJelsovskyKamadaLangfordSaito03}.
The idea of quandle colorings of spatial graphs using flows was generalized to colorings using $G$-family of quandles in~\cite{IshiiIwakiriJangOshiro13} and colorings using multiple conjugation quandles in~\cite{Ishii15MCQ}.
The notion of a $G$-family of quandles is generalized to a $G$-family of biquandles and the notion of a partially multiplicative biquandle was introduced in~\cite{IshiiNelson16}.
In our previous paper~\cite{IshiiIwakiriKamadaKimMatsuzakiOshiroPMB}, we introduced a multiple conjugation biquandle as a generalization of a multiple conjugation quandle, and proved that it is the universal algebra for defining a semi-arc coloring invariant for handlebody-links.
We note that our multiple conjugation biquandle is different from a partially multiplicative biquandle of~\cite{IshiiNelson16}.
In the paper we also extended the notion of $n$-parallel biquandle operations for any integer $n$, not only for positive integers, and showed that any biquandle gives a multiple conjugation biquandle via $n$-parallel biquandle operations.

In this paper, we discuss homology theory for multiple conjugation biquandles.
The homology of a multiple conjugation biquandle is the homology of the prismatic chain complex, which is related to the homology of foams introduced by J.~S.~Carter, modulo a certain subchain complex.
Then we can use the homology theory to construct cocycle invariants of handlebody-links.

The paper is organized as follows.
In Section~\ref{sect:MCB}, we recall the definition of a multiple conjugation biquandle from~\cite{IshiiIwakiriKamadaKimMatsuzakiOshiroPMB} and how to obtain a multiple conjugation biquandle from a $G$-family of biquandles.
In Section~\ref{sect:X-set}, the notion of an action of a multiple conjugation biquandle on a set is introduced.
A set is called an $X$-set if an action of a multiple conjugation biquandle $X$ is equipped.
In Section~\ref{sect:coloring}, we recall colorings of an $S^1$-oriented handlebody-link by a multiple conjugation biquandle from~\cite{IshiiIwakiriKamadaKimMatsuzakiOshiroPMB}.
In Section~\ref{sect:chaincomplex}, we introduce a chain complex for a multiple conjugation biquandle, called the prismatic chain complex, and in Section~\ref{sect:degenerate}, the (co)homology group of a multiple conjugation biquandle is defined as the (co)homology group of the prismatic chain complex modulo a certain subchain complex.
In Section~\ref{sect:cocycle}, when a multiple conjugation biquandle
$X\times\mathbb{Z}_{\operatorname{type}X_Y}=\bigsqcup_{x\in X}\{x\}\times\mathbb{Z}_{\operatorname{type}X_Y}$ is obtained from a biquandle $X$ using $n$-parallel operations, we provide a $2$-cocycle (or $3$-cocycle) of the multiple conjugation biquandle $X\times\mathbb{Z}_{\operatorname{type}X_Y}$ from a $2$-cocycle (or $3$-cocycle) of the biquandle $X$ equipped with an $X$-set $Y$.
In Section~\ref{sect:cocycleinvariant}, the cocycle invariant for $S^1$-oriented handlebody-links is defined by using a $2$-cocycle of a multiple conjugation biquandle and we give an example of the cocycle invariant.

\section{Multiple conjugation biquandles (MCB)} \label{sect:MCB}

In this section, we review the definitions, properties and examples of multiple conjugation biquandles.
We have two equivalent definitions for a multiple conjugation biquandle.
The first one is useful to study coloring invariants, the second one is useful to check that a given algebra is a multiple conjugation biquandle, see \cite{IshiiIwakiriKamadaKimMatsuzakiOshiroPMB} for details.


\begin{definition}[\cite{FennRourkeSanderson95,KauffmanRadford03}]
A \textit{biquandle} is a non-empty set $X$ with binary operations $\uline{*},\oline{*}:X\times X\to X$ satisfying the following axioms.
\begin{itemize}
\item[(B1)]
For any $x\in X$, $x\uline{*}x=x\oline{*}x$.
\item[(B2)]
For any $a\in X$, the map $\uline{*}a:X\to X$ sending $x$ to $x\uline{*}a$ is bijective.
\item[]
For any $a\in X$, the map $\oline{*}a:X\to X$ sending $x$ to $x\oline{*}a$ is bijective.
\item[]
The map $S:X\times X\to X\times X$ defined by $S(x,y)=(y\oline{*}x,x\uline{*}y)$ is bijective.
\item[(B3)]
For any $x,y,z\in X$,
\begin{align*}
&(x\uline{*}y)\uline{*}(z\uline{*}y)=(x\uline{*}z)\uline{*}(y\oline{*}z), \\
&(x\uline{*}y)\oline{*}(z\uline{*}y)=(x\oline{*}z)\uline{*}(y\oline{*}z), \\
&(x\oline{*}y)\oline{*}(z\oline{*}y)=(x\oline{*}z)\oline{*}(y\uline{*}z).
\end{align*}
\end{itemize}
\end{definition}

Let $X$ be the disjoint union of groups $G_\lambda$ ($\lambda\in\Lambda$).
We denote by $G_a$ the group $G_\lambda$ to which $a\in X$ belongs.
We denote by $e_\lambda$ the identity of $G_\lambda$.

\begin{definition}[\cite{IshiiIwakiriKamadaKimMatsuzakiOshiroPMB}] \label{def:MGQ1}
A \textit{multiple conjugation biquandle} is a biquandle $(X,\uline{*},\oline{*})$ consisting of the disjoint union of groups $G_\lambda$ ($\lambda\in\Lambda$) satisfying the following axioms.
\begin{itemize}
\item
For any $a,x\in X$, $\uline{*}x:G_a\to G_{a\uline{*}x}$ and $\oline{*}x:G_a\to G_{a\oline{*}x}$ are group homomorphisms.
\item
For any $a,b\in G_\lambda$ and $x\in X$,
\begin{align*}
&x\uline{*}ab=(x\uline{*}a)\uline{*}(b\oline{*}a), \\
&x\oline{*}ab=(x\oline{*}a)\oline{*}(b\oline{*}a), \\
&a^{-1}b\oline{*}a=ba^{-1}\uline{*}a.
\end{align*}
\end{itemize}
\end{definition}

Here is an alternative definition of a multiple conjugation biquandle.

\begin{definition}[\cite{IshiiIwakiriKamadaKimMatsuzakiOshiroPMB}] \label{def:MGQ2}
A \textit{multiple conjugation biquandle} $X$ is the disjoint union of groups $G_\lambda$ ($\lambda\in\Lambda$) with binary operations $\uline{*},\oline{*}:X\times X\to X$ satisfying following axioms.
\begin{itemize}
\item
For any $x,y,z\in X$,
\begin{align*}
&(x\uline{*}y)\uline{*}(z\uline{*}y)=(x\uline{*}z)\uline{*}(y\oline{*}z), \\
&(x\uline{*}y)\oline{*}(z\uline{*}y)=(x\oline{*}z)\uline{*}(y\oline{*}z), \\
&(x\oline{*}y)\oline{*}(z\oline{*}y)=(x\oline{*}z)\oline{*}(y\uline{*}z).
\end{align*}
\item
For any $a,x\in X$, $\uline{*}x:G_a\to G_{a\uline{*}x}$ and $\oline{*}x:G_a\to G_{a\oline{*}x}$ are group homomorphisms.
\item
For any $a,b\in G_\lambda$ and $x\in X$,
\begin{align*}
&x\uline{*}ab=(x\uline{*}a)\uline{*}(b\oline{*}a),
\hspace{1em}x\uline{*}e_\lambda=x, \\
&x\oline{*}ab=(x\oline{*}a)\oline{*}(b\oline{*}a),
\hspace{1em}x\oline{*}e_\lambda=x, \\
&a^{-1}b\oline{*}a=ba^{-1}\uline{*}a.
\end{align*}
\end{itemize}
\end{definition}

Here we show that $G$-families of biquandles are useful to obtain multiple conjugation biquandles.

\begin{definition}[\cite{IshiiNelson16}]
Let $G$ be a group with identity element $e$.
A \textit{$G$-family of biquandles} is a non-empty set $X$ with two families of binary operations $\uline{*}^g,\oline{*}^g:X\times X\to X$ ($g\in G$) satisfying the following axioms.
\begin{itemize}
\item
For any $x,y,z\in X$ and $g,h\in G$,
\begin{align*}
&(x\uline{*}^gy)\uline{*}^h(z\oline{*}^gy)
=(x\uline{*}^hz)\uline{*}^{h^{-1}gh}(y\uline{*}^hz), \\
&(x\oline{*}^gy)\uline{*}^h(z\oline{*}^gy)
=(x\uline{*}^hz)\oline{*}^{h^{-1}gh}(y\uline{*}^hz), \\
&(x\oline{*}^gy)\oline{*}^h(z\oline{*}^gy)
=(x\oline{*}^hz)\oline{*}^{h^{-1}gh}(y\uline{*}^hz).
\end{align*}
\item
For any $x,y\in X$ and $g,h\in G$,
\begin{align*}
&x\uline{*}^{gh}y=(x\uline{*}^gy)\uline{*}^h(y\uline{*}^gy),
\hspace{1em}x\uline{*}^ey=x, \\
&x\oline{*}^{gh}y=(x\oline{*}^gy)\oline{*}^h(y\oline{*}^gy),
\hspace{1em}x\oline{*}^ey=x, \\
&x\uline{*}^gx=x\oline{*}^gx.
\end{align*}
\end{itemize}
\end{definition}

\begin{example}[\cite{IshiiIwakiriKamadaKimMatsuzakiOshiroPMB}] \label{prop:Z-family}
For a biquandle $(X,\uline{*},\oline{*})$, we define \textit{parallel biquandle operations} $\uline{*}^{[n]}$ and $\oline{*}^{[n]}$ by the following rule:
\begin{align*}
a\uline{*}^{[0]}b&=a, & a\uline{*}^{[1]}b&=a\uline{*}b, & a\uline{*}^{[i+j]}b&=(a\uline{*}^{[i]}b)\uline{*}^{[j]}(b\uline{*}^{[i]} b),\text{ and} \\
a\oline{*}^{[0]}b&=a, & a\oline{*}^{[1]}b&=a\oline{*}b, & a\oline{*}^{[i+j]}b&=(a\oline{*}^{[i]}b)\oline{*}^{[j]}(b\oline{*}^{[i]} b)
\end{align*}
for $i,j\in \mathbb{Z}$.
We note that biquandle operations $\uline{*}^{[n]}$ and $\oline{*}^{[n]}$ are well-defined, $b\uline{*}^{[-1]}b$ is the unique element satisfying $(b\uline{*}^{[-1]}b)\uline{*}^{[1]}(b\uline{*}^{[-1]}b)=b$ and $a\uline{*}^{[-1]}b=a\uline{*}^{-1}(b\uline{*}^{[-1]}b)$, where $\uline{*}^{-1}a$ is the inverse of $\uline{*}a$.
Then $(X,(\uline{*}^{[n]})_{n\in\mathbb{Z}},(\oline{*}^{[n]})_{n\in\mathbb{Z}})$ is a $\mathbb{Z}$-family of biquandles.

We define the \textit{type} of a biquandle $X$ by
\[ \operatorname{type}X=\min\{n>0\,|\,\text{$a\uline{*}^{[n]}b=a=a\oline{*}^{[n]}b$ ($\forall a, b\in X$)}\}. \]
If $X$ is finite, then $(X,(\uline{*}^{[n]})_{n\in\mathbb{Z}_{\operatorname{type}X}},(\oline{*}^{[n]})_{n\in\mathbb{Z}_{\operatorname{type}X}})$ is a $\mathbb{Z}_{\operatorname{type}X}$-family of biquandles \cite{IshiiNelson16}.
\end{example}

\begin{example}\label{exam:2.5}(\cite{IshiiIwakiriKamadaKimMatsuzakiOshiroPMB})
Let $G$ be a group with identity $e$, and let $\varphi:G\to Z(G)$ be a homomorphism, where $Z(G)$ is the center of $G$.
\begin{itemize}
\item[(1)]
Let $X$ be a group with a right action of $G$.
We denote by $x^g$ the result of $g$ acting on $x$.
We define binary operations $\uline{*}^g,\oline{*}^g:X\times X\to X$ by $x\uline{*}^gy=(xy^{-1})^gy^{\varphi(g)}$, $x\oline{*}^gy=x^{\varphi(g)}$.
Then $X$ is a $G$-family of biquandles, which we call a \textit{$G$-family of generalized Alexander biquandles}.

\item[(2)]
Let $R$ be a ring and $X$ a right $R[G]$-module, where $R[G]$ is the group ring of $G$ over $R$.
We define binary operations $\uline{*}^g,\oline{*}^g:X\times X\to X$ by $x\uline{*}^gy=xg+y(\varphi(g)-g)$, $x\oline{*}^gy=x\varphi(g)$.
Then $X$ is a $G$-family of biquandles, which we call a \textit{$G$-family of Alexander biquandles}.
\end{itemize}
\end{example}

The following proposition shows that we can construct a multiple conjugation biquandle from any $G$-family of biquandles, where we call it the \textit{associated multiple conjugation biquandle}.
In particular, Example~\ref{prop:Z-family} implies that we can construct a multiple conjugation biquandle from any biquandle.

\begin{proposition}[\cite{IshiiNelson16}] \label{prop:G-family2MCB}
Let $(X,(\uline{*}^g)_{g\in G},(\oline{*}^g)_{g\in G})$ be a $G$-family of biquandles.
Then $X\times G=\bigsqcup_{x\in X}\{x\}\times G$ is a multiple conjugation biquandle with the binary operations $\uline{*},\oline{*}:(X\times G)\times(X\times G)\to X\times G$ defined by
\begin{align*}
&(x,g)\uline{*}(y,h)=(x\uline{*}^hy,h^{-1}gh), &&(x,g)\oline{*}(y,h)=(x\oline{*}^hy,g).
\end{align*}
\end{proposition}

\section{X-sets of multiples conjugation biquandles} \label{sect:X-set}

\begin{definition}
For a multiple conjugation biquandle $X=\bigsqcup_{\lambda\in\Lambda}G_\lambda$, an \textit{$X$-set} is a non-empty set $Y$ with a map $*:Y\times X\to Y$ satisfying the following axioms.
\begin{itemize}
\item
For any $y\in Y$ and $a,b\in G_\lambda$, we have $y*e_\lambda=y$ and $y*(ab)=(y*a)*(b\oline{*}a)$, where $e_\lambda$ is the identity of $G_\lambda$.
\item
For any $y\in Y$ and $a,b\in X$, we have $(y*a)*(b\oline{*}a)=(y*b)*(a\uline{*}b)$.
\end{itemize}
\end{definition}

Any multiple conjugation biquandle $X$ itself is an $X$-set with the map $\uline{*}$ or $\oline{*}$.
Any singleton set $\{y_0\}$ is also an $X$-set with the map $*$ defined by $y_0*x=y_0$ for $x\in X$, which is called a \textit{trivial $X$-set}.
The index set $\Lambda$ is an $X$-set with the map $\uline{*}$ (resp.~$\oline{*}$) defined by $\lambda\uline{*}x=\mu$ (resp.~$\lambda\oline{*}x=\mu$) when $e_\lambda\uline{*}x=e_\mu$ (resp.~$e_\lambda\oline{*}x=e_\mu$) for $\lambda,\mu\in\Lambda$ and $x\in X$.

%
%

\section{Colorings for handlebody-links} \label{sect:coloring}

A \textit{handlebody-link} is the disjoint union of handlebodies embedded in the $3$-sphere $S^3$.
A \textit{handlebody-knot} is a one component handlebody-link.
In this paper, we assume that every component of a handlebody-link is of genus at least $1$.
An \textit{$S^1$-orientation} of a handlebody-link is a collection of $S^1$-orientations of all genus-$1$ components, that are solid tori, of the handlebody-link.
Here an $S^1$-orientation of a solid torus means an orientation of its core $S^1$.
Two $S^1$-oriented handlebody-links are \textit{equivalent} if there is an orientation-preserving self-homeomorphism of $S^3$ which sends one to the other preserving the $S^1$-orientation.

A \textit{Y-orientation} of a trivalent graph $G$, whose vertices are of valency $3$, is a direction of all edges of $G$ satisfying that every vertex of $G$ is both the initial vertex of a directed edge and the terminal vertex of a directed edge (See Figure \ref{fig:Y-orientations}).
In this paper, a trivalent graph may have a circle component, which has no vertices.

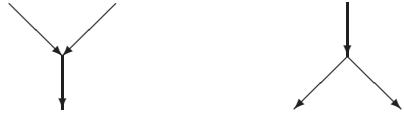
\begin{figure}
\center
\begin{picture}(50,40)
 \put(20,20){\vector(0,-1){20}}
 \put(0,40){\vector(1,-1){20}}
 \put(40,40){\vector(-1,-1){20}}
\end{picture}
\hspace{5em}
\begin{picture}(50,40)
 \put(20,40){\vector(0,-1){20}}
 \put(20,20){\vector(-1,-1){20}}
 \put(20,20){\vector(1,-1){20}}
\end{picture}
\caption{Y-orientations}
\label{fig:Y-orientations}
\end{figure}

A finite graph embedded in $S^3$ is called a \textit{spatial graph}.
For a Y-oriented spatial trivalent graph $K$ and an $S^1$-oriented handlebody-link $H$, we say that $K$ \textit{represents} $H$ if $H$ is a regular neighborhood of $K$ and the $S^1$-orientation of $H$ agrees with the Y-orientation.
Then any $S^1$-oriented handlebody-link can be represented by some Y-oriented spatial trivalent graph.
A \textit{diagram} of an $S^1$-oriented handlebody-link is a diagram of a Y-oriented spatial trivalent graph which represents the handlebody-link.
\textit{R1--R6 moves} are local moves depicted in Figure~\ref{fig:ReidemeisterMoves}.
\textit{Y-oriented R1--R6 moves} are R1--R6 moves between two diagrams with Y-orientations which are identical except in the disk where the move applied.
The following theorem plays a fundamental role in constructing $S^1$-oriented handlebody-link invariants.

\begin{theorem}[\cite{Ishii15Markov}] \label{thm:ReidemeisterMoves}
For a diagram $D_i$ of a Y-oriented spatial trivalent graph $K_i$ $(i=1,2)$, $K_1$ and $K_2$ represent an equivalent $S^1$-oriented handlebody-link if and only if $D_1$ and $D_2$ are related by a finite sequence of Y-oriented R1--R6 moves.
\end{theorem}

\begin{figure}
\mbox{}\hspace{1em}
\begin{minipage}{24pt}
\begin{picture}(24,48)
 \qbezier(0,0)(0,12)(6,24) 
 \qbezier(6,24)(12,36)(18,36) 
 \qbezier(18,12)(12,12)(7.5,21) 
 \qbezier(4.5,27)(0,36)(0,48) 
 \qbezier(18,12)(24,12)(24,24)
 \qbezier(18,36)(24,36)(24,24)
\end{picture}
\end{minipage}
~$\overset{\text{R1}}{\leftrightarrow}$~
\begin{minipage}{0pt}
\begin{picture}(0,48)
 \qbezier(0,0)(0,24)(0,48)
\end{picture}
\end{minipage}
~$\overset{\text{R1}}{\leftrightarrow}$~
\begin{minipage}{24pt}
\begin{picture}(24,48)
 \qbezier(0,48)(0,36)(6,24) 
 \qbezier(6,24)(12,12)(18,12) 
 \qbezier(18,36)(12,36)(7.5,27) 
 \qbezier(4.5,21)(0,12)(0,0) 
 \qbezier(18,12)(24,12)(24,24)
 \qbezier(18,36)(24,36)(24,24)
\end{picture}
\end{minipage}
\hfill
\begin{minipage}{24pt}
\begin{picture}(24,48)
 \qbezier(0,24)(0,30)(12,36) 
 \qbezier(12,36)(24,42)(24,48) 
 \qbezier(24,24)(24,30)(15,34.5) 
 \qbezier(9,37.5)(0,42)(0,48) 
 \qbezier(24,0)(24,6)(12,12) 
 \qbezier(12,12)(0,18)(0,24) 
 \qbezier(0,0)(0,6)(9,10.5) 
 \qbezier(15,13.5)(24,18)(24,24) 
\end{picture}
\end{minipage}
~$\overset{\text{R2}}{\leftrightarrow}$~
\begin{minipage}{24pt}
\begin{picture}(24,48)
 \qbezier(0,0)(0,24)(0,48)
 \qbezier(24,0)(24,24)(24,48)
\end{picture}
\end{minipage}
\hfill
\begin{minipage}{32pt}
\begin{picture}(32,48)
 \qbezier(0,32)(0,36)(8,40) 
 \qbezier(8,40)(16,44)(16,48) 
 \qbezier(16,32)(16,36)(10,39) 
 \qbezier(6,41)(0,44)(0,48) 
 \qbezier(32,32)(32,40)(32,48)
 \qbezier(16,16)(16,20)(24,24) 
 \qbezier(24,24)(32,28)(32,32) 
 \qbezier(32,16)(32,20)(26,23) 
 \qbezier(22,25)(16,28)(16,32) 
 \qbezier(0,16)(0,24)(0,32)
 \qbezier(0,0)(0,4)(8,8) 
 \qbezier(8,8)(16,12)(16,16) 
 \qbezier(16,0)(16,4)(10,7) 
 \qbezier(6,9)(0,12)(0,16) 
 \qbezier(32,0)(32,8)(32,16)
\end{picture}
\end{minipage}
~$\overset{\text{R3}}{\leftrightarrow}$~
\begin{minipage}{32pt}
\begin{picture}(32,48)
 \qbezier(16,32)(16,36)(24,40) 
 \qbezier(24,40)(32,44)(32,48) 
 \qbezier(32,32)(32,36)(26,39) 
 \qbezier(22,41)(16,44)(16,48) 
 \qbezier(0,32)(0,40)(0,48)
 \qbezier(0,16)(0,20)(8,24) 
 \qbezier(8,24)(16,28)(16,32) 
 \qbezier(16,16)(16,20)(10,23) 
 \qbezier(6,25)(0,28)(0,32) 
 \qbezier(32,16)(32,24)(32,32)
 \qbezier(16,0)(16,4)(24,8) 
 \qbezier(24,8)(32,12)(32,16) 
 \qbezier(32,0)(32,4)(26,7) 
 \qbezier(22,9)(16,12)(16,16) 
 \qbezier(0,0)(0,8)(0,16)
\end{picture}
\end{minipage}
\hspace{1em}\mbox{}\vspace{2em}\\
\begin{minipage}{24pt}
\begin{picture}(24,48)
 \qbezier(0,24)(0,30)(12,36) 
 \qbezier(12,36)(24,42)(24,48) 
 \qbezier(24,24)(24,30)(15,34.5) 
 \qbezier(9,37.5)(0,42)(0,48) 
 \qbezier(12,12)(12,6)(12,0) 
 \qbezier(12,12)(0,18)(0,24) 
 \qbezier(12,12)(24,18)(24,24) 
\end{picture}
\end{minipage}
~$\overset{\text{R4}}{\leftrightarrow}$~
\begin{minipage}{24pt}
\begin{picture}(24,48)
 \qbezier(0,48)(0,36)(0,24)
 \qbezier(24,48)(24,36)(24,24)
 \qbezier(12,12)(12,6)(12,0) 
 \qbezier(12,12)(0,18)(0,24) 
 \qbezier(12,12)(24,18)(24,24) 
\end{picture}
\end{minipage}
~$\overset{\text{R4}}{\leftrightarrow}$~
\begin{minipage}{24pt}
\begin{picture}(24,48)
 \qbezier(24,24)(24,30)(12,36) 
 \qbezier(12,36)(0,42)(0,48) 
 \qbezier(0,24)(0,30)(9,34.5) 
 \qbezier(15,37.5)(24,42)(24,48) 
 \qbezier(12,12)(12,6)(12,0) 
 \qbezier(12,12)(0,18)(0,24) 
 \qbezier(12,12)(24,18)(24,24) 
\end{picture}
\end{minipage}
\hfill
\begin{minipage}{24pt}
\begin{picture}(24,48)
 \qbezier(0,24)(0,32)(10,40)\qbezier(14,43)(19,46)(24,48)
 \qbezier(0,24)(5,24)(10,24)\qbezier(14,24)(19,24)(24,24)
 \qbezier(0,24)(0,16)(10,8)\qbezier(14,5)(19,2)(24,0)
 \qbezier(12,0)(12,24)(12,48)
\end{picture}
\end{minipage}
~$\overset{\text{R5}}{\leftrightarrow}$~
\begin{minipage}{24pt}
\begin{picture}(24,48)
 \qbezier(8,24)(8,36)(24,48)
 \qbezier(8,24)(16,24)(24,24)
 \qbezier(8,24)(8,12)(24,0)
 \qbezier(0,0)(0,24)(0,48)
\end{picture}
\end{minipage}
~$\overset{\text{R5}}{\leftrightarrow}$~
\begin{minipage}{24pt}
\begin{picture}(24,48)
 \qbezier(0,24)(0,36)(24,48)
 \qbezier(0,24)(12,24)(24,24)
 \qbezier(0,24)(0,12)(24,0)
 \qbezier(12,0)(12,2)(12,4)
 \qbezier(12,10)(12,16)(12,22)
 \qbezier(12,26)(12,32)(12,38)
 \qbezier(12,44)(12,46)(12,48)
\end{picture}
\end{minipage}
\hfill
\begin{minipage}{24pt}
\begin{picture}(24,48)
 \qbezier(0,48)(6,42)(12,36)
 \qbezier(24,48)(18,42)(12,36)
 \qbezier(12,12)(12,24)(12,36)
 \qbezier(0,0)(6,6)(12,12)
 \qbezier(24,0)(18,6)(12,12)
\end{picture}
\end{minipage}
~$\overset{\text{R6}}{\leftrightarrow}$~
\begin{minipage}{24pt}
\begin{picture}(24,48)
 \qbezier(0,48)(3,36)(6,24)
 \qbezier(24,48)(21,36)(18,24)
 \qbezier(6,24)(12,24)(18,24)
 \qbezier(0,0)(3,12)(6,24)
 \qbezier(24,0)(21,12)(18,24)
\end{picture}
\end{minipage}
\caption{The Reidemeister moves for handlebody-links.}
\label{fig:ReidemeisterMoves}
\end{figure}
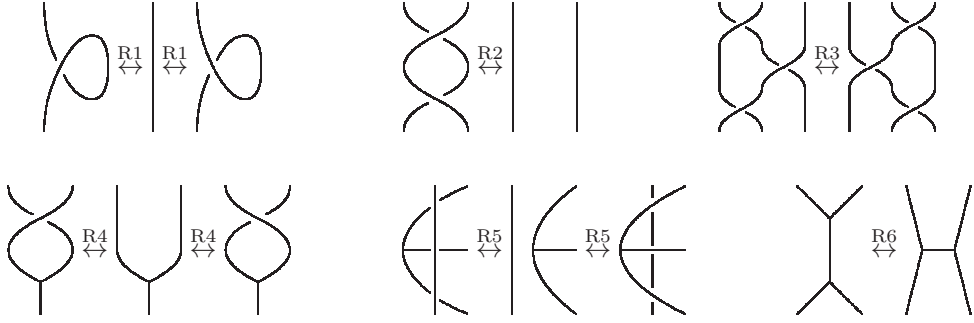

For a diagram $D$ of a Y-oriented spatial trivalent graph, we denote by $\mathcal{SA}(D)$ the set of semi-arcs of $D$, where a semi-arc is a piece of a curve each of whose endpoints is a crossing or a vertex.
We denote by $C(D)$ (resp.~$V(D)$) the set of crossings (resp.~vertices) of $D$.
We denote by $\mathcal{R}(D)$ the set of complementary regions of $D$.

\begin{definition}[\cite{IshiiIwakiriKamadaKimMatsuzakiOshiroPMB}] \label{def:coloring}
Let $X=\bigsqcup_{\lambda\in\Lambda}G_\lambda$ be a multiple conjugation biquandle.
Let $D$ be a diagram of an $S^1$-oriented handlebody-link $H$.
An \textit{$X$-coloring} of $D$ is a map $C:\mathcal{SA}(D)\to X$ satisfying that 
\vspace{1em}
\begin{center}
\begin{minipage}{65pt}
\begin{picture}(65,40)(-5,0)
 \put(40,40){\vector(-1,-1){40}}
 \put(0,40){\line(1,-1){18}}
 \put(22,18){\vector(1,-1){18}}
 \put(5,35){\makebox(0,0){\normalsize$\nearrow$}}
 \put(5,5){\makebox(0,0){\normalsize$\searrow$}}
 \put(35,35){\makebox(0,0){\normalsize$\searrow$}}
 \put(35,5){\makebox(0,0){\normalsize$\nearrow$}}
 \put(-3,40){\makebox(0,0)[r]{\normalsize$a$}}
 \put(-3,0){\makebox(0,0)[r]{\normalsize$b$}}
 \put(43,40){\makebox(0,0)[l]{\normalsize$b\oline{*}a$}}
 \put(43,0){\makebox(0,0)[l]{\normalsize$a\uline{*}b$}}
\end{picture}
\end{minipage}
\hspace{5em}
\begin{minipage}{65pt}
\begin{picture}(65,40)(-5,0)
 \put(0,40){\vector(1,-1){40}}
 \put(40,40){\line(-1,-1){18}}
 \put(18,18){\vector(-1,-1){18}}
 \put(5,35){\makebox(0,0){\normalsize$\nearrow$}}
 \put(5,5){\makebox(0,0){\normalsize$\searrow$}}
 \put(35,35){\makebox(0,0){\normalsize$\searrow$}}
 \put(35,5){\makebox(0,0){\normalsize$\nearrow$}}
 \put(-3,40){\makebox(0,0)[r]{\normalsize$a$}}
 \put(-3,0){\makebox(0,0)[r]{\normalsize$b$}}
 \put(43,40){\makebox(0,0)[l]{\normalsize$b\uline{*}a$}}
 \put(43,0){\makebox(0,0)[l]{\normalsize$a\oline{*}b$}}
\end{picture}
\end{minipage}
\end{center}
\vspace{1em}
holds at each crossing, and
\vspace{1em}
\begin{center}
\begin{minipage}{50pt}
\begin{picture}(50,40)(-5,0)
 \put(20,20){\vector(0,-1){20}}
 \put(0,40){\vector(1,-1){20}}
 \put(40,40){\vector(-1,-1){20}}
 \put(21,10){\makebox(0,0){\normalsize$\rightarrow$}}
 \put(5,35){\makebox(0,0){\normalsize$\nearrow$}}
 \put(35,35){\makebox(0,0){\normalsize$\searrow$}}
 \put(-3,40){\makebox(0,0)[r]{\normalsize$a$}}
 \put(43,40){\makebox(0,0)[l]{\normalsize$a^{-1}b\oline{*}a$}}
 \put(23,0){\makebox(0,0)[l]{\normalsize$b$}}
\end{picture}
\end{minipage}
\hspace{5em}
\begin{minipage}{50pt}
\begin{picture}(50,40)(-5,0)
 \put(20,40){\vector(0,-1){20}}
 \put(20,20){\vector(-1,-1){20}}
 \put(20,20){\vector(1,-1){20}}
 \put(21,30){\makebox(0,0){\normalsize$\rightarrow$}}
 \put(5,5){\makebox(0,0){\normalsize$\searrow$}}
 \put(35,5){\makebox(0,0){\normalsize$\nearrow$}}
 \put(23,40){\makebox(0,0)[l]{\normalsize$b$}}
 \put(-3,0){\makebox(0,0)[r]{\normalsize$a$}}
 \put(43,0){\makebox(0,0)[l]{\normalsize$a^{-1}b\oline{*}a$}}
\end{picture}
\end{minipage}
\end{center}
\vspace{1em}
holds at each vertex, where the normal orientation is obtained by rotating the usual orientation counterclockwise by $\pi/2$ on the diagram.
We denote by $\operatorname{Col}_X(D)$ the set of $X$-colorings of $D$.
\end{definition}
Next theorem shows that for any two diagrams $D$ and $D'$ of the same $S^1$-oriented handlebody-link, there exists a one-to-one correspondence between $\operatorname{Col}_X(D)$ and $\operatorname{Col}_X(D')$.

\begin{theorem}[\cite{IshiiIwakiriKamadaKimMatsuzakiOshiroPMB}] \label{thm:coloring1}
Let $X=\bigsqcup_{\lambda\in\Lambda}G_\lambda$ be a multiple conjugation biquandle.
Let $D$ be a diagram of an $S^1$-oriented handlebody-link $H$.
Let $D'$ be a diagram obtained by applying one of the Y-oriented R1--R6 moves to the diagram $D$ once.
For an $X$-coloring $C$ of $D$, there is a unique $X$-coloring $C'$ of $D'$ which coincides with $C$ except the place where the move is applied.
\end{theorem}
\noindent See Figure~\ref{fig:coloredReidemeisterMove} for $X$-colored Y-oriented R4--R6 moves, where all arcs are directed from top to bottom, except for the Reidemeister moves R4, and $a\triangle b=b^{-1}a\oline{*}b$.

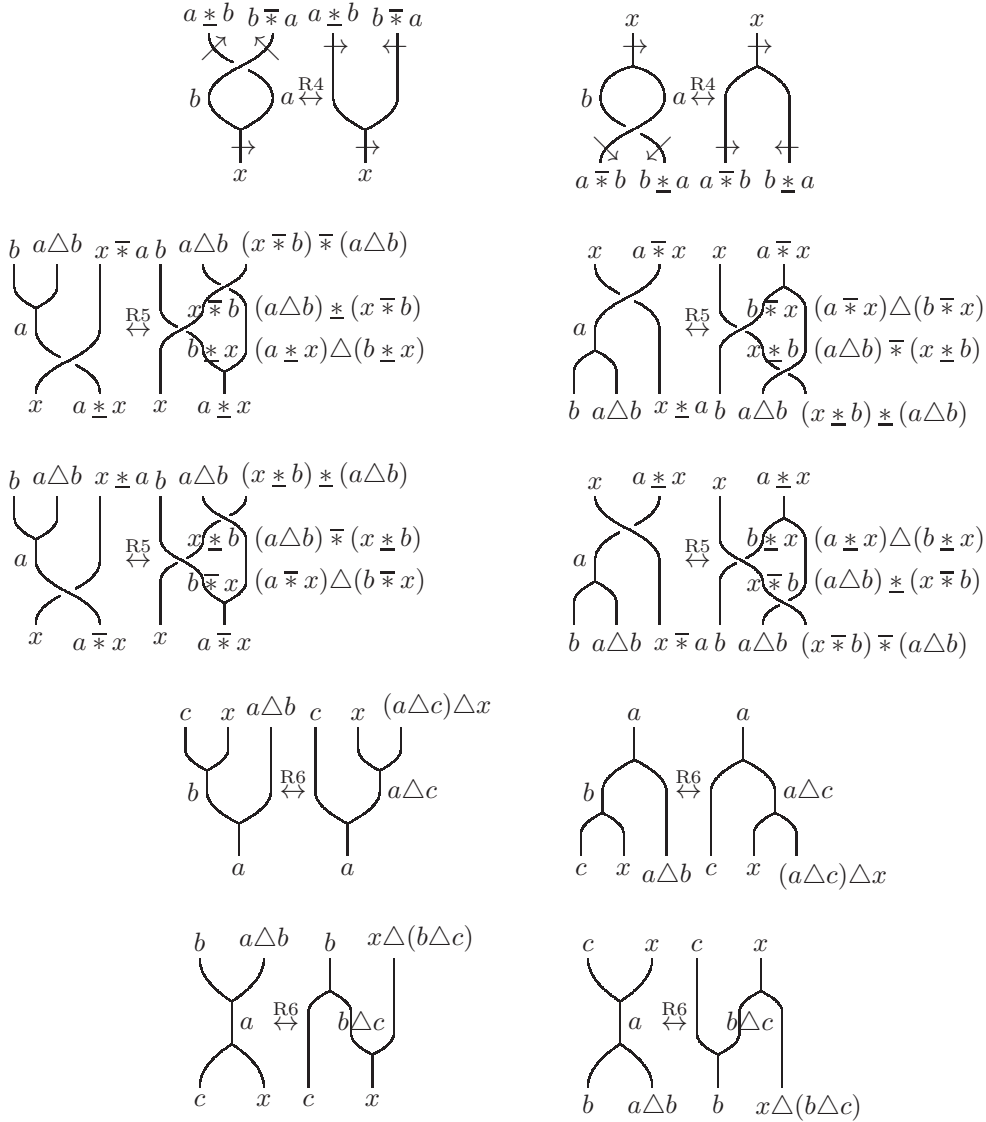
\begin{figure}
\mbox{}\hfill
\begin{minipage}{36pt}
\begin{picture}(24,72)(-6,-12)
 \qbezier(0,24)(0,30)(12,36) 
 \qbezier(12,36)(24,42)(24,48) 
 \qbezier(24,24)(24,30)(15,34.5) 
 \qbezier(9,37.5)(0,42)(0,48) 
 \qbezier(12,12)(12,6)(12,0) 
 \qbezier(12,12)(0,18)(0,24) 
 \qbezier(12,12)(24,18)(24,24) 
 \put(2,43){\makebox(0,0){\normalsize$\nearrow$}}
 \put(22,43){\makebox(0,0){\normalsize$\nwarrow$}}
 \put(13,5){\makebox(0,0){\normalsize$\rightarrow$}}
 \put(0,51){\makebox(0,0)[b]{\normalsize$a\uline{*}b$}}
 \put(24,51){\makebox(0,0)[b]{\normalsize$b\oline{*}a$}}
 \put(12,-3){\makebox(0,0)[t]{\normalsize$x$}}
 \put(-3,24){\makebox(0,0)[r]{\normalsize$b$}}
 \put(27,24){\makebox(0,0)[l]{\normalsize$a$}}
\end{picture}
\end{minipage}
~$\overset{\text{R4}}{\leftrightarrow}$~
\begin{minipage}{24pt}
\begin{picture}(24,72)(0,-12)
 \qbezier(0,48)(0,36)(0,24)
 \qbezier(24,48)(24,36)(24,24)
 \qbezier(12,12)(12,6)(12,0) 
 \qbezier(12,12)(0,18)(0,24) 
 \qbezier(12,12)(24,18)(24,24) 
 \put(1,43){\makebox(0,0){\normalsize$\rightarrow$}}
 \put(23,43){\makebox(0,0){\normalsize$\leftarrow$}}
 \put(13,5){\makebox(0,0){\normalsize$\rightarrow$}}
 \put(0,51){\makebox(0,0)[b]{\normalsize$a\uline{*}b$}}
 \put(24,51){\makebox(0,0)[b]{\normalsize$b\oline{*}a$}}
 \put(12,-3){\makebox(0,0)[t]{\normalsize$x$}}
\end{picture}
\end{minipage}
\hfill
\begin{minipage}{36pt}
\begin{picture}(24,72)(-6,-12)
 \qbezier(12,36)(12,42)(12,48) 
 \qbezier(12,36)(0,33)(0,24) 
 \qbezier(12,36)(24,33)(24,24) 
 \qbezier(0,0)(0,6)(12,12) 
 \qbezier(12,12)(24,18)(24,24) 
 \qbezier(24,0)(24,6)(15,10.5) 
 \qbezier(9,13.5)(0,18)(0,24) 
 \put(13,43){\makebox(0,0){\normalsize$\rightarrow$}}
 \put(2,5){\makebox(0,0){\normalsize$\searrow$}}
 \put(22,5){\makebox(0,0){\normalsize$\swarrow$}}
 \put(12,51){\makebox(0,0)[b]{\normalsize$x$}}
 \put(0,-3){\makebox(0,0)[t]{\normalsize$a\oline{*}b$}}
 \put(24,-3){\makebox(0,0)[t]{\normalsize$b\uline{*}a$}}
 \put(-3,24){\makebox(0,0)[r]{\normalsize$b$}}
 \put(27,24){\makebox(0,0)[l]{\normalsize$a$}}
\end{picture}
\end{minipage}
~$\overset{\text{R4}}{\leftrightarrow}$~
\begin{minipage}{24pt}
\begin{picture}(24,72)(0,-12)
 \qbezier(12,36)(12,42)(12,48) 
 \qbezier(12,36)(0,30)(0,24) 
 \qbezier(12,36)(24,30)(24,24) 
 \qbezier(0,24)(0,12)(0,0)
 \qbezier(24,24)(24,12)(24,0)
 \put(13,43){\makebox(0,0){\normalsize$\rightarrow$}}
 \put(1,5){\makebox(0,0){\normalsize$\rightarrow$}}
 \put(23,5){\makebox(0,0){\normalsize$\leftarrow$}}
 \put(12,51){\makebox(0,0)[b]{\normalsize$x$}}
 \put(0,-3){\makebox(0,0)[t]{\normalsize$a\oline{*}b$}}
 \put(24,-3){\makebox(0,0)[t]{\normalsize$b\uline{*}a$}}
\end{picture}
\end{minipage}
\hfill\mbox{}\vspace{5mm}

\begin{minipage}{38pt}
\begin{picture}(32,72)(0,-12)
 \qbezier(0,40)(0,44)(0,48)
 \qbezier(16,40)(16,44)(16,48)
 \qbezier(8,32)(8,28)(8,24) 
 \qbezier(8,32)(0,36)(0,40) 
 \qbezier(8,32)(16,36)(16,40) 
 \qbezier(32,24)(32,36)(32,48)
 \qbezier(8,0)(8,6)(20,12) 
 \qbezier(20,12)(32,18)(32,24) 
 \qbezier(32,0)(32,6)(23,10.5) 
 \qbezier(17,13.5)(8,18)(8,24) 
 \put(0,51){\makebox(0,0)[b]{\normalsize$b$}}
 \put(16,51){\makebox(0,0)[b]{\normalsize$a\triangle b$}}
 \put(30,51){\makebox(0,0)[bl]{\normalsize$x\oline{*}a$}}
 \put(8,-3){\makebox(0,0)[t]{\normalsize$x$}}
 \put(32,-3){\makebox(0,0)[t]{\normalsize$a\uline{*}x$}}
 \put(5,24){\makebox(0,0)[r]{\normalsize$a$}}
\end{picture}
\end{minipage}
~$\overset{\text{R5}}{\leftrightarrow}$~
\begin{minipage}{96pt}
\begin{picture}(32,72)(0,-12)
 \qbezier(16,32)(16,36)(24,40) 
 \qbezier(24,40)(32,44)(32,48) 
 \qbezier(32,32)(32,36)(26,39) 
 \qbezier(22,41)(16,44)(16,48) 
 \qbezier(0,32)(0,40)(0,48)
 \qbezier(0,16)(0,20)(8,24) 
 \qbezier(8,24)(16,28)(16,32) 
 \qbezier(16,16)(16,20)(10,23) 
 \qbezier(6,25)(0,28)(0,32) 
 \qbezier(32,16)(32,24)(32,32)
 \qbezier(24,8)(24,4)(24,0) 
 \qbezier(24,8)(16,12)(16,16) 
 \qbezier(24,8)(32,12)(32,16) 
 \qbezier(0,0)(0,8)(0,16)
 \put(0,51){\makebox(0,0)[b]{\normalsize$b$}}
 \put(16,51){\makebox(0,0)[b]{\normalsize$a\triangle b$}}
 \put(30,51){\makebox(0,0)[bl]{\normalsize$(x\oline{*}b)\oline{*}(a\triangle b)$}}
 \put(0,-3){\makebox(0,0)[t]{\normalsize$x$}}
 \put(24,-3){\makebox(0,0)[t]{\normalsize$a\uline{*}x$}}
 \put(10,32){\makebox(0,0)[l]{\normalsize$x\oline{*}b$}}
 \put(10,16){\makebox(0,0)[l]{\normalsize$b\uline{*}x$}}
 \put(35,32){\makebox(0,0)[l]{\normalsize$(a\triangle b)\uline{*}(x\oline{*}b)$}}
 \put(35,16){\makebox(0,0)[l]{\normalsize$(a\uline{*}x)\triangle(b\uline{*}x)$}}
\end{picture}
\end{minipage}
\hfill
\begin{minipage}{38pt}
\begin{picture}(32,72)(0,-12)
 \qbezier(8,24)(8,30)(20,36) 
 \qbezier(20,36)(32,42)(32,48) 
 \qbezier(32,24)(32,30)(23,34.5) 
 \qbezier(17,37.5)(8,42)(8,48) 
 \qbezier(8,16)(8,20)(8,24) 
 \qbezier(8,16)(0,12)(0,8) 
 \qbezier(8,16)(16,12)(16,8) 
 \qbezier(0,0)(0,4)(0,8)
 \qbezier(16,0)(16,4)(16,8)
 \qbezier(32,0)(32,12)(32,24)
 \put(8,51){\makebox(0,0)[b]{\normalsize$x$}}
 \put(32,51){\makebox(0,0)[b]{\normalsize$a\oline{*}x$}}
 \put(0,-3){\makebox(0,0)[t]{\normalsize$b$}}
 \put(16,-3){\makebox(0,0)[t]{\normalsize$a\triangle b$}}
 \put(30,-3){\makebox(0,0)[tl]{\normalsize$x\uline{*}a$}}
 \put(5,24){\makebox(0,0)[r]{\normalsize$a$}}
\end{picture}
\end{minipage}
~$\overset{\text{R5}}{\leftrightarrow}$~
\begin{minipage}{96pt}
\begin{picture}(32,72)(0,-12)
 \qbezier(24,40)(24,44)(24,48) 
 \qbezier(24,40)(16,36)(16,32) 
 \qbezier(24,40)(32,36)(32,32) 
 \qbezier(0,32)(0,40)(0,48)
 \qbezier(0,16)(0,20)(8,24) 
 \qbezier(8,24)(16,28)(16,32) 
 \qbezier(16,16)(16,20)(10,23) 
 \qbezier(6,25)(0,28)(0,32) 
 \qbezier(32,16)(32,24)(32,32)
 \qbezier(16,0)(16,4)(24,8) 
 \qbezier(24,8)(32,12)(32,16) 
 \qbezier(32,0)(32,4)(26,7) 
 \qbezier(22,9)(16,12)(16,16) 
 \qbezier(0,0)(0,8)(0,16)
 \put(0,51){\makebox(0,0)[b]{\normalsize$x$}}
 \put(24,51){\makebox(0,0)[b]{\normalsize$a\oline{*}x$}}
 \put(0,-3){\makebox(0,0)[t]{\normalsize$b$}}
 \put(16,-3){\makebox(0,0)[t]{\normalsize$a\triangle b$}}
 \put(30,-3){\makebox(0,0)[tl]{\normalsize$(x\uline{*}b)\uline{*}(a\triangle b)$}}
 \put(10,32){\makebox(0,0)[l]{\normalsize$b\oline{*}x$}}
 \put(10,16){\makebox(0,0)[l]{\normalsize$x\uline{*}b$}}
 \put(35,32){\makebox(0,0)[l]{\normalsize$(a\oline{*}x)\triangle(b\oline{*}x)$}}
 \put(35,16){\makebox(0,0)[l]{\normalsize$(a\triangle b)\oline{*}(x\uline{*}b)$}}
\end{picture}
\end{minipage}\vspace{5mm}

\begin{minipage}{38pt}
\begin{picture}(32,72)(0,-12)
 \qbezier(0,40)(0,44)(0,48)
 \qbezier(16,40)(16,44)(16,48)
 \qbezier(8,32)(8,28)(8,24) 
 \qbezier(8,32)(0,36)(0,40) 
 \qbezier(8,32)(16,36)(16,40) 
 \qbezier(32,24)(32,36)(32,48)
 \qbezier(32,0)(32,6)(20,12) 
 \qbezier(20,12)(8,18)(8,24) 
 \qbezier(8,0)(8,6)(17,10.5) 
 \qbezier(23,13.5)(32,18)(32,24) 
 \put(0,51){\makebox(0,0)[b]{\normalsize$b$}}
 \put(16,51){\makebox(0,0)[b]{\normalsize$a\triangle b$}}
 \put(30,51){\makebox(0,0)[bl]{\normalsize$x\uline{*}a$}}
 \put(8,-3){\makebox(0,0)[t]{\normalsize$x$}}
 \put(32,-3){\makebox(0,0)[t]{\normalsize$a\oline{*}x$}}
 \put(5,24){\makebox(0,0)[r]{\normalsize$a$}}
\end{picture}
\end{minipage}
~$\overset{\text{R5}}{\leftrightarrow}$~
\begin{minipage}{96pt}
\begin{picture}(32,72)(0,-12)
 \qbezier(32,32)(32,36)(24,40) 
 \qbezier(24,40)(16,44)(16,48) 
 \qbezier(16,32)(16,36)(22,39) 
 \qbezier(26,41)(32,44)(32,48) 
 \qbezier(0,32)(0,40)(0,48)
 \qbezier(16,16)(16,20)(8,24) 
 \qbezier(8,24)(0,28)(0,32) 
 \qbezier(0,16)(0,20)(6,23) 
 \qbezier(10,25)(16,28)(16,32) 
 \qbezier(32,16)(32,24)(32,32)
 \qbezier(24,8)(24,4)(24,0) 
 \qbezier(24,8)(16,12)(16,16) 
 \qbezier(24,8)(32,12)(32,16) 
 \qbezier(0,0)(0,8)(0,16)
 \put(0,51){\makebox(0,0)[b]{\normalsize$b$}}
 \put(16,51){\makebox(0,0)[b]{\normalsize$a\triangle b$}}
 \put(30,51){\makebox(0,0)[bl]{\normalsize$(x\uline{*}b)\uline{*}(a\triangle b)$}}
 \put(0,-3){\makebox(0,0)[t]{\normalsize$x$}}
 \put(24,-3){\makebox(0,0)[t]{\normalsize$a\oline{*}x$}}
 \put(10,32){\makebox(0,0)[l]{\normalsize$x\uline{*}b$}}
 \put(10,16){\makebox(0,0)[l]{\normalsize$b\oline{*}x$}}
 \put(35,32){\makebox(0,0)[l]{\normalsize$(a\triangle b)\oline{*}(x\uline{*}b)$}}
 \put(35,16){\makebox(0,0)[l]{\normalsize$(a\oline{*}x)\triangle(b\oline{*}x)$}}
\end{picture}
\end{minipage}
\hfill
\begin{minipage}{38pt}
\begin{picture}(32,72)(0,-12)
 \qbezier(32,24)(32,30)(20,36) 
 \qbezier(20,36)(8,42)(8,48) 
 \qbezier(8,24)(8,30)(17,34.5) 
 \qbezier(23,37.5)(32,42)(32,48) 
 \qbezier(8,16)(8,20)(8,24) 
 \qbezier(8,16)(0,12)(0,8) 
 \qbezier(8,16)(16,12)(16,8) 
 \qbezier(0,0)(0,4)(0,8)
 \qbezier(16,0)(16,4)(16,8)
 \qbezier(32,0)(32,12)(32,24)
 \put(8,51){\makebox(0,0)[b]{\normalsize$x$}}
 \put(32,51){\makebox(0,0)[b]{\normalsize$a\uline{*}x$}}
 \put(0,-3){\makebox(0,0)[t]{\normalsize$b$}}
 \put(16,-3){\makebox(0,0)[t]{\normalsize$a\triangle b$}}
 \put(30,-3){\makebox(0,0)[tl]{\normalsize$x\oline{*}a$}}
 \put(5,24){\makebox(0,0)[r]{\normalsize$a$}}
\end{picture}
\end{minipage}
~$\overset{\text{R5}}{\leftrightarrow}$~
\begin{minipage}{96pt}
\begin{picture}(32,72)(0,-12)
 \qbezier(24,40)(24,44)(24,48) 
 \qbezier(24,40)(16,36)(16,32) 
 \qbezier(24,40)(32,36)(32,32) 
 \qbezier(0,32)(0,40)(0,48)
 \qbezier(16,16)(16,20)(8,24) 
 \qbezier(8,24)(0,28)(0,32) 
 \qbezier(0,16)(0,20)(6,23) 
 \qbezier(10,25)(16,28)(16,32) 
 \qbezier(32,16)(32,24)(32,32)
 \qbezier(32,0)(32,4)(24,8) 
 \qbezier(24,8)(16,12)(16,16) 
 \qbezier(16,0)(16,4)(22,7) 
 \qbezier(26,9)(32,12)(32,16) 
 \qbezier(0,0)(0,8)(0,16)
 \put(0,51){\makebox(0,0)[b]{\normalsize$x$}}
 \put(24,51){\makebox(0,0)[b]{\normalsize$a\uline{*}x$}}
 \put(0,-3){\makebox(0,0)[t]{\normalsize$b$}}
 \put(16,-3){\makebox(0,0)[t]{\normalsize$a\triangle b$}}
 \put(30,-3){\makebox(0,0)[tl]{\normalsize$(x\oline{*}b)\oline{*}(a\triangle b)$}}
 \put(10,32){\makebox(0,0)[l]{\normalsize$b\uline{*}x$}}
 \put(10,16){\makebox(0,0)[l]{\normalsize$x\oline{*}b$}}
 \put(35,32){\makebox(0,0)[l]{\normalsize$(a\uline{*}x)\triangle(b\uline{*}x)$}}
 \put(35,16){\makebox(0,0)[l]{\normalsize$(a\triangle b)\uline{*}(x\oline{*}b)$}}
\end{picture}
\end{minipage}\vspace{5mm}

\mbox{}\hfill
\begin{minipage}{32pt}
\begin{picture}(32,72)(0,-12)
 \qbezier(0,40)(0,44)(0,48)
 \qbezier(16,40)(16,44)(16,48)
 \qbezier(8,32)(8,28)(8,24) 
 \qbezier(8,32)(0,36)(0,40) 
 \qbezier(8,32)(16,36)(16,40) 
 \qbezier(32,24)(32,36)(32,48)
 \qbezier(20,12)(20,6)(20,0) 
 \qbezier(20,12)(8,18)(8,24) 
 \qbezier(20,12)(32,18)(32,24) 
 \put(0,51){\makebox(0,0)[b]{\normalsize$c$}}
 \put(16,51){\makebox(0,0)[b]{\normalsize$x$}}
 \put(32,51){\makebox(0,0)[b]{\normalsize$a\triangle b$}}
 \put(20,-3){\makebox(0,0)[t]{\normalsize$a$}}
 \put(5,24){\makebox(0,0)[r]{\normalsize$b$}}
\end{picture}
\end{minipage}
~$\overset{\text{R6}}{\leftrightarrow}$~
\begin{minipage}{32pt}
\begin{picture}(32,72)(0,-12)
 \qbezier(16,40)(16,44)(16,48)
 \qbezier(32,40)(32,44)(32,48)
 \qbezier(24,32)(24,28)(24,24) 
 \qbezier(24,32)(16,36)(16,40) 
 \qbezier(24,32)(32,36)(32,40) 
 \qbezier(0,24)(0,36)(0,48)
 \qbezier(12,12)(12,6)(12,0) 
 \qbezier(12,12)(0,18)(0,24) 
 \qbezier(12,12)(24,18)(24,24) 
 \put(0,51){\makebox(0,0)[b]{\normalsize$c$}}
 \put(16,51){\makebox(0,0)[b]{\normalsize$x$}}
 \put(25,51){\makebox(0,0)[bl]{\normalsize$(a\triangle c)\triangle x$}}
 \put(12,-3){\makebox(0,0)[t]{\normalsize$a$}}
 \put(27,24){\makebox(0,0)[l]{\normalsize$a\triangle c$}}
\end{picture}
\end{minipage}
\hfill
\begin{minipage}{32pt}
\begin{picture}(32,72)(0,-12)
 \qbezier(20,36)(20,42)(20,48) 
 \qbezier(20,36)(8,30)(8,24) 
 \qbezier(20,36)(32,30)(32,24) 
 \qbezier(32,0)(32,12)(32,24)
 \qbezier(8,16)(8,20)(8,24) 
 \qbezier(8,16)(0,12)(0,8) 
 \qbezier(8,16)(16,12)(16,8) 
 \qbezier(0,0)(0,4)(0,8)
 \qbezier(16,0)(16,4)(16,8)
 \put(20,51){\makebox(0,0)[b]{\normalsize$a$}}
 \put(0,-3){\makebox(0,0)[t]{\normalsize$c$}}
 \put(16,-3){\makebox(0,0)[t]{\normalsize$x$}}
 \put(32,-3){\makebox(0,0)[t]{\normalsize$a\triangle b$}}
 \put(5,24){\makebox(0,0)[r]{\normalsize$b$}}
\end{picture}
\end{minipage}
~$\overset{\text{R6}}{\leftrightarrow}$~
\begin{minipage}{32pt}
\begin{picture}(32,72)(0,-12)
 \qbezier(12,36)(12,42)(12,48) 
 \qbezier(12,36)(0,30)(0,24) 
 \qbezier(12,36)(24,30)(24,24) 
 \qbezier(0,0)(0,12)(0,24)
 \qbezier(24,16)(24,20)(24,24) 
 \qbezier(24,16)(16,12)(16,8) 
 \qbezier(24,16)(32,12)(32,8) 
 \qbezier(16,0)(16,4)(16,8)
 \qbezier(32,0)(32,4)(32,8)
 \put(12,51){\makebox(0,0)[b]{\normalsize$a$}}
 \put(0,-3){\makebox(0,0)[t]{\normalsize$c$}}
 \put(16,-3){\makebox(0,0)[t]{\normalsize$x$}}
 \put(25,-3){\makebox(0,0)[tl]{\normalsize$(a\triangle c)\triangle x$}}
 \put(27,24){\makebox(0,0)[l]{\normalsize$a\triangle c$}}
\end{picture}
\end{minipage}
\hfill\mbox{}\vspace{5mm}

\mbox{}\hfill
\begin{minipage}{24pt}
\begin{picture}(24,72)(0,-12)
 \qbezier(12,32)(12,28)(12,24) 
 \qbezier(12,32)(0,40)(0,48) 
 \qbezier(12,32)(24,40)(24,48) 
 \qbezier(12,16)(12,20)(12,24) 
 \qbezier(12,16)(0,8)(0,0) 
 \qbezier(12,16)(24,8)(24,0) 
 \put(0,51){\makebox(0,0)[b]{\normalsize$b$}}
 \put(24,51){\makebox(0,0)[b]{\normalsize$a\triangle b$}}
 \put(0,-3){\makebox(0,0)[t]{\normalsize$c$}}
 \put(24,-3){\makebox(0,0)[t]{\normalsize$x$}}
 \put(15,24){\makebox(0,0)[l]{\normalsize$a$}}
\end{picture}
\end{minipage}
~$\overset{\text{R6}}{\leftrightarrow}$~
\begin{minipage}{32pt}
\begin{picture}(32,72)(0,-12)
 \qbezier(8,36)(8,42)(8,48) 
 \qbezier(8,36)(0,32)(0,28) 
 \qbezier(8,36)(16,32)(16,28) 
 \qbezier(0,0)(0,14)(0,28)
 \qbezier(16,20)(16,24)(16,28)
 \qbezier(32,20)(32,34)(32,48)
 \qbezier(24,12)(24,6)(24,0) 
 \qbezier(24,12)(16,16)(16,20) 
 \qbezier(24,12)(32,16)(32,20) 
 \put(8,51){\makebox(0,0)[b]{\normalsize$b$}}
 \put(22,51){\makebox(0,0)[bl]{\normalsize$x\triangle(b\triangle c)$}}
 \put(0,-3){\makebox(0,0)[t]{\normalsize$c$}}
 \put(24,-3){\makebox(0,0)[t]{\normalsize$x$}}
 \put(11,24){\makebox(0,0)[l]{\normalsize$b\triangle c$}}
\end{picture}
\end{minipage}
\hfill
\begin{minipage}{24pt}
\begin{picture}(24,72)(0,-12)
 \qbezier(12,32)(12,28)(12,24) 
 \qbezier(12,32)(0,40)(0,48) 
 \qbezier(12,32)(24,40)(24,48) 
 \qbezier(12,16)(12,20)(12,24) 
 \qbezier(12,16)(0,8)(0,0) 
 \qbezier(12,16)(24,8)(24,0) 
 \put(0,51){\makebox(0,0)[b]{\normalsize$c$}}
 \put(24,51){\makebox(0,0)[b]{\normalsize$x$}}
 \put(0,-3){\makebox(0,0)[t]{\normalsize$b$}}
 \put(24,-3){\makebox(0,0)[t]{\normalsize$a\triangle b$}}
 \put(15,24){\makebox(0,0)[l]{\normalsize$a$}}
\end{picture}
\end{minipage}
~$\overset{\text{R6}}{\leftrightarrow}$~
\begin{minipage}{32pt}
\begin{picture}(32,72)(0,-12)
 \qbezier(24,36)(24,42)(24,48) 
 \qbezier(24,36)(16,32)(16,28) 
 \qbezier(24,36)(32,32)(32,28) 
 \qbezier(0,20)(0,34)(0,48)
 \qbezier(16,20)(16,24)(16,28)
 \qbezier(32,0)(32,14)(32,28)
 \qbezier(8,12)(8,6)(8,0) 
 \qbezier(8,12)(0,16)(0,20) 
 \qbezier(8,12)(16,16)(16,20) 
 \put(0,51){\makebox(0,0)[b]{\normalsize$c$}}
 \put(24,51){\makebox(0,0)[b]{\normalsize$x$}}
 \put(8,-3){\makebox(0,0)[t]{\normalsize$b$}}
 \put(22,-3){\makebox(0,0)[tl]{\normalsize$x\triangle(b\triangle c)$}}
 \put(11,24){\makebox(0,0)[l]{\normalsize$b\triangle c$}}
\end{picture}
\end{minipage}
\hfill\mbox{}\vspace{5mm}
\caption{Colored Reidemeister moves.}
\label{fig:coloredReidemeisterMove}
\end{figure}

\begin{definition}(cf. \cite{NelsonPelland13})
Let $X=\bigsqcup_{\lambda\in\Lambda}G_\lambda$ be a multiple conjugation biquandle and $Y$ an $X$-set.
Let $D$ be a diagram of an $S^1$-oriented handlebody-link $H$.
An \textit{$X_Y$-coloring} of $D$ is a map $C:\mathcal{R}(D)\cup\mathcal{SA}(D)\to Y\cup X$ satisfying that the restriction of $C$ on $\mathcal{SA}(D)$ is an $X$-coloring of $D$, $C(\mathcal{R}(D)) \subset Y$ and 
\begin{center}
\begin{picture}(66,40)(-5,0)
 \put(20,40){\vector(0,-1){40}}
 \put(21,30){\makebox(0,0){\normalsize$\rightarrow$}}
 \put(23,40){\makebox(0,0)[l]{\normalsize$a$}}
 \put(-5,14){\framebox(12,12){\normalsize$y$}}
 \put(33,14){\framebox(28,12){\normalsize$y*a$}}
\end{picture}
\end{center}
holds for adjacent regions.
We denote by $\operatorname{Col}_{X_Y}(D)$ the set of $X_Y$-colorings of $D$.
\end{definition}
The next theorem can be proven with the same argument as the one for Theorem~\ref{thm:coloring1}.
It shows that for any two diagrams $D$ and $D'$ of the same $S^1$-oriented handlebody-link, there exists a one-to-one correspondence between $\operatorname{Col}_{X_Y}(D)$ and $\operatorname{Col}_{X_Y}(D')$.

\begin{theorem}
Let $X=\bigsqcup_{\lambda\in\Lambda}G_\lambda$ be a multiple conjugation biquandle and $Y$ an $X$-set.
Let $D$ be a diagram of an $S^1$-oriented handlebody-link $H$.
Let $D'$ be a diagram obtained by applying one of the Y-oriented R1--R6 moves to the diagram $D$ once.
For an $X_Y$-coloring $C$ of $D$, there is a unique $X_Y$-coloring $C'$ of $D'$ which coincides with $C$ except the place where the move is applied.
\end{theorem}

\section{Prismatic chain complex} \label{sect:chaincomplex}

In this section, we define a chain complex called the \textit{prismatic chain complex} for multiple conjugation biquandles that contains aspects of both group and biquandle homology theories.
This is a counterpart, for biquandles, of the prismatic chain complex for multiple conjugation quandles introduced in \cite{CarterIshiiSaitoTanaka16}.

Let $X=\bigsqcup_{\lambda\in\Lambda}G_\lambda$ be a multiple conjugation biquandle, $Y$ an $X$-set.
Let $P_n(X)_Y$ be the free abelian group generated by the elements
\[ (y;x_{1,1},\ldots,x_{1,n_1};\ldots;x_{k,1},\ldots,x_{k,n_k})\in\bigcup_{n_1+\cdots+n_k=n}Y\times\prod_{i=1}^k\bigcup_{\lambda\in\Lambda}G_\lambda^{n_i} \]
if $n\geq0$, let $P_n(X)_Y=0$ otherwise.
We represent
\[ (y;x_{1,1},\ldots,x_{1,n_1};\ldots;x_{k,1},\ldots,x_{k,n_k}) \]
by the noncommutative multiplication form
\[ \langle y\rangle\langle x_{1,1},\ldots,x_{1,n_1}\rangle\cdots\langle x_{k,1},\ldots,x_{k,n_k}\rangle, \]
and assume the linearity with respect to the ``noncommutative multiplication'' as
\begin{align*}
&\langle y\rangle\langle\boldsymbol{x}_1\rangle\cdots(\alpha\langle\boldsymbol{x}_i\rangle+\alpha'\langle\boldsymbol{x}'_i\rangle)\cdots\langle\boldsymbol{x}_k\rangle \\
&=\alpha\langle y\rangle\langle\boldsymbol{x}_1\rangle\cdots\langle\boldsymbol{x}_i\rangle\cdots\langle\boldsymbol{x}_k\rangle
+\alpha'\langle y\rangle\langle\boldsymbol{x}_1\rangle\cdots\langle\boldsymbol{x}'_i\rangle\cdots\langle\boldsymbol{x}_k\rangle,
\end{align*}
where a bold symbol $\boldsymbol{x}_i$ stands for the sequence $x_{i,1},\ldots,x_{i,n_i}$.

We introduce notions for bold symbols:
We denote by $|\langle\boldsymbol{x}\rangle|$ the length $m$ of $\langle\boldsymbol{x}\rangle=\langle x_1,\ldots,x_m\rangle$, and set 
$|\langle y\rangle\langle\boldsymbol{x}_1\rangle\cdots\langle\boldsymbol{x}_k\rangle|
=|\langle\boldsymbol{x}_1\rangle|+\cdots+|\langle\boldsymbol{x}_k\rangle|$.
We denote by $\langle\boldsymbol{x}_{\widehat{j_1,\ldots,j_s}}\rangle$ the bracket
\[ \langle x_1,\ldots,x_{j_1-1},x_{j_1+1},\ldots,x_{j_2-1},x_{j_2+1},\ldots,x_{j_s-1},x_{j_s+1},\ldots,x_m\rangle, \]
and set $\langle\boldsymbol{x}_{i,\widehat{\boldsymbol{j}}}\rangle=\langle(\boldsymbol{x}_i)_{\widehat{\boldsymbol{j}}}\rangle$.
The operations $\uline{*}a$, $\oline{*}a$ and $a\cdot$ are applied on a sequence as follows:
\begin{align*}
&\langle\boldsymbol{x}\uline{*}a\rangle
=\langle x_1\uline{*}a,\ldots,x_m\uline{*}a\rangle,
&\langle\boldsymbol{x}\oline{*}a\rangle
=\langle x_1\oline{*}a,\ldots,x_m\oline{*}a\rangle, \mbox{ and} \\
&\langle a\boldsymbol{x}\rangle
=\langle ax_1,\ldots,ax_m\rangle.
\end{align*}
The operations $\uline{*}a,\oline{*}a$ are also applied on brackets as follows:
\begin{align*}
\langle y\rangle\langle\boldsymbol{x}_1\rangle
\cdots\langle\boldsymbol{x}_i\rangle\uline{*}a
\langle\boldsymbol{x}_{i+1}\rangle
\cdots\langle\boldsymbol{x}_k\rangle
&=\langle y*a\rangle\langle\boldsymbol{x}_1\uline{*}a\rangle
\cdots\langle\boldsymbol{x}_i\uline{*}a\rangle
\langle\boldsymbol{x}_{i+1}\rangle
\cdots\langle\boldsymbol{x}_k\rangle, \mbox{ and}\\
\langle y\rangle\langle\boldsymbol{x}_1\rangle
\cdots\langle\boldsymbol{x}_{i-1}\rangle
\oline{*}a\langle\boldsymbol{x}_i\rangle
\cdots\langle\boldsymbol{x}_k\rangle
&=\langle y\rangle\langle\boldsymbol{x}_1\rangle
\cdots\langle\boldsymbol{x}_{i-1}\rangle
\langle\boldsymbol{x}_i\oline{*}a\rangle
\cdots\langle\boldsymbol{x}_k\oline{*}a\rangle.
\end{align*}

We define a boundary operation $\partial_n:P_n(X)_Y\to P_{n-1}(X)_Y$ by
\[ \partial_n(
\langle y\rangle\langle\boldsymbol{x}_1\rangle
\cdots\langle\boldsymbol{x}_k\rangle)
=\sum_{i=1}^k
(-1)^{|\langle y\rangle\langle\boldsymbol{x}_1\rangle
\cdots\langle\boldsymbol{x}_{i-1}\rangle|}
\langle y\rangle\langle\boldsymbol{x}_1\rangle
\cdots\partial\langle\boldsymbol{x}_i\rangle
\cdots\langle\boldsymbol{x}_k\rangle, \]
where
$\partial(\boldsymbol{x})
=\uline{*}x_1\oline{*}x_1\langle x_1^{-1}\boldsymbol{x}_{\widehat{1}}\rangle
+\sum_{j=1}^{|\langle\boldsymbol{x}_{\widehat{j}}\rangle|}(-1)^j\langle\boldsymbol{x}_{\widehat{j}}\rangle$.
We set $\widetilde{\partial}\langle\boldsymbol{x}\rangle=\uline{*}x_1\oline{*}x_1\langle x_1^{-1}\boldsymbol{x}_{\widehat{1}}\rangle$
and $\widehat{\partial}\langle\boldsymbol{x}\rangle=\sum_{j=1}^{|\langle\boldsymbol{x}_{\widehat{j}}\rangle|}(-1)^j\langle\boldsymbol{x}_{\widehat{j}}\rangle$,
that is,
\begin{align*}
&\widetilde{\partial}\langle x_1,\ldots,x_m\rangle
=\uline{*}x_1\langle x_1^{-1}x_2\oline{*}x_1,\ldots,x_1^{-1}x_m\oline{*}x_1\rangle\oline{*}x_1, \mbox{ and} \\
&\widehat{\partial}\langle x_1,\ldots,x_m\rangle
=\sum_{i=1}^m(-1)^i\langle x_1,\ldots,x_{i-1},x_{i+1},\ldots,x_m\rangle.
\end{align*}
Then $\partial=\widetilde{\partial}+\widehat{\partial}$.
Putting
\begin{align*}
&\widetilde{\partial}_n(
\langle y\rangle\langle\boldsymbol{x}_1\rangle
\cdots\langle\boldsymbol{x}_k\rangle) \\
&=\sum_{i=1}^k
(-1)^{|\langle y\rangle\langle\boldsymbol{x}_1\rangle
\cdots\langle\boldsymbol{x}_{i-1}\rangle|}
\langle y\rangle\langle\boldsymbol{x}_1\rangle
\cdots\widetilde{\partial}\langle\boldsymbol{x}_i\rangle
\cdots\langle\boldsymbol{x}_k\rangle \\
&=\sum_{i=1}^k
(-1)^{|\langle y\rangle\langle\boldsymbol{x}_1\rangle
\cdots\langle\boldsymbol{x}_{i-1}\rangle|}
\langle y*x_{i,1}\rangle\langle\boldsymbol{x}_1\uline{*}x_{i,1}\rangle
\cdots\langle x_{i,1}^{-1}\boldsymbol{x}_{i,\widehat{1}}\oline{*}x_{i,1}\rangle
\cdots\langle\boldsymbol{x}_k\oline{*}x_{i,1}\rangle, \mbox{ and} \\
&\widehat{\partial}_n(
\langle y\rangle\langle\boldsymbol{x}_1\rangle
\cdots\langle\boldsymbol{x}_k\rangle) \\
&=\sum_{i=1}^k
(-1)^{|\langle y\rangle\langle\boldsymbol{x}_1\rangle
\cdots\langle\boldsymbol{x}_{i-1}\rangle|}
\langle y\rangle\langle\boldsymbol{x}_1\rangle
\cdots\widehat{\partial}\langle\boldsymbol{x}_i\rangle
\cdots\langle\boldsymbol{x}_k\rangle \\
&=\sum_{i=1}^k\sum_{j=1}^{|\langle\boldsymbol{x_i}\rangle|}
(-1)^{|\langle y\rangle\langle\boldsymbol{x}_1\rangle
\cdots\langle\boldsymbol{x}_{i-1}\rangle|+j}
\langle y\rangle\langle\boldsymbol{x}_1\rangle
\cdots\langle\boldsymbol{x}_{i,\widehat{j}}\rangle
\cdots\langle\boldsymbol{x}_k\rangle,
\end{align*}
we have $\partial_n=\widetilde{\partial}_n+\widehat{\partial}_n$.

\begin{example}
The boundary maps in $1$-, $2$- and $3$-dimensions are computed as follows:
\begin{align*}
\partial_1(\langle y\rangle\langle x_1\rangle)
&=\langle y*x_1\rangle-\langle y\rangle, \\
\partial_2(\langle y\rangle\langle x_1\rangle\langle x_2\rangle)
&=\langle y\rangle\partial\langle x_1\rangle\langle x_2\rangle
-\langle y\rangle\langle x_1\rangle\partial\langle x_2\rangle \\
&=\langle y*x_1\rangle\langle x_2\oline{*}x_1\rangle
-\langle y\rangle\langle x_2\rangle
-\langle y*x_2\rangle\langle x_1\uline{*}x_2\rangle
+\langle y\rangle\langle x_1\rangle, \\
\partial_2(\langle y\rangle\langle x_1,x_2\rangle)
&=\langle y*x_1\rangle\langle x_1^{-1}x_2\oline{*}x_1\rangle
-\langle y\rangle\langle x_2\rangle
+\langle y\rangle\langle x_1\rangle, \\
\partial_3(\langle y\rangle\langle x_1\rangle\langle x_2\rangle\langle x_3\rangle)
&=\langle y\rangle\partial\langle x_1\rangle\langle x_2\rangle\langle x_3\rangle
-\langle y\rangle\langle x_1\rangle\partial\langle x_2\rangle\langle x_3\rangle
+\langle y\rangle\langle x_1\rangle\langle x_2\rangle\partial\langle x_3\rangle \\
&=\langle y*x_1\rangle\langle x_2\oline{*}x_1\rangle\langle x_3\oline{*}x_1\rangle
-\langle y\rangle\langle x_2\rangle\langle x_3\rangle \\
&\hspace{5mm}
-\langle y*x_2\rangle\langle x_1\uline{*}x_2\rangle\langle x_3\oline{*}x_2\rangle
+\langle y\rangle\langle x_1\rangle\langle x_3\rangle \\
&\hspace{5mm}
+\langle y*x_3\rangle\langle x_1\uline{*}x_3\rangle\langle x_2\uline{*}x_3\rangle
-\langle y\rangle\langle x_1\rangle\langle x_2\rangle, \\
\partial_3(\langle y\rangle\langle x_1\rangle\langle x_2,x_3\rangle)
&=\langle y\rangle\partial\langle x_1\rangle\langle x_2,x_3\rangle
-\langle y\rangle\langle x_1\rangle\partial\langle x_2,x_3\rangle \\
&=\langle y*x_1\rangle\langle x_2\oline{*}x_1,x_3\oline{*}x_1\rangle
-\langle y\rangle\langle x_2,x_3\rangle \\
&\hspace{5mm}
-\langle y*x_2\rangle\langle x_1\uline{*}x_2\rangle\langle x_2^{-1}x_3\oline{*}x_2\rangle \\
&\hspace{5mm}
+\langle y\rangle\langle x_1\rangle\langle x_3\rangle
-\langle y\rangle\langle x_1\rangle\langle x_2\rangle, \\
\partial_3(\langle y\rangle\langle x_1,x_2\rangle\langle x_3\rangle)
&=\langle y\rangle\partial\langle x_1,x_2\rangle\langle x_3\rangle
+\langle y\rangle\langle x_1,x_2\rangle\partial\langle x_3\rangle \\
&=\langle y*x_1\rangle\langle x_1^{-1}x_2\oline{*}x_1\rangle\langle x_3\oline{*}x_1\rangle
-\langle y\rangle\langle x_2\rangle\langle x_3\rangle \\
&\hspace{5mm}
+\langle y\rangle\langle x_1\rangle\langle x_3\rangle
+\langle y*x_3\rangle\langle x_1\uline{*}x_3,x_2\uline{*}x_3\rangle
-\langle y\rangle\langle x_1,x_2\rangle, \mbox{ and} \\
\partial_3(\langle y\rangle\langle x_1,x_2,x_3\rangle)
&=\langle y*x_1\rangle\langle x_1^{-1}x_2\oline{*}x_1,x_1^{-1}x_3\oline{*}x_1\rangle \\
&\hspace{5mm}
-\langle y\rangle\langle x_2,x_3\rangle
+\langle y\rangle\langle x_1,x_3\rangle
-\langle y\rangle\langle x_1,x_2\rangle.
\end{align*}
\end{example}

\begin{figure}
\mbox{}\hfill
\begin{picture}(80,100)
 \put(20,40){\line(1,1){40}}
 \put(60,40){\line(-1,1){17}}
 \put(20,80){\line(1,-1){17}}
\thicklines
 \put(40,20){\line(-1,1){40}}
 \put(40,20){\line(1,1){40}}
 \put(40,100){\line(-1,-1){40}}
 \put(40,100){\line(1,-1){40}}
 \put(23,37){\vector(1,-1){0}}
 \put(23,83){\vector(1,1){0}}
 \put(63,43){\vector(1,1){0}}
 \put(63,77){\vector(1,-1){0}}
 \put(17,83){\makebox(0,0)[br]{$a$}}
 \put(63,83){\makebox(0,0)[bl]{$b\oline{*}a$}}
 \put(17,37){\makebox(0,0)[tr]{$b$}}
 \put(63,37){\makebox(0,0)[tl]{$a\uline{*}b$}}
 \put(40,5){\makebox(0,0){$\langle a\rangle\langle b\rangle$}}
\end{picture}
\begin{picture}(60,100)
 \put(30,60){\makebox(0,0){$\overset{\partial_2}{\to}$}}
\end{picture}
\begin{picture}(80,100)
\thicklines
 \put(38,22){\line(-1,1){36}}
 \put(42,22){\line(1,1){36}}
 \put(38,98){\line(-1,-1){36}}
 \put(42,98){\line(1,-1){36}}
 \put(23,37){\vector(1,-1){0}}
 \put(23,83){\vector(1,1){0}}
 \put(63,43){\vector(1,1){0}}
 \put(63,77){\vector(1,-1){0}}
 \put(17,83){\makebox(0,0)[br]{$\langle a\rangle$}}
 \put(63,83){\makebox(0,0)[bl]{$\langle b\oline{*}a\rangle$}}
 \put(17,37){\makebox(0,0)[tr]{$-\langle b\rangle$}}
 \put(63,37){\makebox(0,0)[tl]{$-\langle a\uline{*}b\rangle$}}
\end{picture}
\hfill\mbox{}\\
\mbox{}\hfill
\begin{picture}(80,100)
 \put(40,30){\line(0,1){16}}
 \put(40,46){\line(-2,1){20}}
 \put(40,46){\line(2,1){20}}
\thicklines
 \put(0,30){\line(3,4){40}}
 \put(0,30){\line(1,0){80}}
 \put(80,30){\line(-3,4){40}}
 \put(22,59){\vector(3,4){0}}
 \put(63,53){\vector(3,-4){0}}
 \put(43,30){\vector(1,0){0}}
 \put(17,59){\makebox(0,0)[br]{$a$}}
 \put(63,59){\makebox(0,0)[bl]{$a^{-1}b\oline{*}a$}}
 \put(40,27){\makebox(0,0)[t]{$b$}}
 \put(40,5){\makebox(0,0){$\langle a,b\rangle$}}
\end{picture}
\begin{picture}(60,100)
 \put(30,60){\makebox(0,0){$\overset{\partial_2}{\to}$}}
\end{picture}
\begin{picture}(80,100)
\thicklines
 \put(3,34){\line(3,4){35}}
 \put(4,30){\line(1,0){72}}
 \put(77,34){\line(-3,4){35}}
 \put(22,59){\vector(3,4){0}}
 \put(63,53){\vector(3,-4){0}}
 \put(43,30){\vector(1,0){0}}
 \put(17,59){\makebox(0,0)[br]{$\langle a\rangle$}}
 \put(63,59){\makebox(0,0)[bl]{$\langle a^{-1}b\oline{*}a\rangle$}}
 \put(40,27){\makebox(0,0)[t]{$-\langle b\rangle$}}
\end{picture}
\hfill\mbox{}
\caption{A geometric interpretation of the boundary operation.}
\label{fig:geominterpretation}
\end{figure}
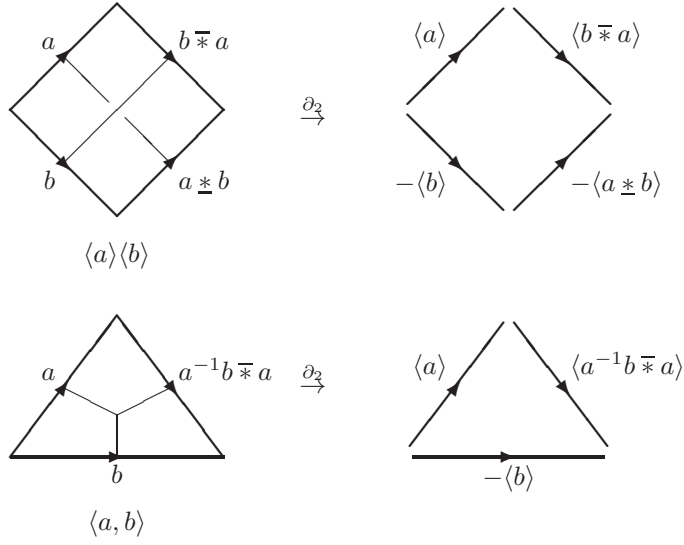

\begin{remark}
As in \cite{CarterIshiiSaitoTanaka16}, there is a geometric interpretation of the boundary operations.
For example, for
$\partial_2(\langle y\rangle\langle x_1\rangle\langle x_2\rangle)$ and $\partial_2(\langle y\rangle\langle x_1,x_2\rangle)$,
see Figure~\ref{fig:geominterpretation}.
\end{remark}

\begin{lemma}
It holds that
$\partial^{(i)}\circ\partial^{(j)}=\partial^{(j)}\circ\partial^{(i)}$ ($i\neq j$),
where $\partial^{(i)}=\widetilde{\partial}^{(i)}+\widehat{\partial}^{(i)}$,
\begin{align*}
&\widetilde{\partial}^{(i)}(
\langle y\rangle\langle\boldsymbol{x}_1\rangle
\cdots\langle\boldsymbol{x}_l\rangle)
=\langle y\rangle\langle\boldsymbol{x}_1\rangle
\cdots\widetilde{\partial}\langle\boldsymbol{x}_i\rangle
\cdots\langle\boldsymbol{x}_l\rangle, \mbox{ and} \\
&\widehat{\partial}^{(i)}(
\langle y\rangle\langle\boldsymbol{x}_1\rangle
\cdots\langle\boldsymbol{x}_l\rangle)
=\langle y\rangle\langle\boldsymbol{x}_1\rangle
\cdots\widehat{\partial}\langle\boldsymbol{x}_i\rangle
\cdots\langle\boldsymbol{x}_l\rangle.
\end{align*}
In particular, the notation
$\langle y\rangle\langle\boldsymbol{x}_1\rangle
\cdots\partial\langle\boldsymbol{x}_i\rangle
\cdots\partial\langle\boldsymbol{x}_j\rangle
\cdots\langle\boldsymbol{x}_l\rangle$
is well-defined.
\end{lemma}

\begin{proof}
We may assume that $i<j$.
Then we have 
$\widetilde{\partial}^{(i)}\circ\widetilde{\partial}^{(j)}
=\widetilde{\partial}^{(j)}\circ\widetilde{\partial}^{(i)}$ since
\begin{align*}
&\widetilde{\partial}^{(i)}\circ\widetilde{\partial}^{(j)}(
\langle y\rangle\langle\boldsymbol{x}_1\rangle
\cdots\langle\boldsymbol{x}_i\rangle
\cdots\langle\boldsymbol{x}_k\rangle
\cdots\langle\boldsymbol{x}_j\rangle
\cdots\langle\boldsymbol{x}_n\rangle)\\
&=\widetilde{\partial}^{(i)}(
\langle y*x_{j,1}\rangle
\langle\boldsymbol{x}_1\uline{*}x_{j,1}\rangle
\cdots\langle\boldsymbol{x}_i\uline{*}x_{j,1}\rangle
\cdots\langle\boldsymbol{x}_k\uline{*}x_{j,1}\rangle
\cdots\langle x_{j,1}^{-1}\boldsymbol{x}_{j,\widehat{1}}\oline{*}x_{j,1}\rangle
\cdots\langle\boldsymbol{x}_n\oline{*}x_{j,1}\rangle) \\
&=
\langle(y*x_{j,1})*(x_{i,1}\uline{*}x_{j,1})\rangle
\langle(\boldsymbol{x}_1\uline{*}x_{j,1})\uline{*}(x_{i,1}\uline{*}x_{j,1})\rangle \\
&\hspace{1em}
\cdots\langle(x_{i,1}\uline{*}x_{j,1})^{-1}(\boldsymbol{x}_{i,\widehat{1}}\uline{*}x_{j,1})\oline{*}(x_{i,1}\uline{*}x_{j,1})\rangle
\cdots\langle(\boldsymbol{x}_k\uline{*}x_{j,1})\oline{*}(x_{i,1}\uline{*}x_{j,1})\rangle \\
&\hspace{1em}
\cdots\langle(x_{j,1}^{-1}\boldsymbol{x}_{j,\widehat{1}}\oline{*}x_{j,1})\oline{*}(x_{i,1}\uline{*}x_{j,1})\rangle
\cdots\langle(\boldsymbol{x}_n\oline{*}x_{j,1})\oline{*}(x_{i,1}\uline{*}x_{j,1})\rangle \\
&=\langle(y*x_{i,1})*(x_{j,1}\oline{*}x_{i,1})\rangle
\langle(\boldsymbol{x}_1\uline{*}x_{i,1})\uline{*}(x_{j,1}\oline{*}x_{i,1})\rangle \\
&\hspace{1em}
\cdots\langle(x_{i,1}^{-1}\boldsymbol{x}_{i,\widehat{1}}\uline{*}x_{j,1})\oline{*}(x_{i,1}\uline{*}x_{j,1})\rangle
\cdots\langle(\boldsymbol{x}_k\uline{*}x_{j,1})\oline{*}(x_{i,1}\uline{*}x_{j,1})\rangle \\
&\hspace{1em}
\cdots\langle(x_{j,1}^{-1}\boldsymbol{x}_{j,\widehat{1}}\oline{*}x_{i,1})\oline{*}(x_{j,1}\oline{*}x_{i,1})\rangle
\cdots\langle(\boldsymbol{x}_n\oline{*}x_{i,1})\oline{*}(x_{j,1}\oline{*}x_{i,1})\rangle \\
&=\langle(y*x_{i,1})*(x_{j,1}\oline{*}x_{i,1})\rangle
\langle(\boldsymbol{x}_1\uline{*}x_{i,1})\uline{*}(x_{j,1}\oline{*}x_{i,1})\rangle \\
&\hspace{1em}
\cdots\langle(x_{i,1}^{-1}\boldsymbol{x}_{i,\widehat{1}}\oline{*}x_{i,1})\uline{*}(x_{j,1}\oline{*}x_{i,1})\rangle
\cdots\langle(\boldsymbol{x}_k\oline{*}x_{i,1})\uline{*}(x_{j,1}\oline{*}x_{i,1})\rangle \\
&\hspace{1em}
\cdots\langle(x_{j,1}\oline{*}x_{i,1})^{-1}(\boldsymbol{x}_{j,\widehat{1}}\oline{*}x_{i,1})\oline{*}(x_{j,1}\oline{*}x_{i,1})\rangle
\cdots\langle(\boldsymbol{x}_n\oline{*}x_{i,1})\oline{*}(x_{j,1}\oline{*}x_{i,1})\rangle \\
&=\widetilde{\partial}^{(j)}(
\langle y*x_{i,1}\rangle
\langle\boldsymbol{x}_1\uline{*}x_{i,1}\rangle
\cdots\langle x_{i,1}^{-1}\boldsymbol{x}_{i,\widehat{1}}\oline{*}x_{i,1}\rangle
\cdots\langle\boldsymbol{x}_k\oline{*}x_{i,1}\rangle
\cdots\langle\boldsymbol{x}_j\oline{*}x_{i,1}\rangle
\cdots\langle\boldsymbol{x}_n\oline{*}x_{i,1}\rangle) \\
&=\widetilde{\partial}^{(j)}\circ\widetilde{\partial}^{(i)}(
\langle y\rangle\langle\boldsymbol{x}_1\rangle
\cdots\langle\boldsymbol{x}_i\rangle
\cdots\langle\boldsymbol{x}_k\rangle
\cdots\langle\boldsymbol{x}_j\rangle
\cdots\langle\boldsymbol{x}_n\rangle).
\end{align*}

Similarly, we have
$\widetilde{\partial}^{(i)}\circ\widehat{\partial}^{(j)}
=\widehat{\partial}^{(j)}\circ\widetilde{\partial}^{(i)}$,
$\widehat{\partial}^{(i)}\circ\widetilde{\partial}^{(j)}
=\widetilde{\partial}^{(j)}\circ\widehat{\partial}^{(i)}$,
$\widehat{\partial}^{(i)}\circ\widehat{\partial}^{(j)}
=\widehat{\partial}^{(j)}\circ\widehat{\partial}^{(i)}$.
Therefore we have  
\begin{align*}
\partial^{(i)}\circ\partial^{(j)}&=\widetilde{\partial}^{(i)}\circ \widetilde{\partial}^{(j)}+\widetilde{\partial}^{(i)}\circ\widehat{\partial}^{(j)} +\widehat{\partial}^{(i)} \circ \widetilde{\partial}^{(j)} +\widehat{\partial}^{(i)} \circ \widehat{\partial}^{(j)} \\
&=\widetilde{\partial}^{(j)}\circ \widetilde{\partial}^{(i)}+\widehat{\partial}^{(j)}\circ \widetilde{\partial}^{(i)} +\widetilde{\partial}^{(j)} \circ \widehat{\partial}^{(i)} +\widehat{\partial}^{(j)} \circ \widehat{\partial}^{(i)} \\
&=\partial^{(j)}\circ\partial^{(i)}.
\end{align*}
\end{proof}

\begin{proposition}
$P_*(X)_Y=(P_n(X)_Y,\partial_n)$ is a chain complex.
\end{proposition}

\begin{proof}
We have 
$\widetilde{\partial}\circ\widetilde{\partial}
+\widetilde{\partial}\circ\widehat{\partial}
+\widehat{\partial}\circ\widetilde{\partial}=0$
since
\begin{align*}
&\langle y\rangle\langle\boldsymbol{x}_1\rangle
\cdots(\widetilde{\partial}\circ\widehat{\partial})\langle\boldsymbol{a}\rangle
\cdots\langle\boldsymbol{x}_n\rangle \\
&=\sum_{j=1}^m(-1)^j\langle y\rangle
\langle\boldsymbol{x}_1\rangle
\cdots\widetilde{\partial}\langle\boldsymbol{a}_{\widehat{j}}\rangle
\cdots\langle\boldsymbol{x}_n\rangle \\
&=-\langle y*a_2\rangle
\langle\boldsymbol{x}_1\uline{*}a_2\rangle
\cdots\langle a_2^{-1}\boldsymbol{a}_{\widehat{1,2}}\oline{*}a_2\rangle
\cdots\langle\boldsymbol{x}_n\oline{*}a_2\rangle \\
&\hspace{1em}+\sum_{j=2}^m(-1)^j\langle y*a_1\rangle
\langle\boldsymbol{x}_1\uline{*}a_1\rangle
\cdots\langle a_1^{-1}\boldsymbol{a}_{\widehat{1,j}}\oline{*}a_1\rangle
\cdots\langle\boldsymbol{x}_n\oline{*}a_1\rangle, \\
&\langle y\rangle\langle\boldsymbol{x}_1\rangle
\cdots(\widehat{\partial}\circ\widetilde{\partial})\langle\boldsymbol{a}\rangle
\cdots\langle\boldsymbol{x}_n\rangle \\
&=\langle y*a_1\rangle
\langle\boldsymbol{x}_1\uline{*}a_1\rangle
\cdots\widehat{\partial}\langle a_1^{-1}\boldsymbol{a}_{\widehat{1}}\oline{*}a_1\rangle
\cdots\langle\boldsymbol{x}_n\oline{*}a_1\rangle \\
&=\sum_{j=2}^m(-1)^{j-1}
\langle y*a_1\rangle\langle\boldsymbol{x}_1\uline{*}a_1\rangle
\cdots\langle a_1^{-1}\boldsymbol{a}_{\widehat{1,j}}\oline{*}a_1\rangle
\cdots\langle\boldsymbol{x}_n\oline{*}a_1\rangle, \mbox{ and} \\
&\langle y\rangle\langle\boldsymbol{x}_1\rangle
\cdots(\widetilde{\partial}\circ\widetilde{\partial})\langle\boldsymbol{a}\rangle
\cdots\langle\boldsymbol{x}_n\rangle \\
&=\langle y*a_1\rangle\langle\boldsymbol{x}_1\uline{*}a_1\rangle
\cdots\widetilde{\partial}\langle a_1^{-1}\boldsymbol{a}_{\widehat{1}}\oline{*}a_1\rangle
\cdots\langle\boldsymbol{x}_n\oline{*}a_1\rangle) \\
&=\langle(y*a_1)*(a_1^{-1}a_2\oline{*}a_1)\rangle
\langle(\boldsymbol{x}_1\uline{*}a_1)\uline{*}(a_1^{-1}a_2\oline{*}a_1)\rangle \\
&\hspace{1em}\cdots\langle(a_1^{-1}a_2\oline{*}a_1)^{-1}(a_1^{-1}\boldsymbol{a}_{\widehat{1,2}}\oline{*}a_1)\oline{*}(a_1^{-1}a_2\oline{*}a_1)\rangle \\
&\hspace{1em}\cdots\langle(\boldsymbol{x}_n\oline{*}a_1)\oline{*}(a_1^{-1}a_2\oline{*}a_1)\rangle \\
&=\langle y*a_2\rangle
\langle\boldsymbol{x}_1\uline{*}a_2\rangle
\cdots\langle(a_2^{-1}\boldsymbol{a}_{\widehat{1,2}}\oline{*}a_1)\oline{*}(a_1^{-1}a_2\oline{*}a_1)\rangle
\cdots\langle\boldsymbol{x}_n\oline{*}a_2\rangle \\
&=\langle y*a_2\rangle
\langle\boldsymbol{x}_1\uline{*}a_2\rangle
\cdots\langle a_2^{-1}\boldsymbol{a}_{\widehat{1,2}}\oline{*}a_2\rangle
\cdots\langle\boldsymbol{x}_n\oline{*}a_2\rangle.
\end{align*}
Besides, we have $\widehat{\partial}\circ\widehat{\partial}=0$ since 
\begin{align*}
&\langle y\rangle\langle\boldsymbol{x}_1\rangle
\cdots(\widehat{\partial}\circ\widehat{\partial})\langle\boldsymbol{a}\rangle
\cdots\langle\boldsymbol{x}_n\rangle \\
&=\sum_{i=1}^m(-1)^i
\langle y\rangle\langle\boldsymbol{x}_1\rangle
\cdots\widehat{\partial}\langle\boldsymbol{a}_{\widehat{i}}\rangle
\cdots\langle\boldsymbol{x}_n\rangle \\
&=\sum_{i<j}(-1)^{i+j-1}\langle y\rangle
\langle\boldsymbol{x}_1\rangle
\cdots\langle\boldsymbol{a}_{\widehat{i,j}}\rangle
\cdots\langle\boldsymbol{x}_n\rangle
+\sum_{i>j}(-1)^{i+j}
\langle y\rangle\langle\boldsymbol{x}_1\rangle
\cdots\langle\boldsymbol{a}_{\widehat{j,i}}\rangle
\cdots\langle\boldsymbol{x}_n\rangle \\
&=0.
\end{align*}
Thus $\partial\circ\partial=0$ holds.
Therefore
\begin{align*}
&(\partial_{n-1}\circ\partial_n)(
\langle y\rangle\langle\boldsymbol{x}_1\rangle
\cdots\langle\boldsymbol{x}_l\rangle) \\
&=\sum_{i=1}^l
(-1)^{|\langle y\rangle\langle\boldsymbol{x}_1\rangle
\cdots\langle\boldsymbol{x}_{i-1}\rangle|}
\partial_{n-1}(\langle y\rangle\langle\boldsymbol{x}_1\rangle
\cdots\partial\langle\boldsymbol{x}_i\rangle
\cdots\langle\boldsymbol{x}_l\rangle) \\
&=\sum_{i<j}
(-1)^{|\langle y\rangle\langle\boldsymbol{x}_1\rangle
\cdots\langle\boldsymbol{x}_{i-1}\rangle|
+|\langle y\rangle\langle\boldsymbol{x}_1\rangle
\cdots\partial\langle\boldsymbol{x}_i\rangle
\cdots\langle\boldsymbol{x}_{j-1}\rangle|}
\langle y\rangle\langle\boldsymbol{x}_1\rangle
\cdots\partial\langle\boldsymbol{x}_i\rangle
\cdots\partial\langle\boldsymbol{x}_j\rangle
\cdots\langle\boldsymbol{x}_l\rangle \\
&\hspace{1em}+\sum_{i}
(-1)^{2|\langle y\rangle\langle\boldsymbol{x}_1\rangle
\cdots\langle\boldsymbol{x}_{i-1}\rangle|}
\langle y\rangle\langle\boldsymbol{x}_1\rangle
\cdots\partial\circ\partial\langle\boldsymbol{x}_j\rangle
\cdots\langle\boldsymbol{x}_l\rangle \\
&\hspace{1em}+\sum_{i>j}
(-1)^{|\langle y\rangle\langle\boldsymbol{x}_1\rangle
\cdots\langle\boldsymbol{x}_{i-1}\rangle|
+|\langle y\rangle\langle\boldsymbol{x}_1\rangle
\cdots\langle\boldsymbol{x}_{j-1}\rangle|}
\langle y\rangle\langle\boldsymbol{x}_1\rangle
\cdots\partial\langle\boldsymbol{x}_j\rangle
\cdots\partial\langle\boldsymbol{x}_i\rangle
\cdots\langle\boldsymbol{x}_l\rangle \\
&=0.
\end{align*}
\end{proof}

\section{Homology groups of multiple conjugation biquandles} \label{sect:degenerate}

In this section, we first define a degenerate subcomplex of the prismatic chain complex for multiple conjugation biquandles that is analogous to the subcomplex of the degenerate subcomplex for multiple conjugation quandles in \cite{CarterIshiiSaitoTanaka16}.

\begin{notation}
\begin{itemize}
\item
For any $s,t,k\in\mathbb{N}$ with $k\leq s$, let
\begin{align*}
&M(s,t)=\{\mu:\{1,\ldots,s\}\to\{1,\ldots,t\}\,|\,i<j\Rightarrow\mu(i)<\mu(j)\}, \\
&M(s,t,\widehat{k})=\{\mu:\{1,\ldots,s\}\setminus\{k\}\to\{1,\cdots,t\}\,|\,i<j\Rightarrow\mu(i)<\mu(j)\}, \\
&M(s,t,k)=\{\mu:\{1,\cdots,s\}\to\{1,\ldots,t\}\,|\,\text{$\mu(k)=\mu(k+1)$ and} \\
&\hspace{5cm}\text{$i<j\Rightarrow\mu(i)<\mu(j)$ if $(i,j)\neq(k,k+1)$}\}.
\end{align*}
\item For any map $\mu \in M(s,t)$ and any $j\in\mathbb{Z}$, we define $\lfloor j;\mu\rfloor \in\{0, 1,\ldots,s\}$ as follows.
\begin{align}
\lfloor j;\mu\rfloor &=
\left\{
\begin{array}{ll}
0 &  (j < \mu(1)), \\
k \mbox{ satisfying $\mu(k) \le j < \mu(k+1)$} & (\mu(1) \le j < \mu(s)),\\
s &  (\mu(s) \le j).
\end{array}
\right. \nonumber
\end{align}
\end{itemize}

We omit the map $\mu$ as $\lfloor j\rfloor=\lfloor j;\mu\rfloor$ unless it causes confusion.
\end{notation}

\begin{notation}
For
$\langle\boldsymbol{a}\rangle=\langle a_1,\ldots,a_s\rangle$ and
$\langle\boldsymbol{b}\rangle=\langle b_1,\ldots,b_t\rangle$
($a_i,b_j\in G_\lambda$), let
\begin{itemize}
\item
$\langle\langle\boldsymbol{a}\rangle\langle\boldsymbol{b}\rangle\rangle
=\displaystyle\sum_{\mu\in M(s,s+t)}
\langle\langle\boldsymbol{a}\rangle\langle\boldsymbol{b}\rangle\rangle_\mu$,
and
\item
$\langle\langle\boldsymbol{a}\rangle\langle\boldsymbol{b}\rangle\rangle_\mu
=(-1)^{\sum_{k=1}^s(\mu(k)-k)}
\langle
a_{\lfloor 1;\mu\rfloor}b_{1-\lfloor 1; \mu\rfloor},
\ldots,
a_{\lfloor s+t; \mu \rfloor}b_{s+t-\lfloor s+t;\mu \rfloor}
\rangle$, \\
where $a_0=b_0=e_{\lambda}$.

\end{itemize}
\end{notation}

\begin{example}
\begin{itemize}
\item
$\langle\langle a\rangle\langle b\rangle\rangle$
represents a chain with two terms as follows:
\[ \langle\langle a\rangle\langle b\rangle\rangle
=\langle a,ab\rangle-\langle b,ab \rangle. \]
\item
$\langle\langle a_1,a_2\rangle\langle b\rangle\rangle$
represents a chain with three terms as follows:
\[ \langle\langle a_1,a_2\rangle\langle b\rangle\rangle
=\langle a_1,a_2,a_2b\rangle-\langle a_1,a_1b,a_2b \rangle+\langle b,a_1b,a_2b \rangle. \]
\item
$\langle\langle a_1,a_2\rangle\langle b_1,b_2\rangle\rangle$
represents a chain with six terms as follows:
\begin{align*}
\langle\langle a_1,a_2\rangle\langle b_1,b_2\rangle\rangle
&=\langle a_1,a_2,a_2b_1,a_2b_2\rangle
-\langle a_1,a_1b_1,a_2b_1,a_2b_2\rangle \\
&\hspace{1em}
+\langle b_1,a_1b_1,a_2b_1,a_2b_2\rangle
+\langle a_1,a_1b_1,a_1b_2,a_2b_2\rangle \\
&\hspace{1em}
-\langle b_1,a_1b_1,a_1b_2,a_2b_2\rangle
+\langle b_1,b_2,a_1b_2,a_2b_2\rangle.
\end{align*}
Refer to the following diagram which shows
$\langle\langle a_1,a_2\rangle\langle b_1,b_2\rangle\rangle$:
\[ \small\begin{array}{c@{\,}c@{\,}c@{\,}c@{\,}c}
&& \langle b_1,a_1,a_2,b_2\rangle \\
& \swarrow\hspace{-1em}\nearrow && \searrow\hspace{-1em}\nwarrow \\
\langle a_1,a_2,b_1,b_2\rangle\leftrightarrow-\langle a_1,b_1,a_2,b_2\rangle &&&& -\langle b_1,a_1,b_2,a_2\rangle\leftrightarrow\langle b_1,b_2,a_1,a_2\rangle. \\
& \searrow\hspace{-1em}\nwarrow && \swarrow\hspace{-1em}\nearrow \\
&& \langle a_1,b_1,b_2,a_2\rangle
\end{array} \]
\end{itemize}
\end{example}

\begin{definition}
Let $D_n(X)_Y$ be the submodule of $P_n(X)_Y$ generated by elements of the form
$\langle y\rangle\langle\boldsymbol{a}_1\rangle\cdots
\langle\boldsymbol{a}\rangle\langle\boldsymbol{b}\rangle
\cdots\langle\boldsymbol{a}_k\rangle
-\langle y\rangle\langle\boldsymbol{a}_1\rangle\cdots
\langle\langle\boldsymbol{a}\rangle\langle\boldsymbol{b}\rangle\rangle
\cdots\langle\boldsymbol{a}_k\rangle$.
When $n\leq1$, we set $D_n(X)_Y=0$.
\end{definition}

\begin{example}
The submodule $D_2(X)_Y$ is generated by the elements of the form
\[ \langle y\rangle\langle a\rangle\langle b\rangle
-\langle y\rangle\langle a,ab\rangle
+\langle y\rangle\langle b,ab\rangle, \]
and $D_3(X)_Y$ is generated by the elements of the form
\begin{align*}
&\langle y\rangle\langle a\rangle\langle b\rangle\langle x\rangle
-\langle y\rangle\langle a,ab\rangle\langle x\rangle
+\langle y\rangle\langle b,ab\rangle\langle x\rangle, \\
&\langle y\rangle\langle x\rangle\langle b\rangle\langle c\rangle
-\langle y\rangle\langle x\rangle\langle b,bc\rangle
+\langle y\rangle\langle x\rangle\langle c,bc\rangle, \\
&\langle y\rangle\langle a,b\rangle\langle c\rangle
-\langle y\rangle\langle a,b,bc\rangle
+\langle y\rangle\langle a,ac,bc\rangle
-\langle y\rangle\langle c,ac,bc\rangle, \\
&\langle y\rangle\langle a\rangle\langle b,c\rangle
-\langle y\rangle\langle a,ab,ac\rangle
+\langle y\rangle\langle b,ab,ac\rangle
-\langle y\rangle\langle b,c,ac\rangle
\end{align*}
for $y\in Y$, $a,b,c\in G_\lambda$, $x\in X$.
\end{example}


\begin{remark}
As in \cite{CarterIshiiSaitoTanaka16}, the degenerate submodule $D_n(X)_Y$ corresponds to simplicial decompositions of prismatic complexes.
For example, for
$\langle y\rangle\langle a\rangle\langle b\rangle
-\langle y\rangle\langle a,ab\rangle
+\langle y\rangle\langle b,ab\rangle=0$,
see Figure~\ref{fig:geominterpretation2}, and for
$\langle y\rangle\langle a,b\rangle\langle c\rangle
-\langle y\rangle\langle a,b,bc\rangle
+\langle y\rangle\langle a,ac,bc\rangle
-\langle y\rangle\langle c,ac,bc\rangle=0$,
see Figure~\ref{fig:geominterpretation3}.
\end{remark}

\begin{figure}
\mbox{}\hfill
\begin{picture}(80,100)
 \put(20,40){\line(1,1){40}}
 \put(60,40){\line(-1,1){17}}
 \put(20,80){\line(1,-1){17}}
\thicklines
 \put(40,20){\line(-1,1){40}}
 \put(40,20){\line(1,1){40}}
 \put(40,100){\line(-1,-1){40}}
 \put(40,100){\line(1,-1){40}}
 \put(23,37){\vector(1,-1){0}}
 \put(23,83){\vector(1,1){0}}
 \put(63,43){\vector(1,1){0}}
 \put(63,77){\vector(1,-1){0}}
 \put(17,83){\makebox(0,0)[br]{$a$}}
 \put(63,83){\makebox(0,0)[bl]{$b\oline{*}a$}}
 \put(17,37){\makebox(0,0)[tr]{$b$}}
 \put(63,37){\makebox(0,0)[tl]{$a\uline{*}b$}}
 \put(40,5){\makebox(0,0){$\langle a\rangle\langle b\rangle$}}
\end{picture}
\begin{picture}(40,100)
 \put(20,60){\makebox(0,0){$=$}}
\end{picture}
\begin{picture}(80,100)
 \put(40,70){\line(-2,1){20}}
 \put(40,70){\line(2,1){20}}
 \put(40,50){\line(0,1){20}}
 \put(40,50){\line(-2,-1){20}}
 \put(40,50){\line(2,-1){20}}
\thicklines
 \put(40,20){\line(-1,1){40}}
 \put(40,20){\line(1,1){40}}
 \put(0,60){\line(1,0){80}}
 \put(40,100){\line(-1,-1){40}}
 \put(40,100){\line(1,-1){40}}
 \put(23,37){\vector(1,-1){0}}
 \put(23,83){\vector(1,1){0}}
 \put(43,60){\vector(1,0){0}}
 \put(63,43){\vector(1,1){0}}
 \put(63,77){\vector(1,-1){0}}
 \put(17,83){\makebox(0,0)[br]{$a$}}
 \put(63,83){\makebox(0,0)[bl]{$b\oline{*}a$}}
 \put(17,37){\makebox(0,0)[tr]{$b$}}
 \put(63,37){\makebox(0,0)[tl]{$a\uline{*}b$}}
 \put(43,63){\makebox(0,0)[bl]{$ab$}}
 \put(40,5){\makebox(0,0){$\langle a,ab\rangle-\langle b,ab\rangle$}}
\end{picture}
\hfill\mbox{}
\caption{A geometric interpretation of  the degenerate submodules.}
\label{fig:geominterpretation2}
\end{figure}

\begin{figure}
\begin{center}
 \includegraphics[width=100mm]{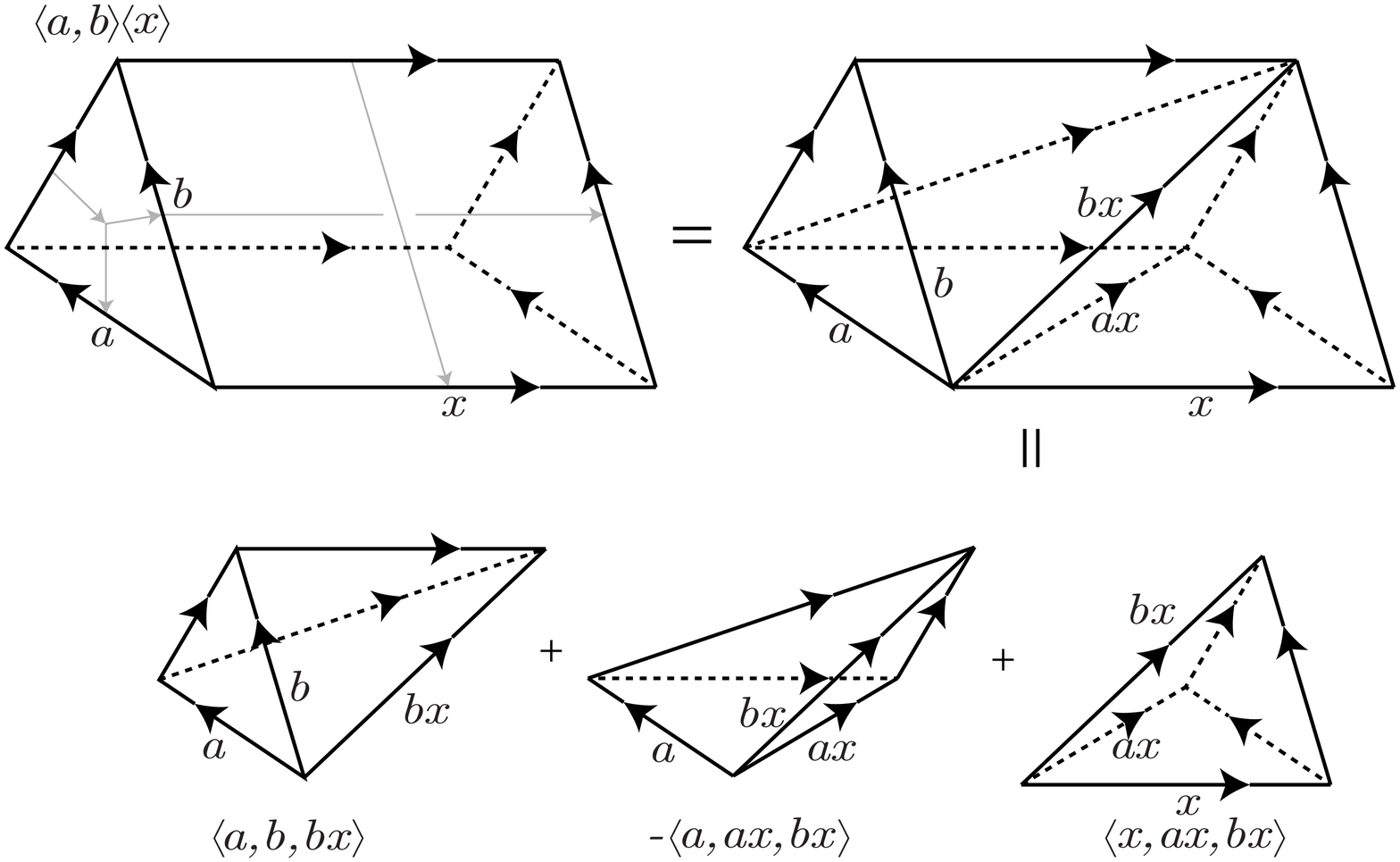}
\end{center}
\caption{A geometric interpretation of  the degenerate submodules.}
\label{fig:geominterpretation3}
\end{figure}

\begin{proposition}
$D_*(X)_Y=(D_n(X)_Y,\partial_n)$ is a subcomplex of $P_*(X)_Y$.
\end{proposition}

\begin{proof}
Let
$\langle\boldsymbol{a}\rangle=\langle a_1,\ldots,a_s\rangle$,
$\langle\boldsymbol{b}\rangle=\langle b_1,\ldots,b_t\rangle$
($a_i,b_j\in G_\lambda$).
By the definition,
\begin{align*}
\partial(\langle\boldsymbol{a}\rangle \langle\boldsymbol{b}\rangle)
&=\uline{*}a_1
\langle a_1^{-1}\boldsymbol{a}_{\widehat{1}}\oline{*}a_1\rangle
\langle\boldsymbol{b}\oline{*}a_1\rangle
\oline{*}a_1 \\
&\hspace{1em}+\sum_{i=1}^s(-1)^i
\langle\boldsymbol{a}_{\widehat{i}}\rangle
\langle\boldsymbol{b}\rangle \\
&\hspace{1em}+(-1)^s\uline{*}b_1
\langle\boldsymbol{a}\uline{*}b_1\rangle
\langle b_1^{-1}\boldsymbol{b}_{\widehat{1}}\oline{*}b_1\rangle
\oline{*}b_1 \\
&\hspace{1em}+\sum_{i=1}^t(-1)^{s+i}
\langle\boldsymbol{a}\rangle
\langle\boldsymbol{b}_{\widehat{i}}\rangle.
\end{align*}
It is sufficient to show that
$\partial(\langle\langle\boldsymbol{a}\rangle\langle\boldsymbol{b}\rangle\rangle)$ coincides with
\begin{align}
&\uline{*}a_1\langle
\langle a_1^{-1}\boldsymbol{a}_{\widehat{1}}\oline{*}a_1\rangle
\langle\boldsymbol{b}\oline{*}a_1\rangle
\rangle\oline{*}a_1 \label{eq1} \\
&+\sum_{i=1}^s(-1)^i\langle
\langle\boldsymbol{a}_{\widehat{i}}\rangle
\langle\boldsymbol{b}\rangle
\rangle \label{eq2} \\
&+(-1)^s\uline{*}b_1\langle
\langle\boldsymbol{a}\uline{*}b_1\rangle
\langle b_1^{-1}\boldsymbol{b}_{\widehat{1}}\oline{*}b_1\rangle
\rangle\oline{*}b_1 \label{eq3} \\
&+\sum_{i=1}^t(-1)^{s+i}\langle
\langle\boldsymbol{a}\rangle
\langle\boldsymbol{b}_{\widehat{i}}\rangle
\rangle. \label{eq4}
\end{align}

For the first term \eqref{eq1}, we have
\begin{align*}
\eqref{eq1}
&=\uline{*}a_1\left(
\sum_{\mu\in M(s-1,s+t-1)}(-1)^{\sum_{k=1}^{s-1}(\mu(k)-k)}\langle\boldsymbol{d}\rangle
\right)\oline{*}a_1 \\
&=\sum_{\mu \in M(s-1,s+t-1)}
(-1)^{\sum_{k=1}^{s-1}(\mu(k)-k)}
\uline{*}a_1\langle\boldsymbol{d}\rangle\oline{*}a_1,
\end{align*}
where $\langle\boldsymbol{d}\rangle=\langle d_1,\ldots,d_{s+t-1}\rangle$ and
\[ d_j
=(a_1^{-1}a_{\lfloor j;\mu\rfloor+1}\oline{*}a_1)(b_{j-\lfloor j;\mu \rfloor}\oline{*}a_1)
=a_1^{-1}a_{\lfloor j;\mu\rfloor+1}b_{j-\lfloor j;\mu \rfloor}\oline{*}a_1. \]
We define $\widetilde{\mu}\in M(s,s+t-1,\widehat{1})$ for $\mu\in M(s-1,s+t-1)$ by $\widetilde{\mu}(k)=\mu(k-1)$, and this induces a one-to-one correspondence from $M(s-1,s+t-1)$ to $M(s,s+t-1,\widehat{1})$.
It is noted that
\[ \lfloor j;\mu\rfloor=\begin{cases}
0 & \text{if $1\leq j<\widetilde{\mu}(2)$,} \\
\lfloor j;\widetilde{\mu}\rfloor-1 & \text{otherwise.}
\end{cases} \]
Then we have
\[ \eqref{eq1}
=\sum_{\widetilde{\mu}\in M(s,s+t-1,\widehat{1})}
(-1)^{\sum_{k=2}^{s}(\widetilde{\mu}(k)-k+1)}
\uline{*}a_1\langle\boldsymbol{d}\rangle\oline{*}a_1, \]
and
\[ d_j=\begin{cases}
b_j\oline{*}a_1 & \text{if $1\leq j<\widetilde{\mu}(2)$,} \\
a_1^{-1}a_{\lfloor j;\widetilde{\mu}\rfloor}b_{j-\lfloor j;\widetilde{\mu}\rfloor+1}\oline{*}a_1 & \text{otherwise}.
\end{cases} \]

For the second term \eqref{eq2}, we have
\[ \eqref{eq2}
=\sum_{i=1}^s(-1)^i\sum_{\mu\in M(s-1,s+t-1)}
(-1)^{\sum_{k=1}^{s-1}(\mu(k)-k)}
\langle\boldsymbol{d}\rangle, \]
where $\langle\boldsymbol{d}\rangle=\langle d_1,\ldots,d_{s+t-1}\rangle$ and
\[ d_j=\begin{cases}
a_{\lfloor j;\mu\rfloor}b_{j-\lfloor j;\mu\rfloor} & \text{if $\lfloor j;\mu\rfloor\leq i-1$,} \\
a_{\lfloor j;\mu\rfloor+1}b_{j-\lfloor j;\mu\rfloor} & \text{if $\lfloor j;\mu\rfloor\geq i$.}
\end{cases} \]
We define $\widetilde{\mu}\in M(s,s+t-1,\widehat{i})$ for $\mu\in M(s-1,s+t-1)$ by 
\[ \widetilde{\mu}(k)=\begin{cases}
\mu(k) & \text{if $k \leq i-1$,} \\
\mu(k-1) & \text{if $k \geq i+1$,}
\end{cases} \]
and this induces a one-to-one correspondence from $M(s-1,s+t-1)$ to $M(s,s+t-1,\widehat{i})$.
It is noted that
\[ \lfloor j;\widetilde{\mu}\rfloor=\begin{cases}
\lfloor j;\mu\rfloor & \text{if $\lfloor j;\mu\rfloor\leq i-1$,} \\
\lfloor j;\mu\rfloor+1 & \text{if $\lfloor j;\mu\rfloor\geq i$.}
\end{cases} \]
Then we have
\begin{align*}
\eqref{eq2}
&=\sum_{i=1}^s(-1)^i\sum_{\widetilde{\mu}\in M(s,s+t-1,\widehat{i})}
(-1)^{\sum_{k=1}^{i-1}(\widetilde{\mu}(k)-k)+\sum_{k=i+1}^s(\widetilde{\mu}(k)-k+1)}
\langle\boldsymbol{d}\rangle \\
&=-\sum_{i=1}^s\sum_{\widetilde{\mu}\in M(s,s+t-1,\widehat{i})}
(-1)^{\sum_{k=1}^{i-1}(\widetilde{\mu}(k)-k+1)+\sum_{k=i+1}^s(\widetilde{\mu}(k)-k+1)}
\langle\boldsymbol{d}\rangle,
\end{align*}
and
\[ d_j=\begin{cases}
a_{\lfloor j;\widetilde{\mu}\rfloor}b_{j-\lfloor j;\widetilde{\mu}\rfloor} & \text{if $\lfloor j;\widetilde{\mu}\rfloor\leq i-1$,} \\
a_{\lfloor j;\widetilde{\mu}\rfloor}b_{j-\lfloor j;\widetilde{\mu}\rfloor+1} & \text{if $\lfloor j;\widetilde{\mu}\rfloor\geq i+1$.}
\end{cases} \]

For the third term \eqref{eq3}, we have
\begin{align*}
\eqref{eq3}
&=(-1)^s\uline{*}b_1\left(
\sum_{\mu\in M(s, s+t-1)}(-1)^{\sum_{k=1}^s(\mu(k)-k)}\langle\boldsymbol{d}\rangle
\right)\oline{*}b_1 \\
&=\sum_{\mu\in M(s, s+t-1)}(-1)^{\sum_{k=1}^s(\mu(k)-k+1)}
\uline{*}b_1\langle\boldsymbol{d}\rangle\oline{*}b_1,
\end{align*}
where $\langle\boldsymbol{d}\rangle=\langle d_1,\ldots,d_{s+t-1}\rangle$ and
\begin{align*}
d_j
&=(a_{\lfloor j\rfloor}\uline{*}b_1)(b_1^{-1}b_{j-\lfloor j\rfloor+1}\oline{*}b_1)
=(a_{\lfloor j\rfloor}b_1b_1^{-1}\uline{*}b_1)(b_1^{-1}b_{j-\lfloor j\rfloor+1}\oline{*}b_1) \\
&=(b_1^{-1}a_{\lfloor j\rfloor}b_1\oline{*}b_1)(b_1^{-1}b_{j-\lfloor j\rfloor+1}\oline{*}b_1)
=b_1^{-1}a_{\lfloor j\rfloor}b_{j-\lfloor j\rfloor+1}\oline{*}b_1.
\end{align*}

For the fourth term \eqref{eq4}, we have
\begin{align*}
\eqref{eq4}
&=\sum_{i=1}^t(-1)^{s+i}\left(
\sum_{\mu\in M(s,s+t-1)}(-1)^{\sum_{k=1}^s(\mu(k) -k)}
\langle\boldsymbol{d}\rangle\right) \\
&=\sum_{i=1}^t(-1)^i\sum_{\mu\in M(s, s+t-1)}(-1)^{\sum_{k=1}^s(\mu(k)-k+1)}
\langle\boldsymbol{d}\rangle,
\end{align*}
where
\[ d_j=\begin{cases}
a_{\lfloor j\rfloor}b_{j-\lfloor j\rfloor} & \text{if $j-\lfloor j\rfloor\leq i-1$,} \\
a_{\lfloor j\rfloor}b_{j-\lfloor j\rfloor+1} & \text{if $j-\lfloor j\rfloor\geq i$.}
\end{cases} \]

On the other hand, by the definition,
\begin{align}
\partial(\langle\langle\boldsymbol{a}\rangle\langle\boldsymbol{b}\rangle\rangle)
&=\partial\left(
\sum_{\mu\in M(s,s+t)}(-1)^{\sum_{k=1}^s(\mu(k)-k)}\langle\boldsymbol{c}\rangle
\right) \nonumber\\
&=\sum_{\mu\in M(s,s+t)}(-1)^{\sum_{k=1}^s(\mu(k) -k)}\left(
\uline{*}c_1\langle c_1^{-1}\boldsymbol{c}_{\widehat{1}}\oline{*}c_1\rangle\oline{*}c_1
+\sum_{i=1}^{s+t}(-1)^i\langle\boldsymbol{c}_{\widehat{i}}\rangle
\right) \nonumber \\
&=\sum_{\begin{subarray}{c} \mu\in M(s,s+t), \\ \mu(1)=1 \end{subarray}}
(-1)^{\sum_{k=1}^s(\mu(k)-k)}
\uline{*}a_1\langle a_1^{-1}\boldsymbol{c}_{\widehat{1}}\oline{*}a_1\rangle\oline{*}a_1 \label{eq1'} \\
&\hspace{1em}+\sum_{\begin{subarray}{c} \mu\in M(s,s+t), \\ \mu(1)>1 \end{subarray}}
(-1)^{\sum_{k=1}^s(\mu(k)-k)}
\uline{*}b_1\langle b_1^{-1}\boldsymbol{c}_{\widehat{1}}\oline{*}b_1\rangle\oline{*}b_1 \label{eq3'} \\
&\hspace{1em}+\sum_{\mu\in M(s,s+t)}
(-1)^{\sum_{k=1}^s(\mu(k)-k)}
\sum_{i=1}^{s+t}(-1)^i\langle\boldsymbol{c}_{\widehat{i}}\rangle. \label{eq24'}
\end{align}
where $\langle\boldsymbol{c}\rangle=\langle c_1,\ldots,c_{s+t}\rangle$ and $c_j=a_{\lfloor j\rfloor}b_{j-\lfloor j\rfloor}$.

For the first term \eqref{eq1'}, we define $\widetilde{\mu}\in M(s,s+t-1,\widehat{1})$ for $\mu\in M(s,s+t)$ with $\mu(1)=1$ by $\widetilde{\mu}(k)=\mu(k)-1$, and this induces a one-to-one correspondence from $\{\mu\in M(s,s+t)\,|\,\mu(1)=1\}$ to $M(s,s+t-1,\widehat{1})$.
It is noted that
\[ \lfloor j+1;\mu\rfloor=\begin{cases}
1 & \text{if $1\leq j<\widetilde{\mu}(2)$,} \\
\lfloor j;\widetilde{\mu}\rfloor & \text{otherwise.}
\end{cases} \]
Then we have
\[ \eqref{eq1'}
=\sum_{\widetilde{\mu}\in M(s,s+t-1,\widehat{1})}
(-1)^{\sum_{k=2}^s(\widetilde{\mu}(k)-k+1)}
\uline{*}a_1\langle\boldsymbol{d}\rangle\oline{*}a_1, \]
where
$\langle\boldsymbol{d}\rangle=\langle d_1,\ldots,d_{s+t-1}\rangle$ and
\[ d_j=a_1^{-1}c_{j+1}\oline{*}a_1
=\begin{cases}
b_j\oline{*}a_1 & \text{if $1\leq j<\widetilde{\mu}(2)$,} \\
a_1^{-1}a_{\lfloor j;\widetilde{\mu}\rfloor}b_{j-\lfloor j;\widetilde{\mu}\rfloor+1}\oline{*}a_1 & \text{otherwise.}
\end{cases} \]
Thus we have $\eqref{eq1}=\eqref{eq1'}$.

For the second term \eqref{eq3'}, we define $\widetilde{\mu}\in M(s,s+t-1)$ for $\mu\in M(s,s+t)$ with $\mu(1)>1$ by $\widetilde{\mu}(k)=\mu(k)-1$, and this induces a one-to-one correspondence from $\{\mu\in M(s,s+t)\,|\,\mu(1)>1\}$ to $M(s,s+t-1)$.
It is noted that $\lfloor j+1;\mu\rfloor=\lfloor j;\widetilde{\mu}\rfloor$.
Then we have
\[ \eqref{eq3'}=
\sum_{\widetilde{\mu}\in M(s, s+t-1)}
(-1)^{\sum_{k=1}^s(\widetilde{\mu}(k)-k+1)}
\uline{*}b_1\langle\boldsymbol{d}\rangle\oline{*}b_1, \]
where $\langle\boldsymbol{d}\rangle=\langle d_1,\ldots,d_{s+t-1}\rangle$ and
\[ d_j=b_1^{-1}c_{j+1}\oline{*}b_1
=b_1^{-1}a_{\lfloor j;\widetilde{\mu}\rfloor}b_{j-\lfloor j;\widetilde{\mu}\rfloor+1}\oline{*}b_1. \]
Thus we have $\eqref{eq3}=\eqref{eq3'}$.

For the last term \eqref{eq24'}, we have
\begin{align}
\eqref{eq24'}
&=\sum_{i=1}^{s+t}(-1)^i\left(
\sum_{\begin{subarray}{c} \mu \in M(s, s+t), \\ i\in\operatorname{Im}\mu \end{subarray}}
(-1)^{\sum_{k=1}^s(\mu(k)-k)}\langle\boldsymbol{c}_{\widehat{i}}\rangle
+\sum_{\begin{subarray}{c} \mu \in M(s,s+t), \\ i\not\in\operatorname{Im}\mu \end{subarray}}
(-1)^{\sum_{k=1}^s(\mu(k)-k)}\langle\boldsymbol{c}_{\widehat{i}}\rangle
\right) \nonumber\\
&=\sum_{i=1}^{s+t}\sum_{\begin{subarray}{c}
\mu\in M(s,s+t), \\ i\in\operatorname{Im}\mu,~i+1\not\in\overline{\operatorname{Im}\mu}
\end{subarray}}
(-1)^i(-1)^{\sum_{k=1}^s(\mu(k)-k)}
\langle\boldsymbol{c}_{\widehat{i}}\rangle \label{eq24-1} \\
&\hspace{1em}+\sum_{i=1  }^{s+t-1}\sum_{\begin{subarray}{c}
\mu\in M(s,s+t), \\ i\in\operatorname{Im}\mu,~i+1\in\overline{\operatorname{Im}\mu}
\end{subarray}}
(-1)^i(-1)^{\sum_{k=1}^s(\mu(k)-k)}
\langle\boldsymbol{c}_{\widehat{i}}\rangle \label{eq24-2} \\
&\hspace{1em}+\sum_{i=1}^{s+t}\sum_{\begin{subarray}{c}
\mu\in M(s,s+t), \\ i\in\overline{\operatorname{Im}\mu},~i+1\not\in\operatorname{Im}\mu
\end{subarray}}
(-1)^i(-1)^{\sum_{k=1}^s(\mu(k)-k)}
\langle\boldsymbol{c}_{\widehat{i}}\rangle \label{eq24-3} \\
&\hspace{1em}+\sum_{i=1}^{s+t-1}\sum_{\begin{subarray}{c}
\mu\in M(s,s+t), \\ i\in \overline{\operatorname{Im}\mu},~i+1\in\operatorname{Im}\mu
\end{subarray}}
(-1)^i(-1)^{\sum_{k=1}^s(\mu(k)-k)}
\langle\boldsymbol{c}_{\widehat{i}}\rangle, \label{eq24-4}
\end{align}
where $\langle\boldsymbol{c}\rangle=\langle c_1,\ldots,c_{s+t}\rangle$ and $\overline{\operatorname{Im}\mu}=\{1,\ldots , s+t\} \setminus \operatorname{Im}\mu$.

For \eqref{eq24-1}, we define $\widetilde{\mu}\in M(s,s+t-1;\widehat{\ell})$ 
for $\mu\in M(s,s+t)$ with $\mu(\ell)=i$ and $i+1\not\in\overline{\operatorname{Im}\mu}$ by
\[ \widetilde{\mu}(k)=\begin{cases}
\mu(k) & \text{if $k<\ell$,} \\
\mu(k)-1 & \text{if $k>\ell$,}
\end{cases} \]
and this induces a one-to-one correspondence from
\begin{align*}
&\{\mu\in M(s,s+t)\,|\,i\in\operatorname{Im}\mu, i+1\not\in\overline{\operatorname{Im}\mu}\} \\
&\left(=\left(\bigsqcup_{\ell=1}^{s-1}\{\mu\in M(s,s+t)\,|\,\text{$\mu(\ell)=i$, $\mu(\ell+1)=i+1$}\}\right)\cup \{\mu\in M(s,s+t)\,|\,\mu(s)=s+t\}\right) 
\end{align*}
to $\bigsqcup_{\ell=1}^{s-1}\{\widetilde{\mu}\in M(s,s+t-1;\widehat{\ell})\,|\,\widetilde{\mu}(\ell+1)=i\}\cup M(s,s+t-1;\widehat{s})$.
It is noted that
\begin{align*}
\lfloor j;\widetilde{\mu}\rfloor
&=\begin{cases}
\lfloor j;\mu\rfloor & \text{if $j<i$,} \\
\lfloor j+1;\mu\rfloor & \text{if $j\geq i$,}
\end{cases} \\
&=\begin{cases}
\lfloor j;\mu\rfloor & \text{if $\lfloor j;\mu\rfloor\leq\ell-1$,} \\
\lfloor j+1;\mu\rfloor & \text{if $\lfloor j+1;\mu\rfloor\geq\ell+1$,}
\end{cases} \\
&=\begin{cases}
\lfloor j;\mu\rfloor & \text{if $\lfloor j;\widetilde{\mu}\rfloor\leq\ell-1$,} \\
\lfloor j+1;\mu\rfloor & \text{if $\lfloor j;\widetilde{\mu}\rfloor\geq\ell+1$.}
\end{cases}
\end{align*}
Then we have
\begin{alignat*}{4}
\eqref{eq24-1}
&=
&\sum_{i=1}^{s+t}\sum_{\ell=1}^s
&\sum_{\begin{subarray}{c}
\widetilde{\mu}\in M(s,s+t-1;\widehat{\ell}), \\ \text{$\widetilde{\mu}(\ell+1)=i$ if $\ell\not=s$}
\end{subarray}}
&(-1)^{\widetilde{\mu}(\ell+1)}
&(-1)^{\sum_{k=1}^{\ell-1}(\widetilde{\mu}(k)-k)+(\widetilde{\mu}(\ell+1)-\ell)
+\sum_{k=\ell+1}^s(\widetilde{\mu}(k)-k+1)}
&\langle\boldsymbol{d}\rangle \\
&=
&-\sum_{i=1}^{s+t}\sum_{\ell=1}^s
&\sum_{\begin{subarray}{c}
\widetilde{\mu}\in M(s,s+t-1;\widehat{\ell}), \\ \text{$\widetilde{\mu}(\ell+1)=i$ if $\ell\not=s$}
\end{subarray}}
&&(-1)^{\sum_{k=1}^{\ell-1}(\widetilde{\mu}(k)-k+1)
+\sum_{k=\ell+1}^s(\widetilde{\mu}(k)-k+1)}
&\langle\boldsymbol{d}\rangle \\
&=
&-\sum_{\ell=1}^s
&\sum_{\widetilde{\mu}\in M(s,s+t-1;\widehat{\ell})}
&&(-1)^{\sum_{k=1}^{\ell-1}(\widetilde{\mu}(k)-k+1)
+\sum_{k=\ell+1}^s(\widetilde{\mu}(k)-k+1)}
&\langle\boldsymbol{d}\rangle,
\end{alignat*}
where $\langle\boldsymbol{d}\rangle=\langle d_1,\ldots,d_{s+t-1}\rangle$ and
\begin{align*}
d_j
&=\begin{cases}
c_j & \text{if $j<i$,} \\
c_{j+1} & \text{if $j\geq i$,}
\end{cases} \\
&=\begin{cases}
a_{\lfloor j;\widetilde{\mu}\rfloor}b_{j-\lfloor j;\widetilde{\mu}\rfloor}
& \text{if $\lfloor j;\widetilde{\mu}\rfloor\leq\ell-1$,} \\
a_{\lfloor j;\widetilde{\mu}\rfloor}b_{j-\lfloor j;\widetilde{\mu}\rfloor+1}
& \text{if $\lfloor j;\widetilde{\mu}\rfloor\geq\ell+1$.}
\end{cases}
\end{align*}
Thus we have $\eqref{eq24-1}=\eqref{eq2}$.

For \eqref{eq24-3}, we define $\widetilde{\mu}\in M(s,s+t-1)$ with $i\not\in\operatorname{Im}\widetilde{\mu}$ for $\mu\in M(s,s+t)$ with $i\in\overline{\operatorname{Im}\mu}$, $i+1\not\in\operatorname{Im}\mu$ by
\[ \widetilde{\mu}(k)=\begin{cases}
\mu(k) & \text{if $\mu(k)<i$,} \\
\mu(k)-1 & \text{if $\mu(k)>i$,}
\end{cases} \]
and this induces a one-to-one correspondence from
\[ \{\mu\in M(s,s+t)\,|\,\text{$i\in\overline{\operatorname{Im}\mu}$, $i+1\not\in \operatorname{Im}\mu$, 
$\ell=i-\lfloor i;\mu\rfloor$}\} \]
to $\{\widetilde{\mu}\in M(s,s+t-1)\,|\,\text{$i\not\in\operatorname{Im}\widetilde{\mu}$, $\ell=i-\lfloor i;\widetilde{\mu}\rfloor$}\}$ for each $\ell$.
It is noted that
\begin{align*}
\lfloor j;\widetilde{\mu}\rfloor
&=\begin{cases}
\lfloor j;\mu\rfloor & \text{if $j<i$,} \\
\lfloor j+1;\mu\rfloor & \text{if $j\geq i$,}
\end{cases} \\
&=\begin{cases}
\lfloor j;\mu\rfloor & \text{if $j-\lfloor j;\mu\rfloor\leq\ell-1$,} \\
\lfloor j+1;\mu\rfloor & \text{if $j+1-\lfloor j+1;\mu\rfloor\geq\ell+1$,}
\end{cases} \\
&=\begin{cases}
\lfloor j;\mu\rfloor & \text{if $j-\lfloor j;\widetilde{\mu}\rfloor\leq\ell-1$,} \\
\lfloor j+1;\mu\rfloor & \text{if $j-\lfloor j;\widetilde{\mu}\rfloor\geq\ell$.}
\end{cases}
\end{align*}
Then we have
\begin{alignat*}{4}
\eqref{eq24-3}
&=
&\sum_{i=1}^{s+t}\sum_{\ell=1}^t
&&\sum_{\begin{subarray}{c}
\widetilde{\mu}\in M(s,s+t-1), \\
i\not\in\operatorname{Im}\widetilde{\mu},\ell=i-\lfloor i;\widetilde{\mu}\rfloor
\end{subarray}}
&(-1)^i
&(-1)^{\sum_{\mu(k)<i}(\widetilde{\mu}(k)-k)+\sum_{\mu(k)>i}(\widetilde{\mu}(k)-k+1)}
&\langle\boldsymbol{d}\rangle \\
&=
&\sum_{i=1}^{s+t}\sum_{\ell=1}^t
&(-1)^\ell
&\sum_{\begin{subarray}{c}
\widetilde{\mu} \in M(s,s+t-1), \\ i\not\in\operatorname{Im}\widetilde{\mu},\ell=i-\lfloor i;\widetilde{\mu}\rfloor
\end{subarray}}
&&(-1)^{\sum_{\mu(k)<i}(\mu(k)-k+1)+\sum_{\mu(k)>i}(\mu(k)-k+1)}
&\langle\boldsymbol{d}\rangle \\
&=
&\sum_{i=1}^{s+t}\sum_{\ell=1}^t
&(-1)^\ell
&\sum_{\begin{subarray}{c}
\widetilde{\mu}\in M(s,s+t-1), \\
i\not\in\operatorname{Im}\widetilde{\mu},\ell=i-\lfloor i;\widetilde{\mu}\rfloor
\end{subarray}}
&&(-1)^{\sum_{k=1}^s(\mu(k)-k+1)}
&\langle\boldsymbol{d}\rangle \\
&=
&\sum_{\ell=1}^t
&(-1)^\ell
&\sum_{\widetilde{\mu}\in M(s,s+t-1)}
&&(-1)^{\sum_{k=1}^s(\mu(k)-k+1)}
&\langle\boldsymbol{d}\rangle,
\end{alignat*}
where $\langle\boldsymbol{d}\rangle=\langle d_1,\ldots,d_{s+t-1}\rangle$ and
\begin{align*}
d_j
&=\begin{cases}
c_j & \text{if $j<i$,} \\
c_{j+1} & \text{if $j\geq i$,}
\end{cases} \\
&=\begin{cases}
a_{\lfloor j;\widetilde{\mu}\rfloor}b_{j-\lfloor j;\widetilde{\mu}\rfloor}
& \text{if $j-\lfloor j;\widetilde{\mu}\rfloor\leq\ell-1$,} \\
a_{\lfloor j;\widetilde{\mu}\rfloor}b_{j-\lfloor j;\widetilde{\mu}\rfloor+1}
& \text{if $j-\lfloor j;\widetilde{\mu}\rfloor\geq\ell$.}
\end{cases}
\end{align*}
Thus we have $\eqref{eq24-3}=\eqref{eq4}$.

For \eqref{eq24-2}, we define $\widetilde{\mu}\in M(s,s+t-1)$ with $\widetilde{\mu}(\ell)=i$ for $\mu\in M(s,s+t)$ with $\mu(\ell)=i$ and $i+1\in\overline{\operatorname{Im}\mu}$ by
\[ \widetilde{\mu}(k)=\begin{cases}
\mu(k) & \text{if $k\leq\ell$,} \\
\mu(k)-1 & \text{if $k>\ell$,}
\end{cases} \]
and this induces a one-to-one correspondence from
\begin{align*}
&\{\mu\in M(s,s+t)\,|\,\text{$i\in\operatorname{Im}\mu$, $i+1\in\overline{\operatorname{Im}\mu}$}\} \\
&\left(=\bigsqcup_{\ell=1}^s\{\mu\in M(s, s+t)\,|\,\text{$\mu(\ell)=i$, $i+1\in\overline{\operatorname{Im}\mu}$}\}\right)
\end{align*}
to $\bigsqcup_{\ell=1}^s\{\widetilde{\mu}\in M(s,s+t-1)\,|\,\widetilde{\mu}(\ell)=i\}$.
It is noted that
\begin{align*}
\lfloor j;\widetilde{\mu}\rfloor
&=\begin{cases}
\lfloor j;\mu\rfloor & \text{if $j<i$,} \\
\lfloor j+1;\mu\rfloor & \text{if $j\geq i$,}
\end{cases} \\
&=\begin{cases}
\lfloor j;\mu\rfloor & \text{if $\lfloor j;\mu\rfloor\leq\ell-1$,} \\
\lfloor j+1;\mu\rfloor & \text{if $\lfloor j+1;\mu\rfloor\geq\ell$,}
\end{cases} \\
&=\begin{cases}
\lfloor j;\mu\rfloor & \text{if $\lfloor j;\widetilde{\mu}\rfloor\leq\ell-1$,} \\
\lfloor j+1;\mu\rfloor & \text{if $\lfloor j;\widetilde{\mu}\rfloor\geq\ell$.}
\end{cases}
\end{align*}
Then we have
\begin{alignat*}{4}
\eqref{eq24-2}
&=
&\sum_{i=1}^{s+t-1}\sum_{\ell=1}^s
&&\sum_{\begin{subarray}{c}
\widetilde{\mu}\in M(s,s+t-1), \\ \widetilde{\mu}(\ell)=i
\end{subarray}}
&(-1)^i
&(-1)^{\sum_{k=1}^\ell(\widetilde{\mu}(k)-k)+\sum_{k=\ell+1}^s(\widetilde{\mu}(k)-k+1)}
&\langle\boldsymbol{d}\rangle \\
&=
&\sum_{i=1}^{s+t-1}\sum_{\ell=1}^s
&&\sum_{\begin{subarray}{c}
\widetilde{\mu}\in M(s,s+t-1), \\ \widetilde{\mu}(\ell)=i
\end{subarray}}
&(-1)^\ell
&(-1)^{\sum_{k=1}^{\ell-1}(\widetilde{\mu}(k)-k)+\sum_{k=\ell+1}^s(\widetilde{\mu}(k)-k+1)}
&\langle\boldsymbol{d}\rangle \\
&=
&-\sum_{i=1}^{s+t-1}\sum_{\ell=1}^s
&&\sum_{\begin{subarray}{c}
\widetilde{\mu}\in M(s,s+t-1), \\ \widetilde{\mu}(\ell)=i
\end{subarray}}
&&(-1)^{\sum_{k=1}^{\ell-1}(\widetilde{\mu}(k)-k+1)+\sum_{k=\ell+1}^s(\widetilde{\mu}(k)-k+1)}
&\langle\boldsymbol{d}\rangle \\
&=
&-\sum_{\ell=1}^s
&&\sum_{\widetilde{\mu}\in M(s,s+t-1)}
&&(-1)^{\sum_{k=1}^{\ell-1}(\widetilde{\mu}(k)-k+1)+\sum_{k=\ell+1}^s(\widetilde{\mu}(k)-k+1)}
&\langle\boldsymbol{d}\rangle,
\end{alignat*}
where $\langle\boldsymbol{d}\rangle=\langle d_1,\ldots,d_{s+t-1}\rangle$ and
\begin{align*}
d_j
&=\begin{cases}
c_j & \text{if $j<i$,} \\
c_{j+1} & \text{if $j\geq i$.}
\end{cases} \\
&=\begin{cases}
a_{\lfloor j;\widetilde{\mu}\rfloor}b_{j-\lfloor j;\widetilde{\mu}\rfloor}
& \text{if $\lfloor j;\widetilde{\mu}\rfloor\leq\ell-1$,} \\
a_{\lfloor j;\widetilde{\mu}\rfloor}b_{j-\lfloor j;\widetilde{\mu}\rfloor+1}
& \text{if $\lfloor j;\widetilde{\mu}\rfloor\geq\ell$.}
\end{cases}
\end{align*}

For \eqref{eq24-4}, we define $\widetilde{\mu}\in M(s,s+t-1)$ with $\widetilde{\mu}(\ell)=i$ for $\mu\in M(s,s+t)$ with $i\in\overline{\operatorname{Im}\mu}$ and $\mu(\ell)=i+1$ by
\[ \widetilde{\mu}(k)=\begin{cases}
\mu(k) & \text{if $\mu(k)<i$,} \\
\mu(k)-1 & \text{if $\mu(k)>i$,}
\end{cases} \]
and this induces a one-to-one correspondence from
\[ \{\mu\in M(s, s+t)\,|\,\text{$i\in\overline{\operatorname{Im}\mu}$, $\mu(\ell)=i+1$}\} \]
to $\{\widetilde{\mu}\in M(s,s+t-1)\,|\,\widetilde{\mu}(\ell)=i\}$ for each $\ell$.
It is noted that
\begin{align*}
\lfloor j;\widetilde{\mu}\rfloor
&=\begin{cases}
\lfloor j;\mu\rfloor & \text{if $j<i$,} \\
\lfloor j+1;\mu\rfloor & \text{if $j\geq i$,}
\end{cases} \\
&=\begin{cases}
\lfloor j;\mu\rfloor & \text{if $\lfloor j;\mu\rfloor\leq\ell-1$, $j\not=i$,} \\
\ell & \text{if $j=i$,} \\
\lfloor j+1;\mu\rfloor & \text{if $\lfloor j+1;\mu\rfloor\geq\ell$,}
\end{cases} \\
&=\begin{cases}
\lfloor j;\mu\rfloor & \text{if $\lfloor j;\widetilde{\mu}\rfloor\leq\ell-1$,} \\
\lfloor j+1;\mu\rfloor & \text{if $(\lfloor j;\widetilde{\mu}\rfloor\geq\ell$.}
\end{cases}
\end{align*}
Then we have
\begin{alignat*}{4}
\eqref{eq24-4}
&=
&\sum_{i=1}^{s+t-1}\sum_{\ell=1}^s
&&\sum_{\begin{subarray}{c}
\widetilde{\mu}\in M(s,s+t-1), \\ \widetilde{\mu}(\ell)=i
\end{subarray}}
&(-1)^i
&(-1)^{\sum_{k=1}^{\ell-1}(\widetilde{\mu}(k)-k)+\sum_{k=\ell}^s(\widetilde{\mu}(k)-k+1)}
&\langle\boldsymbol{d}\rangle \\
&=
&\sum_{i=1}^{s+t-1}\sum_{\ell=1}^s
&&\sum_{\begin{subarray}{c}
\widetilde{\mu}\in M(s,s+t-1), \\ \widetilde{\mu}(\ell)=i
\end{subarray}}
&(-1)^{-\ell+1}
&(-1)^{\sum_{k=1}^{\ell-1}(\widetilde{\mu}(k)-k)+\sum_{k=\ell+1}^s(\widetilde{\mu}(k)-k+1)}
&\langle\boldsymbol{d}\rangle \\
&=
&\sum_{i=1}^{s+t-1}\sum_{\ell=1}^s
&&\sum_{\begin{subarray}{c}
\widetilde{\mu}\in M(s,s+t-1), \\ \widetilde{\mu}(\ell)=i
\end{subarray}}
&&(-1)^{\sum_{k=1}^{\ell-1}(\widetilde{\mu}(k)-k+1)+\sum_{k=\ell+1}^s(\widetilde{\mu}(k)-k+1)}
&\langle\boldsymbol{d}\rangle \\
&=
&\sum_{\ell=1}^s
&&\sum_{\widetilde{\mu} \in M(s,s+t-1)}
&&(-1)^{\sum_{k=1}^{\ell-1}(\widetilde{\mu}(k)-k+1)+\sum_{k=\ell+1}^s(\widetilde{\mu}(k)-k+1)}
&\langle\boldsymbol{d}\rangle,
\end{alignat*}
where $\langle\boldsymbol{d}\rangle=\langle d_1,\ldots,d_{s+t-1}\rangle$ and
\begin{align*}
d_j
&=\begin{cases}
c_j & \text{if $j<i$,} \\
c_{j+1} & \text{if $j\geq i$,}
\end{cases} \\
&=\begin{cases}
a_{\lfloor j;\widetilde{\mu}\rfloor}b_{j-\lfloor j;\widetilde{\mu}\rfloor}
& \text{if $\lfloor j;\widetilde{\mu}\rfloor\leq\ell-1$,} \\
a_{\lfloor j;\widetilde{\mu}\rfloor}b_{j-\lfloor j;\widetilde{\mu}\rfloor+1}
& \text{if $\lfloor j;\widetilde{\mu}\rfloor\geq\ell$.}
\end{cases}
\end{align*}
This implies that $\eqref{eq24-4}=-\eqref{eq24-2}$.

As the consequence, we have
\begin{align*}
\partial(\langle\langle\boldsymbol{a}\rangle\langle\boldsymbol{b}\rangle\rangle)
&=\eqref{eq1'}+\eqref{eq3'}+\eqref{eq24'} \\
&=\eqref{eq1'}+\eqref{eq3'}+\eqref{eq24-1}+\eqref{eq24-2}+\eqref{eq24-3}+\eqref{eq24-4} \\
&=\eqref{eq1}+\eqref{eq3}+\eqref{eq2}+\eqref{eq24-2}+\eqref{eq4}-\eqref{eq24-2} \\
&=\eqref{eq1}+\eqref{eq2}+\eqref{eq3}+\eqref{eq4}.
\end{align*}
\end{proof}

Put $C_n(X)_Y=P_n(X)_Y/D_n(X)_Y$.
For an abelian group $A$, we define the chain and cochain complexes by
\begin{align*}
&C_n(X;A)_Y=C_n (X)_Y \otimes A,
&& \partial_n=\partial_n\otimes\mathrm{id} \text{ and} \\
&C^n(X;A)_Y=\operatorname{Hom}(C_n(X)_Y,A),
&& \delta^n(f)=f\circ\partial_{n+1}.
\end{align*}
Let $C_*(X;A)_Y=(C_n(X;A)_Y,\partial_n)$ and $C^*(X;A)_Y=(C^n(X;A)_Y,\delta^n)$.

\begin{definition}
The \textit{$n$th homology group} $H_n(X;A)_Y$ and \textit{$n$th cohomology group} $H^n(X;A)_Y$ of the multiple conjugation biquandle $X$ and the $X$-set $Y$ with coefficient group $A$ are defined by
\begin{align*}
&H_n(X;A)_Y=H_n(C_*(X;A)_Y)
&&\text{and}
&&H^n(X;A)_Y=H^n(C^*(X;A)_Y).
\end{align*}
Note that we omit the coefficient group $A$ if $A=\mathbb{Z}$ as usual, and we omit $Y$ if $Y$ is a trivial $X$-set.
\end{definition}

\section{Cocycles of multiple conjugation biquandles} \label{sect:cocycle}
Let $X$ be a biquandle.
As in Example~\ref{prop:Z-family}, we can regard $X$ as a $\mathbb{Z}$-family of biquandles with the parallel operations $\uline{*}^{[n]}$ and $\oline{*}^{[n]}$.
Then as in Proposition~\ref{prop:G-family2MCB}, we have the associated multiple conjugation biquandle $X\times\mathbb{Z}=\bigsqcup_{x\in X}\{x\}\times\mathbb{Z}$.
For an $X$-set $Y$, we define $y*^{[n]}x\in Y$ for $y\in Y$, $x\in X$, $n\in\mathbb{Z}$ by the following rule:
\begin{align*}
&y*^{[0]}x=y,
&&y*^{[1]}x=y*x,
&&y*^{[i+j]}x=(y*^{[i]}x)*^{[j]}(x\uline{*}^{[i]}x).
\end{align*}
In particular, when $n$ is a positive integer,
\[ y*^{[n]}x=(\cdots(((y*x)*(x\uline{*}^{[1]}x))*(x\uline{*}^{[2]}x)*\cdots)*(x\uline{*}^{[n-1]}x). \]
Then an $X$-set $Y$ is regarded as an $(X\times\mathbb{Z})$-set by $y*(x,n)=y*^{[n]}x$.
Conversely an $(X\times\mathbb{Z})$-set $Y$ is regarded as an $X$-set by $y*x=y*(x,0)$.

\begin{definition}
The \textit{type} of a biquandle $X$ equipped with an $X$-set $Y$ is defined by
\[
\operatorname{type}X_Y
=\min \left\{ n>0  ~\bigg|~
\begin{array}{l}\text{$a\uline{*}^{[n]}b=a=a\oline{*}^{[n]}b$ ($\forall a,b\in X$)}, \\
\text{$y*^{[n]}a=y$ ($\forall y\in Y$, $\forall a\in X$)}
\end{array}
 \right\}.
\]
\end{definition}

Let $X$ be a biquandle and $Y$ an $X$-set.
Again, as in Example~\ref{prop:Z-family} and Proposition~\ref{prop:G-family2MCB}, we can regard $X$ as a $\mathbb{Z}_{\operatorname{type}X_Y}$-family of biquandles with the parallel operations $\uline{*}^{[n]}$ and $\oline{*}^{[n]}$, and we have the associated multiple conjugation biquandle $X\times\mathbb{Z}_{\operatorname{type}X_Y}=\bigsqcup_{x\in X}\{x\}\times\mathbb{Z}_{\operatorname{type}X_Y}$. Suppose that $\operatorname{type}X_Y$ is finite.

Let $A$ be an abelian group.

A {\it biquandle $n$-cocycle} $\theta\in C^2_{\rm BQ}(X;A)_Y$ (cf. \cite{CarterElhamdadiSaito04, LebedVendramin17}) is the linear extension $\theta: \mathbb Z(Y\times X^n) \to  A$ of a map  
$\theta: Y\times X^n \to  A$ satisfying the following cocycle conditions:
\begin{itemize}
\item  For any $x_1, \ldots , x_n \in X$ and $y\in Y$, if $x_{i}=x_{i+1}$ for some $i\in \{1,\ldots ,n-1\}$, then  $\theta(\langle y \rangle \langle x_1\rangle \cdots \langle x_n \rangle)=0$. 
\item For any $x_1, \ldots , x_n \in X$ and $y\in Y$, 
$$
\begin{array}{l}
\theta\bigg( ~\displaystyle \sum_{i=1}^n (-1)^i \Big\{ \langle y \rangle \langle x_1 \rangle \cdots \langle x_{i-1} \rangle \langle x_{i+1} \rangle\cdots  \langle x_n \rangle  \\
\hspace{2.3cm}- \langle y \rangle \langle x_1 \rangle \cdots \langle x_{i-1} \rangle\uline{*}x_i \oline{*}x_i \langle x_{i+1} \rangle\cdots  \langle x_n \rangle  \Big\} \bigg)
=0.
\end{array}
$$
\end{itemize}

Next proposition shows that when we have a biquandle $2$-cocycle $\theta\in C^2_{\rm BQ}(X;A)_Y$, we can construct a $2$-cocycle $\widetilde{\theta}$ of the associated multiple conjugation biquandle $X\times\mathbb{Z}_{\operatorname{type}X_Y}$. 

\begin{proposition}
For a biquandle $2$-cocycle $\theta\in C^2_{\rm BQ}(X;A)_Y$, we define
$\widetilde{\theta}:P_2(X\times\mathbb{Z}_{\operatorname{type}X_Y})_Y\to A$ by
\begin{align*}
\widetilde{\theta}(\langle y\rangle\langle a\rangle\langle b\rangle)
=\sum_{i=0}^{i_1-1}\sum_{j=0}^{i_2-1}\theta(
\langle y*^{[i]}x_1\rangle\langle x_1\oline{*}^{[i]}x_1\rangle\langle x_2\oline{*}^{[i]}x_1\rangle\uline{*}^{[j]}(x_2\oline{*}^{[i]}x_1))
\end{align*}
for $a =(x_1,i_1),b=(x_2,i_2)\in X\times\mathbb{Z}_{\operatorname{type}X_Y}$, and by
\begin{align}
\widetilde{\theta}(\langle y\rangle\langle a,b\rangle)=0 \label{eq:7-2cycS2}
\end{align}
for $a,b\in X\times\mathbb{Z}_{\operatorname{type}X_Y}$.
Suppose that, for any  $x_1,x_2 \in X$ and $y \in Y$,
\begin{align*}
& \sum_{i=0}^{\operatorname{type}X_Y-1}\theta(
\langle y*^{[i]}x_1\rangle\langle x_1\oline{*}^{[i]}x_1\rangle\langle x_2\oline{*}^{[i]}x_1\rangle)=0, \\
& \sum_{j=0}^{\operatorname{type}X_Y-1}\theta(
\langle y*^{[j]}x_2\rangle\langle x_1\uline{*}^{[j]}x_2\rangle\langle x_2\uline{*}^{[j]}x_2\rangle)=0.
\end{align*}
Then $\widetilde{\theta}\in C^2(X\times\mathbb{Z}_{\operatorname{type}X_Y};A)_Y$, i.e., it is a $2$-cocycle of the associated multiple conjugation biquandle $X\times\mathbb{Z}_{\operatorname{type}X_Y}$.
\end{proposition}

\begin{proof}
The last two equalities imply the well-definedness of the map $\widetilde{\theta}$.
(It does not depend on the choice of positive integers $i_1$ and $i_2$ representing elements $i_1$ and $i_2$ of $\mathbb{Z}_{\operatorname{type}X_Y}$).
It is sufficient to verify that $\widetilde{\theta}$ vanishes on the following chains:
\begin{align}
& \partial_3(\langle y\rangle\langle a\rangle\langle b\rangle\langle c\rangle), \label{eq:7-2cycP1} \\
& \partial_3(\langle y\rangle\langle a\rangle\langle b,c\rangle),
\partial_3(\langle y\rangle\langle a,b\rangle\langle c\rangle), \label{eq:7-2cycP3} \\
& \partial_3(\langle y\rangle\langle a,b,c\rangle), \label{eq:7-2cycP4} \\
&\langle y\rangle\langle a\rangle\langle b\rangle
-\langle y\rangle\langle a,ab\rangle
+\langle y\rangle\langle b,ab\rangle \label{eq:7-2cycP5}
\end{align}
for all $y\in Y$, $a,b,c\in X\times\mathbb{Z}_{\operatorname{type}X_Y}$ whenever the multiplication is defined.

For chains in \eqref{eq:7-2cycP1}, it is seen that $\widetilde{\theta}$ vanishes using the fact that $\theta$ is a biquandle 2-cocycle.
On chains in 
\eqref{eq:7-2cycP3}, we see that $\widetilde{\theta}$ vanishes without using any assumption on $\theta$.
These are observed by considering the geometric meaning of parallel biquandle operations, see Figure~\ref{fig:parallel} (see also \cite[Proposition~6.2 and Section~7]{IshiiIwakiri12}).
By \eqref{eq:7-2cycS2}, we see that $\widetilde{\theta}$ vanishes on chains in \eqref{eq:7-2cycP4}.
Since $\theta(y,x,x)=0$ for any $x,y\in X$, we see that
$\widetilde{\theta}(\langle y\rangle\langle a\rangle\langle b\rangle)=0$
if $a=(x,i_1)$ and $b=(x,i_2)$ for some $x\in X$ and $i_1,i_2\in\mathbb{Z}_{\operatorname{type}X_Y}$.
Thus,
$\widetilde{\theta}(\langle y\rangle\langle a\rangle\langle b\rangle)=0$.
By \eqref{eq:7-2cycS2},
$\widetilde{\theta}(\langle y\rangle\langle a,ab\rangle)
=\widetilde{\theta}(\langle y\rangle\langle b,ab\rangle)=0$,
and we see that $\widetilde{\theta}$ vanishes on chains in \eqref{eq:7-2cycP5}.
\end{proof}

\begin{figure}
\begin{minipage}{70pt}
\begin{picture}(70,70)(-20,-15)
 \put(-20,-15){\framebox(70,70){}}
 \put(40,40){\vector(-1,-1){40}}
 \put(0,40){\line(1,-1){18}}
 \put(22,18){\vector(1,-1){18}}
 \put(0,43){\makebox(0,0)[b]{\footnotesize$(x_1,i_1)$}}
 \put(0,-3){\makebox(0,0)[t]{\footnotesize$(x_2,i_2)$}}
 \put(-5,14){\framebox(12,12){\normalsize$y$}}
\end{picture}
\end{minipage}
\begin{minipage}{270pt}
\begin{picture}(270,180)(-30,-20)
 \put(140,140){\vector(-1,-1){140}}
 \put(170,140){\vector(-1,-1){140}}
 \multiput(172.5,127.5)(5,-5){3}{\circle*{2}}
 \put(220,140){\vector(-1,-1){140}}
 \put(0,140){\line(1,-1){67.5}}
 \put(30,140){\line(1,-1){52.5}}
 \multiput(77.5,127.5)(-5,-5){3}{\circle*{2}}
 \put(80,140){\line(1,-1){27.5}}
 \put(72.5,67.5){\line(1,-1){10}}
 \put(87.5,82.5){\line(1,-1){10}}
 \put(112.5,107.5){\line(1,-1){10}}
 \put(87.5,52.5){\line(1,-1){20}}
 \put(102.5,67.5){\line(1,-1){20}}
 \put(127.5,92.5){\line(1,-1){20}}
 \put(112.5,27.5){\vector(1,-1){27.5}}
 \put(127.5,42.5){\vector(1,-1){42.5}}
 \put(152.5,67.5){\vector(1,-1){67.5}}
 \put(0,60){\framebox(20,20){$y$}}
 \put(-10,145){\makebox(0,0)[b]{\footnotesize$x_1\!\uline{*}^{[0]}\!x_1$}}
 \put(30,145){\makebox(0,0)[b]{\footnotesize$x_1\!\uline{*}^{[1]}\!x_1$}}
 \put(80,145){\makebox(0,0)[b]{\footnotesize$x_1\!\uline{*}^{[i_1-1]}\!x_1$}}
 \put(-10,-5){\makebox(0,0)[t]{\footnotesize$x_2\!\uline{*}^{[0]}\!x_2$}}
 \put(30,-5){\makebox(0,0)[t]{\footnotesize$x_2\!\uline{*}^{[1]}\!x_2$}}
 \put(80,-5){\makebox(0,0)[t]{\footnotesize$x_2\!\uline{*}^{[i_2-1]}\!x_2$}}
\end{picture}
\end{minipage}
\vspace{5mm}\\
\begin{minipage}{70pt}
\begin{picture}(70,70)(-10,-15)
 \put(-10,-15){\framebox(75,70){}}
 \put(20,20){\vector(0,-1){20}}
 \put(0,40){\vector(1,-1){20}}
 \put(40,40){\vector(-1,-1){20}}
 \put(0,43){\makebox(0,0)[b]{\footnotesize$(x,i)$}}
 \put(40,43){\makebox(0,0)[b]{\footnotesize$(x\!\uline{*}^{[i]}\!x,j\!-\!i)$}}
 \put(20,-3){\makebox(0,0)[t]{\footnotesize$(x,j)$}}
 \put(-5,0){\framebox(12,12){\normalsize$y$}}
\end{picture}
\end{minipage}
\begin{minipage}{270pt}
\begin{picture}(270,180)(-30,-20)
 \put(140,140){\line(-1,-1){25}}
 \put(170,140){\line(-1,-1){45}}
 \multiput(172.5,127.5)(5,-5){3}{\circle*{2}}
 \put(220,140){\line(-1,-1){65}}
 \put(0,140){\line(1,-1){65}}
 \put(30,140){\line(1,-1){45}}
 \multiput(77.5,127.5)(-5,-5){3}{\circle*{2}}
 \put(80,140){\line(1,-1){25}}
 \put(65,75){\vector(0,-1){75}}
 \put(75,95){\vector(0,-1){95}}
 \multiput(85,40)(5,0){3}{\circle*{2}}
 \put(105,115){\vector(0,-1){115}}
 \put(115,115){\vector(0,-1){115}}
 \put(125,95){\vector(0,-1){95}}
 \multiput(135,40)(5,0){3}{\circle*{2}}
 \put(155,75){\vector(0,-1){75}}
 \put(0,60){\framebox(20,20){$y$}}
 \put(-10,145){\makebox(0,0)[b]{\footnotesize$x\!\uline{*}^{[0]}\!x$}}
 \put(30,145){\makebox(0,0)[b]{\footnotesize$x\!\uline{*}^{[1]}\!x$}}
 \put(80,145){\makebox(0,0)[b]{\footnotesize$x\!\uline{*}^{[i-1]}\!x$}}
 \put(140,145){\makebox(0,0)[b]{\footnotesize$x\!\uline{*}^{[i]}\!x$}}
 \put(175,145){\makebox(0,0)[b]{\footnotesize$x\!\uline{*}^{[i+1]}\!x$}}
 \put(220,145){\makebox(0,0)[b]{\footnotesize$x\!\uline{*}^{[j-1]}\!x$}}
 \put(55,-5){\makebox(0,0)[t]{\footnotesize$x\!\uline{*}^{[0]}\!x$}}
 \put(85,-5){\makebox(0,0)[t]{\footnotesize$x\!\uline{*}^{[1]}\!x$}}
 \put(155,-5){\makebox(0,0)[t]{\footnotesize$x\!\uline{*}^{[j-1]}\!x$}}
\end{picture}
\end{minipage}
\caption{A geometric interpretation of the parallel biquandle operations.}
\label{fig:parallel}
\end{figure}
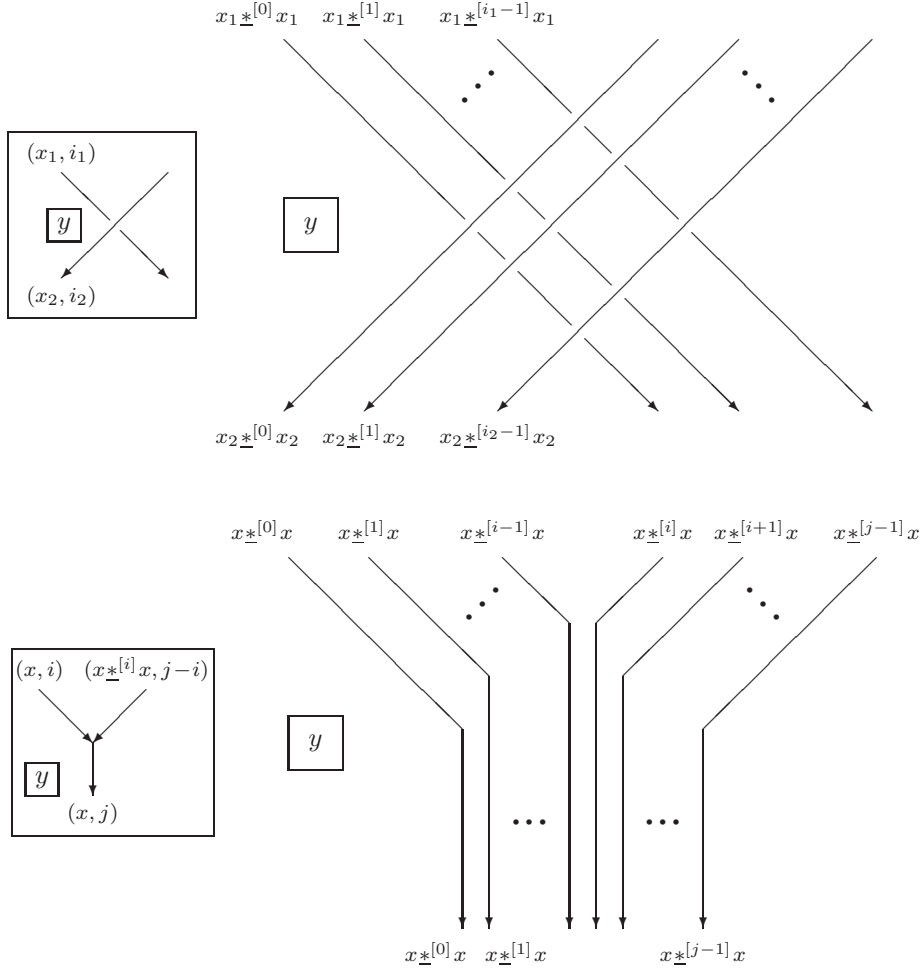

Next proposition shows that when we have a biquandle $3$-cocycle $\theta\in C^3_{\rm BQ}(X;A)_Y$, we can construct a $3$-cocycle $\widetilde{\theta}$ of the associated multiple conjugation biquandle $X\times\mathbb{Z}_{\operatorname{type}X_Y}$.
\begin{proposition}
For a biquandle $3$-cocycle $\theta\in C^3_{\rm BQ}(X;A)_Y$, we define
$\widetilde{\theta}:P_3(X\times\mathbb{Z}_{\operatorname{type}X_Y})_Y\to A$
by
\begin{align*}
\widetilde{\theta}(\langle y\rangle\langle a\rangle\langle b\rangle\langle c\rangle)
&=\sum_{i=0}^{i_1-1}\sum_{j=0}^{i_2-1}\sum_{k=0}^{i_3-1}\theta(
(\langle y*^{[i]}x_1\rangle\langle x_1\oline{*}^{[i]}x_1\rangle\langle x_2\oline{*}^{[i]}x_1\rangle\langle x_3\oline{*}^{[i]}x_1\rangle \\
&\hspace{7em}\uline{*}^{[j]}(x_2\oline{*}^{[i]}x_1))
\uline{*}^{[k]}((x_3\oline{*}^{[i]}x_1)\uline{*}^{[j]}(x_2\oline{*}^{[i]}x_1)))
\end{align*}
for $a=(x_1,i_1),b=(x_2,i_2),c=(x_3,i_3)\in X\times\mathbb{Z}_{\operatorname{type}X_Y}$, and by
\begin{align}
&\widetilde{\theta}(\langle y\rangle\langle a,b\rangle\langle c\rangle)=0,
&&\widetilde{\theta}(\langle y\rangle\langle a\rangle\langle b,c\rangle)=0,
&&\widetilde{\theta}(\langle y\rangle\langle a,b,c\rangle)=0 \label{eq:7-3cycS2}
\end{align}
for $a,b,c\in X\times\mathbb{Z}_{\operatorname{type}X_Y}$.
Suppose that, for any $x_1,x_2,x_3\in X$ and $y\in Y$,
\begin{align*}
&\sum_{i=0}^{\operatorname{type}X_Y-1}\theta(
\langle y*^{[i]}x_1\rangle\langle x_1\oline{*}^{[i]}x_1\rangle\langle x_2\oline{*}^{[i]}x_1\rangle\langle x_3\oline{*}^{[i]}x_1\rangle)=0, \\
&\sum_{j=0}^{\operatorname{type}X_Y-1}\theta(
\langle y*^{[j]}x_2\rangle\langle x_1\uline{*}^{[j]}x_2\rangle\langle x_2\uline{*}^{[j]}x_2\rangle\langle x_3\uline{*}^{[j]}x_2\rangle)=0, \\
&\sum_{k=0}^{\operatorname{type}X_Y-1}\theta(
\langle y*^{[k]}x_3\rangle\langle x_1\uline{*}^{[k]}x_3\rangle\langle x_2\uline{*}^{[k]}x_3\rangle\langle x_3\uline{*}^{[k]}x_3\rangle)=0.
\end{align*}
Then $\widetilde{\theta}\in C^3(X\times\mathbb{Z}_{\operatorname{type}X_Y};A)_Y$, i.e., it is a $3$-cocycle of the associated multiple conjugation biquandle $X\times\mathbb{Z}_{\operatorname{type}X_Y}$.
\end{proposition}

\begin{proof}
The last three equalities imply the well-definedness of the map $\widetilde{\theta}$.
(It does not depend on the choice of positive integers $i_1$, $i_2$ and $i_3$ representing elements $i_1$, $i_2$ and $i_3$ of $\mathbb{Z}_{\operatorname{type}X_Y}$).
It is sufficient to verify that $\widetilde{\theta}$ vanishes on the following chains:
\begin{align} 
&\partial_4(\langle y\rangle\langle a\rangle\langle b\rangle\langle c\rangle\langle d \rangle), \label{eq:7-3cycP1} \\
&\partial_4(\langle y\rangle\langle a\rangle\langle b\rangle\langle c,d\rangle),
\partial_4(\langle y\rangle\langle a\rangle\langle b,c\rangle \langle d\rangle),
\partial_4(\langle y\rangle\langle a,b\rangle\langle c\rangle \langle d\rangle),
\partial_4(\langle y\rangle\langle a,b\rangle\langle c,d\rangle), \label{eq:7-3cycP5} \\
&\partial_4(\langle y\rangle\langle a\rangle\langle b,c,d\rangle),
\partial_4(\langle y\rangle\langle a,b,c\rangle\langle d\rangle),
\partial_4(\langle y\rangle\langle a,b,c,d\rangle), \label{eq:7-3cycP8} \\
&\langle y\rangle\langle a\rangle\langle b\rangle\langle c\rangle
-\langle y\rangle\langle a,ab\rangle\langle c\rangle
+\langle y\rangle\langle b,ab\rangle\langle c\rangle, \label{eq:7-3cycP9} \\
&\langle y\rangle\langle a\rangle\langle b\rangle\langle c\rangle
-\langle y\rangle\langle a\rangle\langle b,bc\rangle
+\langle y\rangle\langle a\rangle\langle c,bc\rangle, \label{eq:7-3cycP10} \\
&\langle y\rangle\langle a,b\rangle\langle c\rangle
-\langle y\rangle\langle a,b,bc \rangle
+\langle y\rangle\langle a,ac,bc \rangle
-\langle y\rangle\langle c,ac,bc \rangle, \label{eq:7-3cycP11} \\
&\langle y\rangle\langle a\rangle\langle b,c\rangle
-\langle y\rangle\langle a,ab,ac \rangle
+\langle y\rangle\langle b,ab,ac \rangle
-\langle y\rangle\langle b,c,ac\rangle \label{eq:7-3cycP12}
\end{align}
for all $y\in Y$ and $a,b,c,d\in X\times\mathbb{Z}_{\operatorname{type}X_Y}$ whenever the multiplication is defined.

For chains in \eqref{eq:7-3cycP1}, it is seen that $\widetilde{\theta}$ vanishes using the fact that $\theta$ is a biquandle 3-cocycle.
On chains in 
\eqref{eq:7-3cycP5}, we see that $\widetilde{\theta}$ vanishes without using any assumption on $\theta$.
These are observed by considering the geometric meaning of parallel biquandle operations (cf.~\cite[Proposition~6.2 and Sections~7 and 9]{IshiiIwakiri12}).
By \eqref{eq:7-3cycS2}, we see that $\widetilde{\theta}$ vanishes on chains in 
\eqref{eq:7-3cycP8} and on chains in \eqref{eq:7-3cycP11} and \eqref{eq:7-3cycP12}.
Since $\theta(y,x,x,x')=\theta(y,x,x',x')=0$ and by \eqref{eq:7-3cycS2}, we see that $\widetilde{\theta}$ vanishes on chains in \eqref{eq:7-3cycP9} and \eqref{eq:7-3cycP10}.
\end{proof}

Let $G$ be a group with identity $e$, and $\varphi:G\to Z(G)$ a homomorphism, where $Z(G)$ is the center of $G$.
For a ring $R$, let $X$ be a right $R[G]$-module, where $R[G]$ is the group ring of $G$ over $R$.
We define binary operations $\uline{*}^g,\oline{*}^g:X\times X\to X$ by $x\uline{*}^gy=xg+y(\varphi(g)-g)$ and $x\oline{*}^gy=x\varphi(g)$.
Then, as shown in Example~\ref{exam:2.5}, $X$ is a $G$-family of biquandles, which we call a $G$-family of Alexander biquandles.
As in Proposition~\ref{prop:G-family2MCB}, we have the associated multiple conjugation biquandle $X\times G=\bigsqcup_{x\in X}\{x\}\times G$.
For an abelian group $A$, let $\lambda:G\to A$ be a homomorphism and $f:X^n\to A$ a $G$-invariant $A$-multilinear map, which is an $A$-multilinear map satisfying $f(x_1g,\ldots,x_ng)=f(x_1,\ldots,x_n)$ for any $g\in G$.

\begin{proposition} \label{prop:cocycle}
Let $n=2$.
\begin{itemize}
\item[(1)]
Define $\Phi_f:C_2(X\times G)\to A$ by
\begin{align*}
&\Phi_f(\langle(x_1,g_1)\rangle\langle(x_2,g_2)\rangle)
=\lambda(g_1)f(x_1-x_2,x_2(1-\varphi(g_2)g_2^{-1})), \text{ and} \\
&\Phi_f(\langle(x_1,g_1),(x_2,g_2)\rangle)=0
\end{align*}
for $(x_1,g_1),(x_2,g_2)\in X\times G$.
Then $\Phi_f $ is a $2$-cocycle of the associated multiple conjugation biquandle $X\times G$.
\item[(2)]
Define $\Phi_f:C_2(X\times G)_{X\times G}\to A$ by
\begin{align*}
&\Phi_f(\langle(x,g)\rangle\langle(x_1,g_1)\rangle\langle(x_2,g_2)\rangle) \\
&=\lambda(g)f((x-x_1)(1-\varphi(g_1)^{-1}g_1),x_1-x_2,x_2(1-\varphi(g_2)g_2^{-1})), \text{ and} \\
&\Phi_f(\langle(x,g)\rangle\langle(x_1,g_1),(x_2,g_2)\rangle)=0
\end{align*}
for $(x,g),(x_1,g_1),(x_2,g_2)\in X\times G$.
Then $\Phi_f $ is a $2$-cocycle of the associated multiple conjugation biquandle $X\times G$ and itself as an $(X\times G)$-set with the map $*=\uline{*}$.
\item[(2)']
Define $\Phi_f:C_2(X\times G)_X\to A$ by
\begin{align*}
&\Phi_f(\langle x\rangle\langle(x_1,g_1)\rangle\langle(x_2,g_2)\rangle) \\
&=f((x-x_1)(1-\varphi(g_1)^{-1}g_1),x_1-x_2,x_2(1-\varphi(g_2)g_2^{-1})), \text{ and} \\
&\Phi_f(\langle x\rangle\langle(x_1,g_1),(x_2,g_2)\rangle)=0
\end{align*}
for $x\in X$ and $(x_1,g_1),(x_2,g_2)\in X\times G$.
Then $\Phi_f $ is a $2$-cocycle of the associated multiple conjugation biquandle $X\times G$ and  $X$ as an $(X\times G)$-set with the map $*$ such that $y*(x,g)=y\uline{*}^gx$.
\end{itemize}
\end{proposition}

\begin{proof}
This is seen by a direct computation.
The details will appear in a forthcoming paper.
\end{proof}

\section{Cocycle invariants} \label{sect:cocycleinvariant}

For an $X_Y$-coloring $C$ of an $S^1$-oriented handlebody-link diagram, we define the local chains $w(\xi;C)\in C_2(X)_Y$ at each crossing $\xi$ and each vertex $\xi$ of $D$ by
\[ \begin{array}{cc}
w\left(~
\begin{minipage}{45pt}
\begin{picture}(45,40)(-5,0)
 \put(40,40){\vector(-1,-1){40}}
 \put(0,40){\line(1,-1){18}}
 \put(22,18){\vector(1,-1){18}}
 \put(5,35){\makebox(0,0){\normalsize$\nearrow$}}
 \put(5,5){\makebox(0,0){\normalsize$\searrow$}}
 \put(-3,40){\makebox(0,0)[r]{\normalsize$a$}}
 \put(-3,0){\makebox(0,0)[r]{\normalsize$b$}}
 \put(-5,14){\framebox(12,12){\normalsize$y$}}
\end{picture}
\end{minipage}
~;C\right)=\langle y\rangle\langle a\rangle\langle b\rangle,
&w\left(~
\begin{minipage}{45pt}
\begin{picture}(45,40)(-5,0)
 \put(0,40){\vector(1,-1){40}}
 \put(40,40){\line(-1,-1){18}}
 \put(18,18){\vector(-1,-1){18}}
 \put(5,35){\makebox(0,0){\normalsize$\nearrow$}}
 \put(5,5){\makebox(0,0){\normalsize$\searrow$}}
 \put(-3,40){\makebox(0,0)[r]{\normalsize$b$}}
 \put(-3,0){\makebox(0,0)[r]{\normalsize$a$}}
 \put(-5,14){\framebox(12,12){\normalsize$y$}}
\end{picture}
\end{minipage}
~;C\right)=-\langle y\rangle\langle a\rangle\langle b\rangle, \vspace{1em}\\
w\left(~
\begin{minipage}{50pt}
\begin{picture}(50,40)(-5,0)
 \put(20,20){\vector(0,-1){20}}
 \put(0,40){\vector(1,-1){20}}
 \put(40,40){\vector(-1,-1){20}}
 \put(21,10){\makebox(0,0){\normalsize$\rightarrow$}}
 \put(5,35){\makebox(0,0){\normalsize$\nearrow$}}
 \put(35,35){\makebox(0,0){\normalsize$\searrow$}}
 \put(23,0){\makebox(0,0)[l]{\normalsize$b$}}
 \put(-3,40){\makebox(0,0)[r]{\normalsize$a$}}
 \put(-5,0){\framebox(12,12){\normalsize$y$}}
\end{picture}
\end{minipage}
~;C\right)=\langle y\rangle\langle a,b\rangle,
&w\left(~
\begin{minipage}{50pt}
\begin{picture}(50,40)(-5,0)
 \put(20,40){\vector(0,-1){20}}
 \put(20,20){\vector(-1,-1){20}}
 \put(20,20){\vector(1,-1){20}}
 \put(21,30){\makebox(0,0){\normalsize$\rightarrow$}}
 \put(5,5){\makebox(0,0){\normalsize$\searrow$}}
 \put(35,5){\makebox(0,0){\normalsize$\nearrow$}}
 \put(23,40){\makebox(0,0)[l]{\normalsize$b$}}
 \put(-3,0){\makebox(0,0)[r]{\normalsize$a$}}
 \put(-5,28){\framebox(12,12){\normalsize$y$}}
\end{picture}
\end{minipage}
~;C\right)=-\langle y\rangle\langle a,b\rangle.
\end{array} \]
We define a chain by $W(D;C)=\sum_{\xi\in C(D)\cup V(D)}w(\xi;C)\in C_2(X)_Y$.

\begin{lemma}
The chain $W(D;C)$ is a $2$-cycle of $C_*(X)_Y$.
Furthermore, for cohomologous $2$-cocycles $\theta,\theta'$ of $C^*(X;A)_Y$, we have $\theta(W(D;C))=\theta'(W(D;C))$.
\end{lemma}

\begin{proof}
Refer to the proof of Lemma~9 in~\cite{CarterIshiiSaitoTanaka16}, which is the quandle version of this lemma.
The proof can be directly applied to our case with our boundary map.
\end{proof}

For a $2$-cocycle $\theta\in C^2(X;A)_Y$, we define
\begin{align*}
\mathcal{H}(D)&=\{[W(D;C)]\in H_2(X)_Y\,|\,C\in\operatorname{Col}_{X_Y}(D)\}, \text{ and} \\
\Phi_\theta(D)&=\{\theta(W(D;C))\in A\,|\,C\in\operatorname{Col}_{X_Y}(D)\}
\end{align*}
as multisets.

\begin{theorem}
Let $H$ be an $S^1$-oriented handlebody-link, $D$ a diagram of $H$.
Then $\mathcal{H}(D)$ and $\Phi_\theta(D)$ are invariants of $H$.
\end{theorem}

\begin{proof}
It is sufficient to show that $[W(D;C)]\in H_2(X)_Y$ is invariant under each $X_Y$-colored Reidemeister move.
Here we show this property for the cases of an R4 move and an R5 move.
We leave the proof of the other cases to the reader.

\begin{figure}[h]
\begin{center}
\begin{minipage}{48pt}
\begin{picture}(48,72)(-12,-12)
 \qbezier(0,24)(0,30)(12,36) 
 \qbezier(12,36)(24,42)(24,48) 
 \qbezier(24,24)(24,30)(15,34.5) 
 \qbezier(9,37.5)(0,42)(0,48) 
 \qbezier(12,12)(12,6)(12,0) 
 \qbezier(12,12)(0,18)(0,24) 
 \qbezier(12,12)(24,18)(24,24) 
 \put(2,43){\makebox(0,0){\normalsize$\nearrow$}}
 \put(22,43){\makebox(0,0){\normalsize$\nwarrow$}}
 \put(13,5){\makebox(0,0){\normalsize$\rightarrow$}}
 \put(0,51){\makebox(0,0)[b]{\normalsize$a\uline{*}b$}}
 \put(24,51){\makebox(0,0)[b]{\normalsize$b\oline{*}a$}}
 \put(12,-3){\makebox(0,0)[t]{\normalsize$b^{-1}a\oline{*}b$}}
 \put(-3,24){\makebox(0,0)[r]{\normalsize$b$}}
 \put(27,24){\makebox(0,0)[l]{\normalsize$a$}}
 \put(6,18){\framebox(12,12){\normalsize$y$}}
\end{picture}
\end{minipage}
\hspace{0.5cm}$\overset{\text{R4}}{\leftrightarrow}$\hspace{0.5cm}
\begin{minipage}{64pt}
\begin{picture}(64,72)(-30,-12)
 \qbezier(0,48)(0,36)(0,24)
 \qbezier(24,48)(24,36)(24,24)
 \qbezier(12,12)(12,6)(12,0) 
 \qbezier(12,12)(0,18)(0,24) 
 \qbezier(12,12)(24,18)(24,24) 
 \put(1,43){\makebox(0,0){\normalsize$\rightarrow$}}
 \put(23,43){\makebox(0,0){\normalsize$\leftarrow$}}
 \put(13,5){\makebox(0,0){\normalsize$\rightarrow$}}
 \put(0,51){\makebox(0,0)[b]{\normalsize$a\uline{*}b$}}
 \put(24,51){\makebox(0,0)[b]{\normalsize$b\oline{*}a$}}
 \put(12,-3){\makebox(0,0)[t]{\normalsize$b^{-1}a\oline{*}b$}}
 \put(-30,18){\framebox(24,12){\normalsize$y*b$}}
\end{picture}
\end{minipage}
\end{center}
\caption{Invariance under an R4 move.}
\label{fig:R4-invariance}
\end{figure}
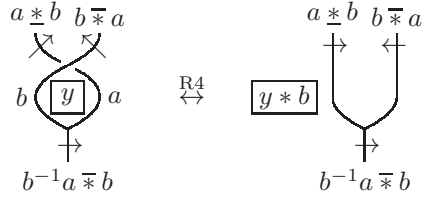

Let $(D,C)$ and $(D',C')$ be $X_Y$-colored diagrams, of $H$, which differ by the R4 move shown in Figure~\ref{fig:R4-invariance}.
Then we have
\begin{align*}
&W(D;C)
=-\langle y\rangle\langle a\rangle\langle b\rangle
-\langle y\rangle\langle b,a\rangle+\mathcal{C}
=-\langle y\rangle\langle a,ab\rangle
+\langle y\rangle\langle b,ab\rangle
-\langle y\rangle\langle b,a\rangle+\mathcal{C}, \\
&W(D';C')=-\langle y*b\rangle\langle b^{-1}a\oline{*}b,a\uline{*}b\rangle+\mathcal{C}
=-\langle y*b\rangle\langle b^{-1}a\oline{*}b,b^{-1}(ab)\oline{*}b\rangle+\mathcal{C},
\end{align*}
for some chain $\mathcal{C}$ in $C_2(X)_Y$.
On the other hand,
\[ \partial_3(\langle y\rangle\langle  b,a,ab\rangle)
=\langle y*b\rangle\langle b^{-1}a\oline{*}b,b^{-1}(ab)\oline{*}b\rangle
-\langle y\rangle\langle a,ab\rangle+\langle y\rangle\langle b,ab\rangle
-\langle y\rangle\langle b,a\rangle, \]
which implies that $W(D;C)-W(D';C')=\partial_3(\langle y\rangle\langle b,a,ab\rangle)$.
Therefore we have $[W(D;C)]=[W(D';C')]\in H_2(X)_Y$.

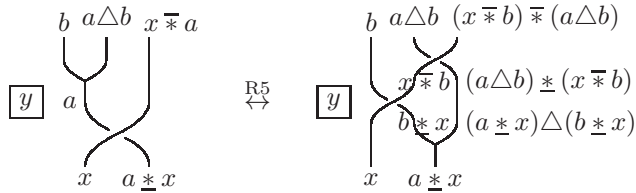
\begin{figure}[h]
\begin{center}
\begin{minipage}{70pt}
\begin{picture}(70,72)(-20,-12)
 \qbezier(0,40)(0,44)(0,48)
 \qbezier(16,40)(16,44)(16,48)
 \qbezier(8,32)(8,28)(8,24) 
 \qbezier(8,32)(0,36)(0,40) 
 \qbezier(8,32)(16,36)(16,40) 
 \qbezier(32,24)(32,36)(32,48)
 \qbezier(8,0)(8,6)(20,12) 
 \qbezier(20,12)(32,18)(32,24) 
 \qbezier(32,0)(32,6)(23,10.5) 
 \qbezier(17,13.5)(8,18)(8,24) 
 \put(0,51){\makebox(0,0)[b]{\normalsize$b$}}
 \put(16,51){\makebox(0,0)[b]{\normalsize$a\triangle b$}}
 \put(30,51){\makebox(0,0)[bl]{\normalsize$x\oline{*}a$}}
 \put(8,-3){\makebox(0,0)[t]{\normalsize$x$}}
 \put(32,-3){\makebox(0,0)[t]{\normalsize$a\uline{*}x$}}
 \put(5,24){\makebox(0,0)[r]{\normalsize$a$}}
 \put(-20,18){\framebox(12,12){\normalsize$y$}}
\end{picture}
\end{minipage}
\hspace{0.5cm}$\overset{\text{R5}}{\leftrightarrow}$\hspace{0.5cm}
\begin{minipage}{120pt}
\begin{picture}(120,72)(-20,-12)
 \qbezier(16,32)(16,36)(24,40) 
 \qbezier(24,40)(32,44)(32,48) 
 \qbezier(32,32)(32,36)(26,39) 
 \qbezier(22,41)(16,44)(16,48) 
 \qbezier(0,32)(0,40)(0,48)
 \qbezier(0,16)(0,20)(8,24) 
 \qbezier(8,24)(16,28)(16,32) 
 \qbezier(16,16)(16,20)(10,23) 
 \qbezier(6,25)(0,28)(0,32) 
 \qbezier(32,16)(32,24)(32,32)
 \qbezier(24,8)(24,4)(24,0) 
 \qbezier(24,8)(16,12)(16,16) 
 \qbezier(24,8)(32,12)(32,16) 
 \qbezier(0,0)(0,8)(0,16)
 \put(0,51){\makebox(0,0)[b]{\normalsize$b$}}
 \put(16,51){\makebox(0,0)[b]{\normalsize$a\triangle b$}}
 \put(30,51){\makebox(0,0)[bl]{\normalsize$(x\oline{*}b)\oline{*}(a\triangle b)$}}
 \put(0,-3){\makebox(0,0)[t]{\normalsize$x$}}
 \put(24,-3){\makebox(0,0)[t]{\normalsize$a\uline{*}x$}}
 \put(10,32){\makebox(0,0)[l]{\normalsize$x\oline{*}b$}}
 \put(10,16){\makebox(0,0)[l]{\normalsize$b\uline{*}x$}}
 \put(35,32){\makebox(0,0)[l]{\normalsize$(a\triangle b)\uline{*}(x\oline{*}b)$}}
 \put(35,16){\makebox(0,0)[l]{\normalsize$(a\uline{*}x)\triangle(b\uline{*}x)$}}
 \put(-20,18){\framebox(12,12){\normalsize$y$}}
\end{picture}
\end{minipage}
\end{center}
\caption{Invariance under an R5 move.}
\label{fig:R5-invariance}
\end{figure}

Let $(D,C)$ and $(D',C')$ be $X_Y$-colored diagrams, of $H$, which differ by the R5 move shown in Figure~\ref{fig:R5-invariance}, where all arcs are directed from top to bottom, and $a\triangle b=b^{-1}a\oline{*}b$.
Then we have
\begin{align*}
&W(D;C)=\langle y\rangle\langle b,a\rangle
+\langle y\rangle\langle a\rangle\langle x\rangle+\mathcal{C}, \\
&W(D';C')=\langle y\rangle\langle b\rangle\langle x\rangle
+\langle y*b\rangle\langle b^{-1}a\oline{*} b\rangle\langle x\oline{*} b\rangle
+\langle y*x\rangle\langle b\uline{*}x\rangle\langle a\uline{*}x\rangle+\mathcal{C},
\end{align*}
for some chain $\mathcal{C}$ in $C_2(X)_Y$.
On the other hand,
\begin{align*}
\partial_3(\langle y\rangle\langle b,a\rangle\langle x\rangle)
&=\langle y*b\rangle\langle b^{-1}a\oline{*}b\rangle\langle x\oline{*}b\rangle
-\langle y\rangle\langle a\rangle\langle x\rangle \\
&\hspace{1em}
+\langle y\rangle\langle b\rangle\langle x\rangle
+\langle y*x\rangle\langle b\uline{*}x\rangle\langle a\uline{*}x\rangle
-\langle y\rangle\langle b,a\rangle,
\end{align*}
which implies that
$W(D;C)-W(D';C')=-\partial_3(\langle y\rangle\langle b,a\rangle\langle x\rangle)$.
Therefore we have $[W(D;C)]=[W(D';C')]\in H_2(X)_Y$.
\end{proof}

For an $S^1$-oriented handlebody-link $H$, we denote by $H^*$ the mirror image of $H$, and denote by $-H$ the $S^1$-oriented handlebody-link obtained from $H$ by reversing its $S^1$-orientation.
Then we have the following proposition.

\begin{proposition}
For an $S^1$-oriented handlebody-link $H$, $\mathcal{H}(-H^*)=-\mathcal{H}(H)$ and $\Phi_\theta(-H^*)=-\Phi_\theta(H)$.
\end{proposition}

\begin{proof}
Let $D$ be a diagram of $H$.
We suppose that $D$ is depicted in an $xy$-plane $\mathbb{R}^2$.
Let $D^*$ be the image of $D$ by the involution $\varphi:\mathbb{R}^2\to\mathbb{R}^2$ defined by $\varphi(x,y)=(-x,y)$.
Then $D^*$ is a diagram of $H^*$.
We obtain the diagram $-D^*$ of $-H^*$ by reversing the $S^1$-orientation of $H^*$.
Then the composition $C^*=C\circ\varphi$ is an $X_Y$-coloring of $-D^*$.
It is sufficient to show that $W(D;C)=-W(-D^*;C^*)$, which follows from
\[ \begin{array}{cc}
w\left(~
\begin{minipage}{45pt}
\begin{picture}(45,40)(-5,0)
 \put(40,40){\vector(-1,-1){40}}
 \put(0,40){\line(1,-1){18}}
 \put(22,18){\vector(1,-1){18}}
 \put(5,35){\makebox(0,0){\normalsize$\nearrow$}}
 \put(5,5){\makebox(0,0){\normalsize$\searrow$}}
 \put(-3,40){\makebox(0,0)[r]{\normalsize$a$}}
 \put(-3,0){\makebox(0,0)[r]{\normalsize$b$}}
 \put(-5,14){\framebox(12,12){\normalsize$y$}}
\end{picture}
\end{minipage}
~;C\right)=\langle y\rangle\langle a\rangle\langle b\rangle,
&w\left(~
\begin{minipage}{45pt}
\begin{picture}(45,40)
 \put(40,0){\vector(-1,1){40}}
 \put(22,22){\vector(1,1){18}}
 \put(18,18){\line(-1,-1){18}}
 \put(35,35){\makebox(0,0){\normalsize$\nwarrow$}}
 \put(35,5){\makebox(0,0){\normalsize$\swarrow$}}
 \put(43,40){\makebox(0,0)[l]{\normalsize$a$}}
 \put(43,0){\makebox(0,0)[l]{\normalsize$b$}}
 \put(33,14){\framebox(12,12){\normalsize$y$}}
\end{picture}
\end{minipage}
~;C^*\right)=-\langle y\rangle\langle a\rangle\langle b\rangle, \vspace{1em}\\
w\left(~
\begin{minipage}{50pt}
\begin{picture}(50,40)(-5,0)
 \put(20,20){\vector(0,-1){20}}
 \put(0,40){\vector(1,-1){20}}
 \put(40,40){\vector(-1,-1){20}}
 \put(21,10){\makebox(0,0){\normalsize$\rightarrow$}}
 \put(5,35){\makebox(0,0){\normalsize$\nearrow$}}
 \put(35,35){\makebox(0,0){\normalsize$\searrow$}}
 \put(23,0){\makebox(0,0)[l]{\normalsize$b$}}
 \put(-3,40){\makebox(0,0)[r]{\normalsize$a$}}
 \put(-5,0){\framebox(12,12){\normalsize$y$}}
\end{picture}
\end{minipage}
~;C\right)=\langle y\rangle\langle a,b\rangle,
&w\left(~
\begin{minipage}{50pt}
\begin{picture}(50,40)(-5,0)
 \put(20,20){\vector(1,1){20}}
 \put(20,20){\vector(-1,1){20}}
 \put(20,0){\vector(0,1){20}}
 \put(19,10){\makebox(0,0){\normalsize$\leftarrow$}}
 \put(5,35){\makebox(0,0){\normalsize$\swarrow$}}
 \put(35,35){\makebox(0,0){\normalsize$\nwarrow$}}
 \put(23,0){\makebox(0,0)[l]{\normalsize$b$}}
 \put(43,40){\makebox(0,0)[l]{\normalsize$a$}}
 \put(33,0){\normalsize\framebox(12,12){$y$}}
\end{picture}
\end{minipage}
~;C^*\right)=-\langle y\rangle\langle a,b\rangle.
\end{array} \]
\end{proof}

\begin{example}
Let $G=SL(2,\mathbb{Z}_6)$ and $X=\mathbb{Z}_6^2$, where we regard $X$ as the right $\mathbb{Z}_6[G]$-module with matrix product
\[ (x,y)\cdot\begin{pmatrix} a & b \\ c & d \end{pmatrix}=(ax+cy,bx+dy). \]
We define a homomorphism $\varphi:G\to Z(G)$ by
\[ \varphi(\begin{pmatrix} a & b \\ c & d \end{pmatrix})
=\begin{pmatrix} (-1)^{(a+b+c+1)(b+c+d+1)} & 0 \\ 0 & (-1)^{(a+b+c+1)(b+c+d+1)} \end{pmatrix}. \]
Then as shown in (2) of Example~\ref{exam:2.5} and Proposition~\ref{prop:G-family2MCB}, $X\times G=\bigsqcup_{x\in X}\{x\}\times G$ can be regarded as a multiple conjugation biquandle with the operations $(x,g)\uline{*}(y,h)=(xg+y(\varphi(g)-g),h^{-1}gh)$ and $(x,g)\oline{*}(y,h)=(x\varphi(g),g)$.
Set a homomorphism $\lambda:G\to\mathbb{Z}_6$ by
\[ \lambda(\begin{pmatrix} a & b \\ c & d \end{pmatrix})=2(a+d)(b-c)(1-bc). \]
We note that the determinant $\det:G\to\mathbb{Z}_6$ is a $G$-invariant $\mathbb{Z}_6$-multilinear map.
Thus as shown in (1) of Proposition~\ref{prop:cocycle}, we have a $2$-cocycle $\Phi_{\det}:C_2(X\times G)\to\mathbb{Z}_6$, of $X\times G=\bigsqcup_{x\in X}\{x\}\times G$, defined by
\begin{align*}
&\Phi_{\det}(\langle(x_1,g_1)\rangle\langle(x_2,g_2)\rangle)
=\lambda(g_1)\det(x_1-x_2,x_2(1-\varphi(g_2)g_2^{-1})), \text{ and} \\
&\Phi_{\det}(\langle(x_1,g_1),(x_2,g_2)\rangle)=0.
\end{align*}
For the handlebody-knot $5_2$ depicted in Figure~\ref{fig:5_2}, we computed the invariant $\Phi_{\det}(5_2)$ and as the result, we have
\[ \Phi_{\det}(5_2)=\{\text{$0$ ($2987712$ times), $2$ ($157248$ times), $4$ ($157248$ times)}\}. \]
\end{example}

\begin{figure}
\mbox{}\hfill
\begin{picture}(60,65)
 \qbezier(0,30)(0,32)(0,35)
 \qbezier(0,35)(0,50)(15,50)
 \qbezier(25,50)(35,50)(45,50)
 \qbezier(45,50)(60,50)(60,35)
 \qbezier(60,35)(60,32)(60,30)
 \qbezier(0,30)(20,30)(35,30)\qbezier(45,30)(50,30)(60,30)
 \qbezier(60,30)(60,22)(60,15)
 \qbezier(60,15)(60,0)(45,0)
 \qbezier(45,0)(40,0)(35,5)
 \qbezier(35,5)(30,10)(25,15)
 \qbezier(25,15)(20,20)(20,25)
 \qbezier(20,35)(20,45)(20,55)
 \qbezier(20,55)(20,65)(30,65)
 \qbezier(30,65)(40,65)(40,55)
 \qbezier(40,45)(40,40)(40,25)
 \qbezier(40,25)(40,20)(35,15)
 \qbezier(35,15)(34,14)(33,13)
 \qbezier(27,7)(26,6)(25,5)
 \qbezier(25,5)(20,0)(15,0)
 \qbezier(15,0)(0,0)(0,15)
 \qbezier(0,15)(0,22)(0,30)
\end{picture}
\hfill\mbox{}
\caption{The handlebody-knot $5_2$.}
\label{fig:5_2}
\end{figure}
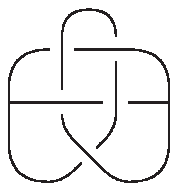

\end{document}